\documentclass[reqno,english,11pt]{amsart}

\usepackage{epsfig}
\usepackage{amsmath,amssymb,amsthm,amsxtra}
\usepackage{enumerate}
\usepackage{verbatim}
\usepackage{amsfonts}
\usepackage{graphicx}
\usepackage{fancyhdr}   
\usepackage{mathrsfs} 
\usepackage{geometry}
\usepackage{multicol}
\usepackage{lipsum}
\usepackage{mwe}

\usepackage{marginnote}

\usepackage{latexsym,verbatim,mathrsfs,color,array}
\usepackage{hyperref}
\usepackage{graphicx}
\usepackage{bbm}
\usepackage{pdfsync}
\usepackage{epstopdf}
\usepackage{cite}
\usepackage{marginnote}

\usepackage{multirow}
\usepackage[font=bf,aboveskip=15pt]{caption}
\usepackage[toc,page]{appendix}

\usepackage{bookmark}
\DeclareFontShape{OT1}{cmr}{bx}{sc}{<-> cmbcsc10}{}
\usepackage{float}
\usepackage{ulem}
\usepackage{syntonly}
\usepackage{mathtools}
\usepackage{bm}

\usepackage{epstopdf}
\usepackage{upgreek}
\usepackage{physics}						%Physics symbols
\usepackage{subfiles}						%handle multi-file projects more comfortably 								
\usepackage{cancel}							%draw diagonal lines (?cancelling? a term)
\usepackage{xcolor}							%Color Package
\usepackage{dsfont}						%needed for \mathds
\usepackage{enumitem}					%to label each 
\usepackage{outlines}

\usepackage{tikz}	%comment out at the end!

\usepackage{pdfpages}
\usepackage{blindtext}

\usepackage{float}
\usepackage[margin=0pt,font=small,labelfont=bf,textfont=it]{caption}

\theoremstyle{plain}

\newtheorem{lemma}{Lemma}[section]
\newtheorem{proposition}{Proposition}[section]
\newtheorem{theorem}{Theorem}
\newtheorem{corollary}{Corollary}[section]
\newtheorem{remark}{Remark}[section]
\newcommand{\bremark}{\begin{remark} \em}
	\newcommand{\eremark}{\end{remark} }

\newcommand{\ext}{\text{ext}}
\newcommand{\intb}{\int_{B_{2R}}}
\newcommand{\MMM}{\mathcal{M}}

\newcommand{\ass}{\quad\mbox{as}\quad}
\newcommand{\ddiv}{\,\text{div}}
\newcommand{\inn}{{\quad\hbox{in } }}
\newcommand{\iinn}{\text{in}}
\newcommand{\onn}{{\quad\hbox{on } }}

\newcommand{\pp}{ {\partial} }

\newcommand{\svc}{{\sharp,c,\ve}}
\newcommand{\svvc}{{\sharp,c,\frac{1}{2}+\ve}}

\newcommand{\cyt}{{\Om\times [t_0,\infty)}}

\newcommand\restr[2]{{% we make the whole thing an ordinary symbol
		\left.\kern-\nulldelimiterspace % automatically resize the bar with \right
		#1 % the function
		\vphantom{\big|} % pretend it's a little taller at normal size
		\right|_{#2} % this is the delimiter
}}

\newcommand{\reff}[1]{(\ref{#1})} 		%Reference already with the parenthesis

\newcommand{\Sinn}{S_{\text{in}}}
\newcommand{\Sout}{S_{\text{out}}}

\newcommand{\blue}[1]{\color{blue}{#1}}

\newcommand{\RR}{{{\mathbb{R}}}}

\newcommand{\CCC}{ {\mathcal C}}
\newcommand{\AAA}{\mathcal{A}}
\newcommand{\BBB}{\mathcal{B}}
\newcommand{\NNN}{\mathcal{N}}
\newcommand{\SSS}{\mathcal{S}}

\newcommand{\LLL}{\mathcal{L}}
\newcommand{\FFF}{\mathcal{F}}
\newcommand{\TTT}{\mathcal{T}}
\newcommand{\JJJ}{\mathcal{J}}
\newcommand{\bn}{ {\mu_{\text{BN}}}}

\newcommand{\ppsi}{\uppsi}
\newcommand{\ay}{\abs{y}}

\newcommand{\g}{\gamma}
\newcommand{\dl}{\dot \Lambda}
\newcommand{\mt}{\mu(t)}
\newcommand{\xt}{\xi(t)}
\newcommand{\Uy}{U(y)}

\newcommand{\at}{\alpha_3}

\newcommand{\Om}{\Omega}	
\newcommand{\ddy}{\,dy}					%Spaces before the differential
\newcommand{\ddx}{\,dx}
\newcommand{\ddu}{\,du}
\newcommand{\dtau}{\,d\tau}	
\newcommand{\dds}{\,ds}

\newcommand{\ddz}{\,dz}

\newcommand{\ddr}{\,dr}

\newcommand{\ve}{\varepsilon}
\newcommand{\eps}{\epsilon}

	\hypersetup{
		colorlinks,
	 	linkcolor={red!70!black},
		citecolor={blue!70!black},
		urlcolor={blue!70!black}
	}
\usepackage{xcolor}	

\newcommand{\MONTH}{
	\ifcase\the\month
	\or January
	\or February
	\or March
	\or April
	\or May
	\or June
	\or July
	\or August
	\or September
	\or October
	\or November
	\or December
	\fi
}
\usepackage{lmodern} 		%to use fontenc  with bitmap font
\usepackage{amsfonts,amsmath,amssymb}
\usepackage{xcolor}	
\usepackage{amssymb,graphicx}
\usepackage{stmaryrd} %for \nnearrow
\usepackage[toc,page]{appendix}
\usepackage{upgreek}
\usepackage{marginnote}
\usepackage{physics}	%Physics symbols
\usepackage{dsfont}%needed for \mathds
\usepackage{enumitem}
\usepackage{microtype}
\long\def\blue#1{{\color{black}#1}}

\numberwithin{equation}{section}
\usepackage{subfiles}

\DeclareMathOperator{\diam}{diam}

\DeclareMathOperator\supp{supp}

\begin{document}
	%\iffalse
	%\date{\today}
	\title[Infinite time blow-up for critical heat equation]{Infinite time blow-up for the three dimensional energy critical heat equation in bounded domains}
	
	\author[G. Ageno]{Giacomo Ageno}
	\address{\noindent
		Department of Pure Mathematics and Mathematical Statistics, Cambridge, United Kingdom}
	\email{ga482@cam.ac.uk}
	
	\author[M. del Pino]{Manuel del Pino}
	\address{\noindent
		Department of Mathematical Sciences,
		University of Bath, Bath BA2 7AY, United Kingdom}
	\email{mdp59@bath.ac.uk}

	\begin{abstract}
		We consider the Dirichlet problem for the energy-critical heat equation
		\begin{align*}
			\begin{cases}
				u_t=\Delta u+u^5 & \text{in} \quad \Omega \times \mathbb{R}^+,\\
				u=0 &\text{on} \quad \partial \Omega \times \mathbb{R}^+,\\ 
				u(x,0)=u_0(x) & \text{in} \quad \Omega,
			\end{cases}
		\end{align*}
		where $\Omega$ is a bounded smooth domain in $\mathbb{R}^3$. Let $H_\gamma(x,y)$ be the regular part of the Green function of $-\Delta-\gamma$ in $\Omega$, where $\gamma \in (0,\lambda_1)$ and $\lambda_1$ is the first Dirichlet eigenvalue of $-\Delta$. Then, given a point $q\in \Omega$ such that $3\gamma(q)<\lambda_1$, where
		\begin{align*}
			\gamma(q){\coloneqq}\sup\{ \gamma>0:  H_\g(q,q)>0 		\},
		\end{align*}
		we prove the existence of a non-radial global positive and smooth solution $u(x,t)$ which blows up in infinite time with spike in $q$. The solution has the asymptotic profile
		\begin{align}
			u(x,t)\sim 3^{\frac{1}{4}} \qty(\frac{\mu(t)}{\mu(t)^2+\abs{x-\xi(t)}^2})^{\frac{1}{2}} \quad \text{as}\quad t \to \infty,
		\end{align}
		where 
		\begin{align*}
			-\ln\qty(\mu(t))= 2\gamma(q) t(1+o(1)),\quad \xi(t)=q+O\qty(\mu(t)) \quad \text{as}\quad t \to \infty.
		\end{align*}
	\end{abstract}
	
	\maketitle
	%{
		%	\hypersetup{linkcolor=black}
		%	%\tableofcontents
		%}

	\section{Introduction and statement of the main result}
	%Intro problem
	We investigate the asymptotic structure of global in time solutions $u(x,t)$ of the energy-critical semilinear heat equation
	\begin{equation}\label{CHE}
		\left\{
		\begin{aligned}
			&u_t=\Delta u +u^5&& \text{in} \quad \Omega \times \mathbb{R}^+,\\
			&u=0 \quad && \text{on} \quad\partial \Omega \times \mathbb{R}^+,\\
			&u(x,0)=u_0(x) && \text{in} \quad \Omega,
		\end{aligned}
		\right.
	\end{equation}
	where $\Omega\subset \mathbb{R}^3$ is a smooth bounded domain and $u_0$ is a smooth initial datum. The energy associated to the solution $u(x,t)$ is 
	\begin{align*}
		E(u){\coloneqq}\frac{1}{2}\int_{\Omega} \abs{\nabla u}^2 \, dx -\frac{1}{6} \int \abs{u}^{6} \, dx.
	\end{align*}
	Since classical solutions of (\ref{CHE}) satisfy
	\begin{align*}
		\frac{d}{dt}E(u(\cdot,t))=-\int_{\Omega} \abs{u_t}^2 \,dx\leq 0,
	\end{align*}
	the energy is a Lyapunov functional for (\ref{CHE}). The stationary equation on the whole space is the Yamabe problem
	\begin{align*}
		\Delta U+U^5=0 \inn  \RR^3.
	\end{align*}
	All positive solutions to this equation are given by the Aubin-Talenti bubbles (see \cite{caff})
	\begin{align}\label{Talentibubble}
		U_{\mu,\xi}(x)=\mu^{-\frac{1}{2}}U\qty(\frac{x-\xi}{\mu}),
	\end{align}
	where $\mu >0,\xi \in \RR^3$ and 
	\begin{align*}
		U(x)=\alpha_3 \frac{1}{\qty(1+\abs{x}^2)^{\frac{1}{2}}},\quad\text{where}\quad \alpha_3{\coloneqq}3^{\frac{1}{4}}.
	\end{align*}
	Consider the Sobolev embedding $H_0^1(\Om)\hookrightarrow L^{p+1}(\Om)$, which is compact for $p\in (1,p_S)$, where $p_S=\frac{n+2}{n-2}$, and the associated constant
	\begin{align*}
		S_p(\Om){\coloneqq}\inf_{0\neq u \in H_0^1(\Om)}\frac{\norm{u}_{H_0^1(\Om)}^2}{\norm{u}_{L^{p+1}(\Om)}^2}.
	\end{align*}
	The Aubin-Talenti bubbles achieve the constant $S_{p_S}(\RR^n)$. Thus, the energy $E(U_{\mu,\xi})=S_{p_S}(\RR^n)$ is invariant with respect to $\mu,\xi$. When $\mu\to 0$ the Aubin-Talenti bubble becomes singular.  
	This is the reason for the loss of compactness in the Sobolev embedding for $p=p_S$. Indeed, Struwe proved in \cite{struwe} that every Palais-Smale sequence associated to the energy functional $E$ looks like
	\begin{equation}\label{bubbledecomp}
		u_n(x)=u_\infty(x)+\sum_{i=1}^{k}U_{\mu_n^i,\xi_n^i}(x)+o(1)	\quad \text{when}\quad n\to \infty,
	\end{equation}
	up to subsequences, for some $k\in \mathbb{N}$, where $u_\infty \in H_0^1(\Omega)$ is a critical point of $E$ and $\mu_n^i\to 0$, $\xi_n^i\in \Omega$. Thus, we say that the compactness is lost by 'bubbling'. When the domain is star-shaped, the Pohozaev identity constrains $u_\infty$ to vanish.

	For classical finite-energy solutions $u(x,t)$ the problem (\ref{CHE}) is well-posed in short time intervals. We refer to the monograph \cite{qs} by Quittner and Souplet for an extended review on this problem and more general semilinear parabolic equations.
	The aim of this paper is exhibiting classical positive finite-energy solutions $u(x,t)$ of (\ref{CHE}) which are globally defined in time and satisfy 
	\begin{align}\label{infblowup}
		\lim\limits_{t\to \infty}\norm{u(\cdot,t)}_{L^\infty(\Omega)}=\infty.
	\end{align}
	Given any smooth function $\varphi(x)\geq 0,\varphi \neq 0$, consider $\alpha>0$ and $u_\alpha(x,0){\coloneqq}\alpha \varphi(x)$ as initial datum. On one hand, if $\alpha$ is sufficiently small, then $u_\alpha(x,t)$ tends uniformly to zero as $t\to \infty$. On the other hand, using the eigenfunction method of Kaplan \cite{kaplan}, for $\alpha$ sufficiently large $u_\alpha(x,t)$ blows-up in finite time. 
	Thus, the threshold number 
	\begin{align*}
		\alpha^*{\coloneqq}\sup\left\{\alpha>0\,:\, \lim\limits_{t\to \infty}\norm{u_\alpha(\cdot,t)}_\infty=0	\right\},
	\end{align*}
	is positive. In 1984, the first rigorous proof of the existence in $L^1$-weak sense of $u_{\alpha^*}(x,t)$ was found by Ni, Sacks and Tavantzis \cite{nst}. 
	Du \cite{du} and Suzuki \cite{suzuki} proved, that, for any unbounded sequence of times $t_n$, $u_{\alpha^*}(x,t_n)$  can be decomposed as in (\ref{bubbledecomp}).
	Thus, when constructing unbounded global solutions for the critical case, it is natural to look for an asymptotic profile as (\ref{Talentibubble}). Galaktionov and V\'azquez \cite{gv} proved that, in the radial case $\Omega=B_1(0)$ with $\varphi$ radial non-increasing, $u_{\alpha^*}(x,t)$ is smooth, global and $u=u_{\alpha^*}$ satisfies (\ref{infblowup}).
	Thus, we naturally wonder what is the asymptotic behavior of global unbounded solutions. Most of the results about the dynamics of threshold solutions in literature concern the radial case. This particular setting allows the construction of specific solutions by means of matched expansions. In \cite{gk2} Galaktionov and King proved that the threshold behavior of $u_{\alpha^*}$ in the radial case is
	\begin{equation}\label{muGK}
		\ln \norm{u_{\alpha^*}(\cdot,t)}_\infty=\left\{
		\begin{aligned}
			\frac{\pi^2}{4}t(1+o(1)) && \text{if} \quad n=3,
			\\2\sqrt{t}(1+o(1)) && \text{if} \quad n=4,
		\end{aligned}
		\right.
	\end{equation}
	and
	\begin{align*}
		\norm{u_{\alpha^*}(\cdot,t)}_\infty= \qty(\gamma_n t)^{\frac{n-2}{2(n-4)}},\quad \text{if}\quad n\geq 5,
	\end{align*}
	for some explicit constants $\gamma_n$. Our main theorem is a non-radial extension in dimension $3$. The existence of positive non-radial unbounded solutions for the Dirichlet problem in dimension $n=4$ remains an open problem, which we will consider in a future work.
	The case of higher dimension $n\geq 5$ has been already extended to the non-radial case by Cort\'azar, del Pino and Musso in \cite{cdm}. They found positive multi-spike global solutions which blow-up by bubbling in infinite time. Here, the term multi-spike refers to the fact that the constructed solution is unbounded in a finite number of points in $\Om$. Sign-changing solutions which blow-up in infinite time have been discovered by del Pino, Musso, Wei and Zheng in \cite{signchangingDir5} for $n\geq 5$, proving stability in case $n=5,6$.
	
	Our solutions involve the Green function $G_\gamma$ associated to the elliptic operator $$L_\gamma=-\Delta-\gamma\quad \text{in}\quad \Omega,$$ where $\gamma \in [0,\lambda_1)$ and $\lambda_1$ is the principal Dirichlet eigenvalue. Namely, for all $y\in \Omega$, $G_\gamma$ satisfies
	\begin{align*}%\label{Ggamma}
		&-\Delta_x G_\gamma(x,y)-\gamma G_\gamma(x,y) =c_3 \delta(x-y) \quad \text{in} \quad \Omega,\\\nonumber
		&G_\gamma(x,y)=0 \quad \text{on} \quad \partial \Omega,
	\end{align*}
	where $\delta(x)$ is the Dirac delta, $c_3{\coloneqq} \alpha_3\omega_3$ and the constant $\omega_3=4\pi$ indicates the area of the unit sphere. The Green function can be decomposed as
	\begin{align*}
		G_\g(x,y)=\Gamma(x-y)-H_\gamma(x,y),
	\end{align*}
	where $\Gamma(x)=\alpha_3\abs{x}^{-1}$ and the regular part $H_\gamma(x,y)$ is defined as the solution, for all $y \in \Omega$, to		
	\begin{align*}%\label{Hgamma}
		&\Delta_x H_\g(x,y)+\gamma H_\gamma(x,y)=\gamma\frac{\alpha_3}{\abs{x-y}} \quad \text{in} \quad \Omega,\\\nonumber
		&H_\gamma(x,y)=\Gamma(x -y)\quad \text{in} \quad \partial \Omega.
	\end{align*}
	The diagonal $R_\g(x){\coloneqq}H_\g(x,x)$ is called Robin function associated to $-\Delta-\g$ in $\Omega$. It turns out (see Lemma \ref{Lemma:gammaexists}) that for any fixed $q\in \Omega$ there exists a unique number $\g(q)\in (0,\lambda_1)$ defined by
	\begin{align*}
		\gamma(q){\coloneqq}\sup \{	\gamma>0: R_\g(q)>0	\}.
	\end{align*}
	Our main theorem shows that, for any $q\in \Om$ such that $3\g(q)<\lambda_1$, there exists a global solution to the problem (\ref{CHE}) which blows-up in infinite time with spike in $x=q$.
	\begin{theorem}\label{mainteo}
		Let $\Omega \subset \RR^3$ be a bounded smooth domain. Let $q$ be a point in $\Omega$ such that
		\begin{align}\label{Assumption1}
			\g(q)<\frac{\lambda_1}{3}.
		\end{align}
		Then, there exist an initial condition $u_0(x)\in C^1(\bar \Omega)$, smooth functions $\xi(t),\mu(t)$ and $\theta(x,t)$ such that the solution $u(x,t)$ to the problem (\ref{CHE}) is a positive unbounded global solution with the asymptotic profile
		\begin{align}\label{asyprof}
			u(x,t)=\mu(t)^{-\frac{1}{2}} U\qty(\frac{x-\xi(t)}{\mu(t)})- \mu(t)^{\frac{1}{2}}\qty(H_\g(x,\xi(t))+\theta(x,t))\ass t \to \infty,
		\end{align}  
		where $\theta$ is a bounded function, and decays uniformly away from the point $q$. Moreover, the parameters $\mt,\xt$ are smooth functions of time and satisfy
		\begin{align}\label{behmuxi}
			\ln(\frac{1}{\mu(t)})= 2\gamma(q) t(1+o(1)),\quad \xi(t)-q=O\qty(\mt) \ass t\to \infty.
		\end{align}
	\end{theorem}
	Furthermore, thanks to the \underline{inner-outer gluing scheme}, which is based only on elliptic and parabolic estimates, as in \cite{cdm} and \cite{dmw1} we get a codimension-1 stability of the solution stated by Theorem \ref{mainteo}. 
	In fact, since condition (\ref{Assumption1}) is stable under small perturbation of $q\in \Omega$, the stability result follows exactly as in \cite[Proof of Corollary 1.1]{cdm} (see Remark \ref{remark:1cod} in \S\ref{sec:finalargument}).
	\begin{corollary}\label{CorStability}
		Let $u$ be the solution stated in Theorem \ref{mainteo} which blows up at $q$. Then, there exists a codimension-$1$ manifold $\MMM$ in $C^{1}(\bar \Om)$ with $u_0 \in \MMM$ and such that if $\tilde u_0\in \MMM$ and it is sufficiently close to $u_0$, then the solution $\tilde u$ to (\ref{CHE}) with initial datum $\tilde u_0$ is global and blows-up in infinite time with spike in $\tilde q$ near $q$ and profile (\ref{asyprof}) with $\ln \norm{\tilde u(\cdot,t)}_\infty= \gamma(\tilde q) t(1+o(1))$ as $t\to \infty$.
	\end{corollary}
	Condition (\ref{Assumption1}) implies that the point $q$ cannot be very close to boundary, since $\gamma(q)\to \lambda_1^{-}$ as $q\to \pp \Om$ (see Lemma \ref{asymptgamma*} in Appendix \ref{app:propH}). 
	Along the proof we need to consider Dirichlet problems of the type
	\begin{align*}
		&u_t =\Delta u+\g u+e^{-2\g t}f(x) \inn \Omega \times \RR^+,\\
		&u(x,t)=0 \onn \pp \Om \times \RR^+,\\
		&u(x,0)=0 \inn \Om,
	\end{align*}
	for some $f(x)\in L^p$ with $p>2$. In order to successfully apply fixed point arguments, we need 
	\begin{align*}
		\norm{u(\cdot,t)}_{\infty}\leq C e^{-2 \g t} 
	\end{align*}
	for $t>1$, which requires condition (\ref{Assumption1}). Such assumption (\ref{Assumption1}) is useful to get rid of a \emph{resonance effect}, lastly due to the fact that both the Dirichlet heat kernel $p_t^{\Omega}(x,y)$ and the parameter $\mu(t)$ decay exponentially fast.
	Indeed, the long-term behavior of the Dirichlet heat kernel is
	\begin{align*}
		p_t^{\Omega}(x,y)\sim \phi_1(x)\phi_1(y)e^{-\lambda_1 t} \ass t \to \infty,
	\end{align*}
	where $\phi_1$ is the positive eigenfunction of $-\Delta$ in $\Om$ with $\norm{\phi_1}_2=1$.
	We recall the properties of the Dirichlet heat kernel in \S\ref{sec:invj}. More specifically, we use assumption (\ref{Assumption1}) in the following steps of the proof:
	\begin{itemize}
		\item to get estimates for $J_1,J_2$ in Lemma \ref{EstJ1fromdotLambda} and Lemma \ref{Lemma:estJ2} respectively;
		\item in Lemma \ref{LinearLemmapsi} for solving the outer problem;
		\item in Proposition \ref{Proposition:invj} for the invertibility theory of the nonlocal operator $\JJJ$.
	\end{itemize}
	%\paragraph*{Relation between $\g$ and $\bn$}
	The number $\g(q)$ is related to the Brezis-Nirenberg problem. Define
	\begin{align*}
		\SSS_a(\Om){\coloneqq}\inf_{u\in H_0^1(\Omega)\setminus\{0\}}  \frac{\int_{\Om}\abs{\nabla u}^2\ddx-a \int_{\Om}\abs{u}^2\ddx}{\qty(\int_{\Om}\abs{u}^6\ddx)^{\frac{1}{3}}}.
	\end{align*}
	In the celebrated work \cite{brezisnirenberg}, Brezis and Nirenberg proved the existence of a constant $\mu_{\text{BN}}\in (0,\lambda_1)$ such that
	\begin{align*}
		\bn{\coloneqq} \inf\{a>0: \SSS_a(\Om)<\SSS_0\}.
	\end{align*}
	Then, Druet \cite{druet} proved $$\min_{q \in \Om}\g(q)= \mu_{\text{BN}}(\Om).$$
	Thus, when $3\mu_{BN}(\Om)<\lambda_1(\Om)$ is true, condition (\ref{Assumption1}) is satisfied in some open set $\mathcal{O}\subset \Om$, and Theorem \ref{mainteo} gives the desired solution with blow-up at any fixed point $q\in \mathcal{O}$. 
	%\paragraph*{The unit ball $B_1$}
	
	When we consider the radial case $\Omega= B_1(0)$ and $q=0$, an explicit computation gives $\gamma(0)=\pi^2/4$, that is consistent with (\ref{muGK}). In fact, this is the minimum value for $\g(q)$ since Brezis and Nirenberg computed $\mu_{BN}(B_1)=\pi^2/4$. By symmetry, we deduce that condition (\ref{Assumption1}) is satisfied in the ball $B_{d^*}$, where $d^*=\abs{q^*}$ and $q^*$ is a point such that $\g\qty(q^*)=\lambda_1/3$.

	Also, we can consider smooth perturbation of the ball. Let $f: \bar B_1 \to \RR^3$ a smooth map and for $t>0$ define
	\begin{align*}
		\Omega_t:=\{x+t f(x)\,:\, x \in B_1\}.
	\end{align*}
	For small $t$ the domain $\Omega_t$ is diffeomorphic to the ball. Writing $\lambda_1$ as Rayleigh quotient and using the definition $\mu_{BN}$ we can easily see that $\mu(\Omega_t)=\mu(B_1)+\ve(t)$ and $\lambda_1(\Omega_t)=\lambda_1(B_1)+\tilde{\ve}(t)$ where $\ve(t),\tilde \ve(t)\to 0$ as $t\to 0$. Thus, for $t$ sufficiently small, the relation $3\mu_{\text{BN}}(\Omega_t)<\lambda_1(\Omega_t)$ holds, and Theorem \ref{mainteo} applies to the domain  $\Omega_t$. This shows that Galaktionov-King's {radial result is stable under small perturbation} of the domain.
	%\paragraph*{The unit cube $\CCC$}
	
	For the unit cube $\CCC_1$ it is known (see Remark 4.3 in \cite{wang}) that $3\mu_{\text{BN}}(\CCC_1)< \lambda_1(\mathcal{C}_1)$. Indeed, from $B_{1/2}\subset \mathcal{C}_1$ and the strict monotonicity of $\mu_{BN}(\Om)$ with respect to $\Om$ we deduce $\mu_{\text{BN}}\qty(\mathcal{C}_1)< \mu_{\text{BN}}\qty(B_{1/2})=\pi^2$. By separation of variables we easily compute $\lambda_1(\mathcal{C}_1)=3\pi^2$, thus
	\begin{align*}
		3 \mu_{\text{BN}}\qty(\mathcal{C}_1)< 3 \mu_{\text{BN}}\qty(B_{1/2})=3 \pi^2 =\lambda_1(\CCC_1).
	\end{align*}
	Hence, a slight modification of Theorem \ref{mainteo} applies: since $\CCC_1$ is a Lipschitz domain, by the parabolic regularity theory we get a smooth solution $u(x,t)$ in $\Omega\times \RR^+$ which is Lipschitz continuous in $\bar \Om \times [t_0,\infty)$.
	
	%\paragraph*{Estimates for other domains}
	Let $\Omega^*$ be the ball with the same volume as $\Omega$. The following estimate holds true:
	\begin{align*}
		\frac{\lambda_1(\Om^*)}{4}\leq \bn(\Om) \leq \frac{\lambda_1(\Om^*)}{4}\min_{x\in \Om}R_0(x)^2.
	\end{align*}
	The lower bound was proved in \cite{brezisnirenberg} by means of a symmetrization argument. Using harmonic transplantation Bandle and Flucher \cite{bandleflucher} proved the upper bound. Thus, if it happens that we know $\min_{x\in \Om} R_0(x)^2<4/3$ we can apply Theorem \ref{mainteo} to $\Om$. Wang \cite{wang} conjectured that $\bn/\lambda_1\in [1/4,4/9)$. In particular, condition $3\mu_{BN}(\Omega)<\lambda_1(\Omega)$ could be false for "very thin rectangles" (see \cite{wang}). The range $[1/4,4/9)$ is supported by numerical computations made by Budd and Humphries in \cite{budd}.
	
	The main differences with respect to the analogue result \cite{cdm} in dimension $n\geq 5$ are the following:
	\begin{itemize}
		\item the main asymptotic behavior in Theorem \ref{mainteo} of the blow-up is dependent on the position of the point $q\in \Om.$ As far as we know, this is a completely new phenomenon;
		\item since condition (\ref{Assumption1}) is not satisfied close to the boundary, we cannot straightforward construct multi-spike solutions in the spirit of \cite{cdm}. Indeed, roughly speaking, such construction requires spikes relatively far from each other and close to the boundary to suitably bound the interaction between the bubbles.
		\item a nonlocal operator controls the dynamic of the parameter $\mu(t)$. A similar operator has been treated in \cite{dmw1}, where the domain $\Om=\RR^3$ allows an explicit inversion of the Laplace transform. 
	\end{itemize}

	The approach developed in this work is inspired by \cite{cdm}, \cite{dmw1} and \cite{harmonic}. It is constructive and allows an accurate analysis of the asymptotic dynamics and stability. 
	Let describe the general strategy. The first step consists in choosing a good approximated solution $u_3$. Here the word 'good' means that the associated error function
	\begin{align*}
		S[u](x,t){\coloneqq}-\pp_t u+\Delta u+u^{5}
	\end{align*}  
	is sufficiently small in $\Omega$. Part of the problem consists in understanding what smallness on $S[u]$ is sufficient to find a perturbation $\tilde \phi$ such that 
	\begin{align*}
		u=u_3+\tilde \phi
	\end{align*}
	is an exact solution to (\ref{CHE}). In \S\ref{sec:ansatz} we start with the scaled Aubin-Talenti bubble as building block and we modify it to match the boundary at the first order. Then we realize that we need two improvements. The first one is a global correction
	useful to get solvability conditions for the elliptic linearized operator around the standard bubble 
	$$
	L[\phi]{\coloneqq}\Delta \phi+5U^4(y)\phi.
	$$ 
	Such improvement produces a nonlocal operator which governs the second order term in the expansion of the scaling parameter $\mu(t)$. This is a \emph{low-dimensional effect}, lastly due to the fact that
	\begin{align*}%\label{loweffect}
		Z_{n+1}(r){\coloneqq}\frac{n-2}{2}U(r)+U'(r)r \notin L^2(\mathbb{R}^n)\quad \text{when} \quad n\in \{3,4\}, 
	\end{align*} 
	where $Z_{n+1}$ is the unique (up to multiples) bounded radial function belonging to the kernel of $L[\phi]$. Actually, the dimensional restriction in \cite{cdm} was specially designed to avoid this effect and the presence of the corresponding nonlocal term.
	Then, by choosing $\g(q)$ as in (\ref{Assumption1}) we reduce the error close to $x=q$; this gives the asymptotic behavior (\ref{behmuxi}) of $\mu(t)$ at the first order. A second correction, local in nature, removes non-radial slow-decay terms and gives the asymptotic for $\xi$ written in (\ref{behmuxi}). At this point we have a sufficiently good ansatz, called $u_3$, to start the so called \emph{inner-outer gluing procedure} in \S\ref{sec:innout}: we decompose the problem in a system of nonlinear problems, namely an inner and an outer problem which are weakly coupled thanks to the smallness of $S[u_3]$. We solve the outer problem in \S\ref{sec:out}, that is a perturbation of the standard heat equation, for suitable parameters $\mu,\xi$ and decaying solution $\phi$ of the inner problem. Then, we look at the inner regime. We can find the inner solution, by fixed point argument, using the adaptation to $n=3$ of the linear theory for the inner problem developed in \cite{cdm}. This requires the solvability of orthogonality conditions which, in \S\ref{sec:Choiceparameters}, we prove to be equivalent to a \emph{nonlocal system} in the parameters $\mu,\xi$. We solve it in \S\ref{sec:solveSystem} using the invertibility of a nonlocal equation, which we achieve in \S\ref{sec:invj} by means of a Laplace transform argument combined with asymptotic properties of the heat kernel $p_t^\Om(x,y)$. At this point we are ready to find the inner solution $\phi$ in \S\ref{sec:finalargument}, which concludes the proof of Theorem \ref{mainteo}.
	
	Of course, the full problem consists in finding the exact initial datum that evolves in an infinite time blow-up solution. We find the positive initial condition
	\begin{align*}
		u(x,t_0)=&\mu(t_0)^{-1/2}U\qty(\frac{x-\xi(t_0)}{\mu(t_0)})-\mu(t_0)^{1/2}H_\g(x,t_0)+\mu_0(t_0)^{1/2}J_1(x,t_0)\\\nonumber
		&+\mu(t_0)^{-1/2}\phi_3\qty(\frac{x-\xi(t_0)}{\mu(t_0)},t_0)\eta_{l(t_0)}\qty(\abs{\frac{x-\xi(t_0)}{\mu(t_0)}})\\\nonumber
		&+\mu_0(t_0)^{1/2}\psi(x,t_0)+\eta_{R(t_0)}\qty(\abs{\frac{x-\xi(t_0)}{\mu(t_0)}})\mu(t_0)^{-1/2}e_0Z_0\qty(\frac{x-\xi(t_0)}{\mu(t_0)}),
	\end{align*}
	for $t_0$ fixed sufficiently large, where the existence of $\mu,\xi,\phi,\psi$ and the constant $e_0$ is a consequence of fixed point arguments, $\eta_l,l,\eta_R,R$ are defined in (\ref{mu0def}), (\ref{defetaL}) and the functions $\phi_3,J_1$ solve the problems (\ref{phi3problem}) and (\ref{ProblemJ_1}). We remark that we do not know if the solution with this initial datum corresponds to a threshold solution in the sense of \cite{nst}.\\

	We conclude this introduction giving a short bibliographic overview about related problems and recent developments. 
	The rigorous construction of blow-up solutions by bubbling, that is a solution $u(x,t)\approx U_{\mu(t),\xi(t)}(x)$ with $\mu \to 0$ for some special profile $U$, has been extensively studied in many important problems with criticality. For instance, in the harmonic map flow \cite{harmonic, RSstableflow, quanti}, in the Patlak-Keller-Segel model for chemotaxis \cite{ghoulmasmoudi, refined, collapsing, ksdelpino}, in the energy-critical wave equation \cite{dkm, km, kst,smoothII} and energy critical Schrödinger map problem \cite{mrr}.
	
	Concerning the Cauchy problem
	\begin{align*}
		&u_t =\Delta u +u^{\frac{n+2}{n-2}} \inn \RR^n \times \RR^+,\\
		&u(x,0)=u_0(x) \inn \RR^n,
	\end{align*}
	infinite blow-up positive solutions have been found in dimension $n=3$ in \cite{dmw1} by del Pino, Musso and Wei, with different blow-up rates depending on the space decay of the initial datum. Recently, Wei, Zhang and Zhou \cite{wei4D} detected analogue solutions in dimension $n=4$. 
	These works were inspired by conjectures presented in \cite{filaking}, where Fila and King used matched asymptotic methods to formally analyze the behavior of infinite blow-up solutions in the radial case, also conjecturing that for $n\geq 5$ such solutions do not exist. However, adding drift terms to the equation, Wang, Wei, Wei, Zhou \cite{drift} have shown examples of positive initial datum which evolves in multi-spike infinite blow-up by bubbling. For $n\geq 7$, del Pino, Musso and Wei \cite{bubbletowercauchy} proved the existence of sign-changing solutions which blow-up in infinite time in the form of tower of bubbles, that is a supersolution of Aubin-Talenti bubbles at a single point. For the analogue backward problem where $t\in (-\infty,0)$, ancient solutions which blow-up in infinite time have been detected by Sun, Wei, Zhang in \cite{ancientblown7} for $n\geq 7$.\\
	%Finite time blow-up
	As we have already mentioned, blow-up for the nonlinear heat equation
	\begin{align*}
		&u_t = \Delta u + \abs{u}^{p-1}u \inn \Omega \times (0,T),
	\end{align*}
	can also happen in finite time $T<\infty$. We call it Type I blow-up if the solution satisfies
	\begin{align*}
		\limsup_{t\to T} (T-t)^{\frac{1}{p-1}}\norm{u(\cdot,t)}_{L^\infty(\Om)}<\infty,
	\end{align*}
	otherwise, if
	\begin{align*}
		\limsup_{t\to T} (T-t)^{\frac{1}{p-1}}\norm{u(\cdot,t)}_{L^\infty(\Om)}=\infty,
	\end{align*} 
	we have Type II blow-up. Several works have focused on constructing finite time blow-up solutions for the Cauchy problem. Positive Type II blow-up solutions do not exist in dimension $n\geq 7$, see Wang and Wei \cite{TypeIINonEx}, or under radial assumptions in any dimension $n\geq 3$, see Matano-Merle \cite{matanomerle1} and the pioneering work by Filippas-Herrero-Vel\'azquez \cite{filippas}. In dimension $n \geq 7$, Collot, Merle and Rapha\"el in \cite{collot} classified the dynamics near the Aubin-Talenti bubble $U$ in the $\dot{H}^1$ topology. In particular, they ruled out the Type II scenario for initial conditions $u_0$ such that $\norm{u_0-U}_{\dot H^1(\RR^n)}$ is sufficiently small. The existence of positive Type II blow-up in dimensions $n\in \{3,4,5,6\}$ is an open problem.
	
	Type II blow-up it is still admissible for sign-changing solutions, and in fact examples have been found.
	Type II blow-up solutions have been constructed by Schweyer \cite{SchweyerTypeII} in dimension $4$ under radial assumption and later by del Pino, Musso, Wei and Zhou in the non-radial setting \cite{typeIIdmwzR4} with admissible multi-spike behavior. Also, Type II blow-up solutions have been detected in dimension $n\in \{3,5,6\}$ in \cite{typeII3Ddpmwzz,type2dmwc5D,Harada5,Harada6} with different blow-up rates. 
	Type II blow-up for the critical heat equation can also happen on curves contained in the boundary of special domains with axial symmetry, see \cite{geodriven}.

	There have been developments also in the nonlocal generalization of these problems. Concerning the fractional heat equation with critical exponent
	\begin{align*}
		u_t = - \qty(-\Delta)^s u + \abs{u}^{\frac{4s}{n-2s}}u,
	\end{align*}
	Cai, Wang, Wei, Yang \cite{caiwei2022} have recently constructed solutions for both the forward and backward Cauchy problem which are sing-changing tower of bubbles at the origin for $n>6s$, and $s\in (0,1)$.  For $n \in (4s,6s)$ and $s\in (0,1)$ blow-up in finite time has been proved in \cite{cwz2020}, which is a fractional continuation of the local Type II blow-up cases $n=4,s=1$ in \cite{SchweyerTypeII} and $n=5,s=1$ in \cite{type2dmwc5D}. Regarding the associated Dirichlet problem,
	Musso, Sire, Wei, Zheng and Zhou provided in \cite{mswzzFrac} the existence of positive multi-spike infinite-time blow-up on bounded smooth domains for $n\in (4s,6s)$ and $s\in (0,1)$.

	\section{Approximate solution and estimate of the associated error}\label{sec:ansatz}
	In this section we construct an approximate solution to the problem
	\begin{align}\label{eqCHE}
		\begin{cases}
			u_t = \Delta u + u^5 & \inn \Omega \times \RR^+,\\
			u=0& \onn \partial \Omega \times \RR^+,
		\end{cases}
	\end{align}
	and we compute the associated error. Without loss of generality, we construct a solution that blows-up at $q=0 \in \Om$.
	The first approximation $u_1$ is chosen by selecting a time-scaled version of the stationary solution to the Yamabe problem
	\begin{align*}
		\Delta U + U^5=0 \inn \RR^3,
	\end{align*}
	properly adjusted to be small at the boundary $\pp \Omega$. This is constructed in \S\ref{constr:u_1}. In order to make the error small at the blow-up point, we need to select a precise first order for the dilation parameter $\mu(t)$, which matches the radial asymptotic found in \cite{gk2}. However, we observe in \S\ref{subsec:erroru_1} that $u_1$ is not close enough to an exact solution to make our perturbative scheme rigorous. In \S\ref{constr:u_2} we make a global improvement $u_2$. Such correction involves a nonlocal operator in the lower order term of $\mu(t)$, similar to a half-fractional Caputo derivative. The last improvement $u_3$ is only local, and it removes slow-decaying terms in non-radial modes by selecting the first order asymptotic of the translation parameter $\xi(t)$.

	\subsection{First global approximation}\label{constr:u_1}
	Our building blocks are the scaled Aubin-Talenti bubbles (\ref{Talentibubble})
	which satisfy
	\begin{align}\label{eqTal}
		\Delta U_{\mu,\xi}+ U_{\mu,\xi}^5=0 \inn \RR^3.
	\end{align}
	We look for a solution of the form $u_1(x,t) \approx U_{\mu(t),\xi(t)}(x)$.
	We make an ansatz for the parameters $\mu(t),\xi(t)$. Assuming that $\mu(t)\to 0$ and $\xi\to 0\in \Om$ as $t\to \infty$, we notice that $U_{\mu,\xi}(x)$ is concentrating around $x=0$ and it is uniformly small away from it. For this reason, we should have
	\begin{align}\label{u_1eqappr}
		\pp_t u_1- \Delta u_1
		&=u_1(x,t)^5 \\\nonumber
		&\approx \delta_0(x-\xi) \int_{\RR^3} \qty(\mu^{-1/2}U\qty(\frac{x-\xi}{\mu}))^5 \ddx\\ \nonumber
		&= \delta_0(x-\xi)\mu^{1/2} \int_{\RR^3} U(y)^5 \ddy \\ \nonumber
		&= \delta_0(x-\xi) c_3 \mu^{1/2}.
	\end{align}
	
	Let $\mu_0(t)$ the first order of $\mu(t)$, that is
	\begin{align*}
		\mu(t)=\mu_0(t)(1+o(1)) \ass t\to \infty.
	\end{align*}
	From (\ref{u_1eqappr}) we define the scaled function 
	$$v(x,t){\coloneqq}\mu^{-1/2}u_1(x,t),$$ 
	which should satisfy
	\begin{align}\label{eqdilated}
		&v_t \approx \Delta v + \qty( -\frac{\dot \mu}{2\mu}	)v + c_3 \delta_0(x-\xi) \inn \Omega \times \RR^+,\\\nonumber
		&v= 0 \onn \partial \Omega \times \RR^+.
	\end{align}
	We choose the parameter $\mu_0(t)$ such that
	\begin{align*}%\label{choice_mu_0}
		-\frac{\dot\mu_0(t)}{2\mu_0(t)}=\gamma, 
	\end{align*}
	for some $\gamma \in \RR^+$ that will be fixed later. This is equivalent to choose
	\begin{align}\label{mu0def}
		\mu_0(t)= b e^{-2\gamma t},
	\end{align}
	for some $b\in \RR^+$. We can fix $b=1$. Indeed, the equation is translation-invariant in time: we construct, for a sufficiently large initial time $t_0$, a solution $u(x,t)$ in $\Omega \times [t_0,\infty)$ and we conclude that $u_0(x,t){\coloneqq} u(x,t-t_0)$ is a solution to (\ref{eqCHE}) in $\Omega \times [0,\infty)$. 
	We observe that after shifting the initial time, the main dilation parameter $\mu_0$ becomes $\mu_0(t-t_0)=e^{2\gamma t_0}e^{-2\gamma t}$.
	
	With this choice (\ref{eqdilated}) reads
	\begin{align*}%\label{eqv_t}
		&v_t \approx  \Delta v +\gamma v+c_3 \delta_0(x-\xi)  \inn \Omega \times \RR^+,\\
		&v=0 \onn \partial \Omega \times \RR^+.
	\end{align*}
	Hence, for large time we should have
	\begin{align}\label{v_app}
		v(x,t)\approx   G_{\gamma}(x,\xi),
	\end{align}
	where $G_\gamma(x,y)$ is the Green function for the boundary value problem
	\begin{align}\label{Ggammaequation}
		&-\Delta_x G_\gamma(x,y)-\gamma G_{\gamma}(x,y)=c_3 \delta(x-y) \inn \Omega,\\\nonumber
		&G(\cdot,y)=0 \onn \partial\Omega.
	\end{align}
	We write
	\begin{align}\label{gdec}
		G_\gamma(x,y)=\Gamma(x-y)-H_\gamma(x,y),
	\end{align}
	where 
	$$
	-\Delta_x \Gamma(x)=c_3 \delta_0(x), \quad \Gamma(x)=\frac{\at}{\abs{x}}
	$$ 
	is (a multiple of) the fundamental solution of the Laplacian in $\RR^3$ and the regular part $H_\gamma(x,y)$, for fixed $y\in \Om$, satisfies
	\begin{align}\label{eqH_g}
		&-\Delta_x H_\gamma(x,y)-\gamma H_\gamma(x,y)=-\gamma \Gamma(x-y)\inn \Omega,\\\nonumber
		&H_\gamma(x,y)=\Gamma(x-y)\onn \partial \Omega .
	\end{align} 
	The function $H_\g(\cdot,y)\in C^{0,1}(\Omega)$ when $\g\in (0,\lambda_1)$. For later purpose, we also write 
	\begin{align}\label{decHg}
		H_\g(x,y)=\theta_\g(x-y)-h_\g (x,y),
	\end{align}
	where 
	\begin{align}\label{defthetag}
		\theta_\g(x){\coloneqq}\alpha_3 \frac{1-\cos(\sqrt{\g}\abs{x})}{\abs{x}}
	\end{align}
	and $h_\g(\cdot,y)\in C^\infty(\Omega)$ solves
	\begin{align}\label{eqhg}
		&\Delta_x h_\g(x,y)+\g h_\g(x,y)=0 \inn \Omega,\\\nonumber
		&h_\g(x,y)=-\alpha_3 \frac{\cos(\sqrt{\g}\abs{x-y})}{\abs{x-y}} \onn \pp \Omega.
	\end{align}
	We also define the Robin function 
	\begin{align*}
		R_\g(x){\coloneqq}H_\g(x,x)=h_\g(x,x).
	\end{align*}
	In terms of the original function $u_1$ the equation \reff{v_app} reads as
	\begin{align*}
		u_1(x,t)\approx \mu^{1/2}\frac{\alpha_3 }{\abs{x-\xi}} -  \mu^{1/2} H_\gamma(x,\xi).
	\end{align*}
	We notice that far away from the origin we have
	\begin{align*}
		U_{\mu,\xi}(x)\approx  \mu^{1/2}\frac{\alpha_3}{\abs{x-\xi}}.
	\end{align*}
	This formal analysis suggests the ansatz
	\begin{align*}
		u_1(x,t){\coloneqq} U_{\mu,\xi}(x)- \mu^{1/2}H_\gamma(x,\xi).
	\end{align*}
	\subsubsection{Dilation parameter $\mu(t)$}
	We write the full dilation parameter in the form
	\begin{equation*}
		\mu=\mu_0(t)e^{2\Lambda(t)},
	\end{equation*}
	for some $\Lambda(t)=o(1)$ as $t\to \infty$ to be found, where
	\begin{equation*}
		\mu_0(t)=e^{-2\gamma t}.
	\end{equation*}
	In this notation we have
	\begin{align*}
		\frac{\dot \mu(t)}{2\mu(t)}
		&=\frac{\dot \mu_0 e^{2\Lambda}}{2\mu_0 e^{2\Lambda}}+\frac{2\dot{\Lambda} \mu_0 e^{2\Lambda}}{2\mu_0 e^{2\Lambda}}\\&
		=-\gamma+\dot \Lambda(t),
	\end{align*}
	and
	\begin{equation*}
		\Lambda(t)=-\int_{t}^{\infty}\dot \Lambda(s)\dds,
	\end{equation*}
	where $\dot \Lambda(s)$ is an integrable function in any $[t_0,\infty)$. 
	\subsection{Error associated to $u_1$}\label{subsec:erroru_1}
	\medskip
	The next step consists in computing the error associated to the first ansatz $u_1$. We define the error operator
	\begin{align*}
		S[u]\coloneqq -\partial_t u+ \Delta u+ u^5.
	\end{align*}
	Of course, solving $S[u]=0$ is equivalent to solve the equation in (\ref{eqCHE}). 
	It is well-known that all bounded solutions to the linearized operator
	\begin{align*}
		\Delta_y \phi+ 5U(y)^{4}\phi =0 \inn \RR^n,
	\end{align*}
	are linear combinations of the functions $\{Z_i\}_{i=1}^4$ defined as
	\begin{align*}
		Z_i(y)\coloneqq\partial_{y_i} U(y), \quad i=1,2,3,
	\end{align*}
	and
	\begin{align*}
		Z_4(y)\coloneqq\frac{1}{2}U(y)+y\cdot \nabla U(y)=\frac{\alpha_3}{2} \frac{1-\abs{y}^2}{\qty(1+\abs{y}^2)^{3/2}}.
	\end{align*}
	We define the scaled variable
	\begin{align*}
		y\coloneqq y(x,t)\coloneqq\frac{x-\xi(t)}{\mu(t)}.
	\end{align*}
	Now, we compute $S[u_1](x,t)$ for $x\neq \xi(t)$. We have
	\begin{align*}
		\Delta u_1&=\mu^{-1/2}\Delta_x U\qty(\frac{x-\xi}{\mu}) - \mu^{1/2}\Delta_x H_\gamma(x,\xi)\\
		&=-\mu^{-5/2}U(y)^5+\mu^{1/2}\qty(\gamma H_\gamma(x,\xi)-\frac{\gamma \alpha_3}{\abs{x-\xi}})\\
		&=-\mu^{-5/2}U(y)^5+\mu^{1/2}\gamma H_\gamma(x,\xi)-\mu^{1/2}  \frac{\gamma\at }{\abs{x-\xi}},
	\end{align*}
	where we used equations \reff{eqTal} and \reff{eqH_g} for $U$ and $H_\gamma$ respectively.
	Using the definition of $Z_4$, the time-derivative can be written as
	\begin{align*}
		\partial_t u_1=&-\frac{1}{2}\frac{\dot \mu}{\mu}\mu^{-1/2}U(y)+\mu^{-1/2}\nabla_y U(y)\cdot \qty[-\frac{\dot \xi}{\mu}-\frac{\dot \mu}{\mu}y]\\
		&-\frac{1}{2}\frac{\dot \mu}{\mu} \mu^{1/2}H_\gamma(x,\xi)-\mu^{1/2}\dot \xi \cdot \nabla_{x_2}H_\gamma(x,\xi)\\
		=&-\qty(\frac{\dot \mu}{2\mu})\qty[\mu^{-1/2}2Z_4(y) +\mu^{1/2}H_\gamma(x,\xi)	]\\&-\mu^{-3/2}\dot\xi \cdot \nabla_y U -\mu^{1/2}\dot \xi \cdot \nabla_{x_2}H_\gamma(x,\xi).
	\end{align*}
	Hence, the error associated to $u_1$ is
	\begin{align}\label{errorS[u_1]}
		S[u_1]=&\dot \Lambda \qty(\mu^{-1/2}2Z_4(y) +\mu^{1/2}H_\g(x,\xi))-\g \mu^{-1/2}\qty( 2Z_4(y)+\frac{ \alpha_3}{\abs{y}})\\\nonumber
		&+\mu^{-3/2}\dot \xi \cdot \nabla_y U(y)+\mu^{1/2}\dot \xi \cdot \nabla_{x_2}H_\gamma(x,\xi) \\\nonumber
		&-\mu^{-3/2} 5U(y)^4 H_\g(x,\xi)\\\nonumber
		&+\mu^{-5/2}\qty[	\qty(U(y)-\mu H_\gamma(x,\xi))^5-U(y)^5+\mu 5U(y)^4 H_\g(x,\xi)].
	\end{align}
	\subsection{Global improvement}\label{constr:u_2}
	The remaining part of this section concerns the improvement of the natural ansatz $u_1$. Later in the argument we will divide the error in outer and inner part. We realize that solving the inner-outer system requires a global and a local improvement. 
	Reading Proposition \ref{propInnerlineartheory}, which is the linear theory for the inner problem, we see that, to get decay in $\phi(y,\tau)$ at distance $R$ we need $a>1$ in the definition of $\norm{h}_{\nu,2+a}$. This smallness at distance $R$ will make the inner and outer regime weakly decoupled. Our particular $h$ will satisfy $\norm{h}_{\nu,2+a}<\infty$ with $a=2$, hence we will use estimate (\ref{estinnSimplified}). Thus, we say that a term is slow-decay in space if it is not controlled by $(1+\abs{y})^{-4}$.
	We can find an exact perturbation with our scheme if we remove such terms. Looking at (\ref{errorS[u_1]}) we observe that all the terms in the first two lines are slow-decay. 
	Using the inequality $\mu(t)\lesssim (1+\abs{y})^{-1}$ we can negotiate decay in time with decay in space if needed in other terms. {For the moment we can assume $\dot \Lambda, \Lambda,  \dot \xi,\xi$ bounded by some power of $\mu(t)$.} Later we shall specify precise norms for these parameters.
	Firstly, we decompose
	\begin{align*}
		\mu^{-3/2} 5U(y)^4 H_\g(x,\xi)=&\mu^{-3/2} 5U(y)^4 \theta_\g(x-\xi) \\
		&-\mu^{-3/2}5U(y)^4 h_\g(x,\xi).
	\end{align*}
	We define
	\begin{align*}
		u_2(x,t)=u_1(x,t)+\mu_0^{1/2}J[\dl](x,t).
	\end{align*}
	The new error reads as 
	\begin{align*}
		S[u_2]&=S[u_1]+\qty(-\pp_t +\Delta_x)\qty(\mu_0^{1/2}J[\dl](x,t))+u_2^5-u_1^5\\
		&=S[u_1]+\mu_0^{1/2}\qty(-\pp_t  + \Delta_x + \gamma )J[\dl]+u_2^5-u_1^5.
	\end{align*}
	Let
	\begin{align*}
	J[\dl](x,t):=J_1[\dl](x,t)+J_2(x,t).
	\end{align*}
	Plugging $S[u_1]$ given by (\ref{errorS[u_1]}) into $S[u_2]$ we get
	\begin{align*}
		S[u_2]=& \mu^{-3/2}\dot \xi \cdot \nabla_y U(y) +\mu^{1/2}\dot \xi \cdot \nabla_{x_2} H_\g(x,\xi)+\mu^{-3/2}5U(y)^4h_\g(x,\xi)\\
		&+\mu_0^{1/2}\qty{(-\pp_t + \Delta_x + \gamma )J_1+\qty(\frac{\mu}{\mu_0})^{\frac{1}{2}}\dl\qty( \mu^{-1} 2Z_4\qty(y)+H_\g(x,\xi))}\\
		&+\mu_0^{1/2}\qty{(-\pp_t + \Delta_x + \gamma ) J_2-\qty(\frac{\mu}{\mu_0})^{\frac{1}{2}}\qty[\gamma \mu^{-1}\qty(2Z_4(y)+\frac{\alpha_3}{\abs{y}})+\mu^{-2}5U^4 \theta_\g(\mu y)]}\\
		&+\mu^{-5/2}\qty[\qty(U(y)-\mu H_\g(x,\xi)+\mu \qty(\frac{\mu_0}{\mu})^{1/2}J[\dot \Lambda](x,t))^5-U(y)^5+\mu 5U(y)^4 H_\g(x,\xi)].
	\end{align*}
	We select $J_1[\dot \Lambda](x,t)$ such that
	\begin{align}\label{ProblemJ_1}
		&\pp_t J_1 = \Delta_x J_1+ \gamma J_1+\qty(\frac{\mu}{\mu_0})^{\frac{1}{2}}\dl\qty( \mu^{-1} 2Z_4\qty(\frac{x-\xi}{\mu})+H_\g(x,\xi))\inn \Omega \times [t_0-1,\infty),\\\nonumber
		&J_1(x,t)=0 \inn \pp \Omega \times [t_0-1,\infty),\\\nonumber
		&J_1(x,t_0-1)=0 \onn \Omega,
	\end{align}
	and 
	\begin{align}\label{ProblemJ2}
		\pp_t J_2 =& \Delta_x J_2+ \gamma J_2-\qty(\frac{\mu}{\mu_0})^{\frac{1}{2}}\bigg[\gamma \qty(\mu^{-1}2Z_4\qty(\frac{x-\xi}{\mu})+\frac{\alpha_3}{\abs{x-\xi}})\\\nonumber
		&+\mu^{-2}5U\qty(\frac{x-\xi}{\mu})^4 \theta_\g(x-\xi)\bigg] \inn \Omega \times [t_0,\infty),\\\nonumber
		J_2(x,t)&=0 \onn \pp \Omega \times [t_0,\infty),\\\nonumber
		J_2(x,t_0&)=0 \inn \Omega.
	\end{align} 
	The choice of defining $J_1$ from the time $t_0-1$, as well as $\dot \Lambda(t)$, will become clear in \S\ref{sec:invj}. For the variable $\xi(t)$ it is enough to define the extension $\xi(t)=\xi(t_0)$ for $t\in [t_0-1,t_0)$.
	With these choices the error associated to $u_2$ reads as
	\begin{align}\label{errS[u_2]}
		S[u_2]=& \mu^{-3/2}\dot \xi \cdot \nabla_y U(y) +\mu^{1/2}\dot \xi \cdot \nabla_{x_2}H_\g(x,\xi)+\mu^{-3/2}5U(y)^4h_\g(x,\xi)\\\nonumber
		&+\mu^{-5/2}\qty[\qty(U(y)-\mu H_\g(x,\xi)+\mu \qty(\frac{\mu_0}{\mu})^{1/2}J[\dot \Lambda](x,t))^5-U(y)^5+\mu 5U(y)^4 H_\g(x,\xi)].
	\end{align}
	\subsubsection{Choice of $\gamma$}
	We observe that with this choice of $J_2$ we remove the singular term $\abs{x-\xi}^{-1}$ from (\ref{errorS[u_1]}).
	At this point, the main error at $x=\xi(t)$ is given by the first order of the nonlinear term
	\begin{align*}
		\mu^{-3/2}5 U(0)^4 R_\g(\xi),
	\end{align*}
	which is, in general, of size $\mu(t)^{-3/2}$. 
	We realize that we can reduce this error by selecting $\gamma$ such that $R_\g(0)=0$. The existence and uniqueness of such number is given by the following lemma.
	\begin{lemma}\label{Lemma:gammaexists}
		There exists a unique $\g=\gamma^*(0) \in (0,{\lambda_1})$ such that $R_{\gamma^*}(0)=0$.
	\end{lemma}
	\begin{proof}
		We consider the function $R_\g(0)$ as a function of $\g$. Lemma A.2 in \cite{nearcrit} shows that 
		$$
		R_\g(0):(0,\lambda_1)\to (-\infty, R_0(0))
		$$ 
		is smooth in $(0,\lambda_1)$ and $\pp_\g R_\g(0)<0$. Lemma \ref{asycloselambda} in \autoref{app:propH} shows that $R_\g(0)\to -\infty$ as $\gamma \to \lambda_1^{-}$. By the maximum principle $H_0(x,y)>0$ for all $x,y\in \Omega$, hence we have $R_0(0)>0$ and the intermediate value theorem gives the existence of
		\begin{align*}
			\gamma^*(0)\coloneqq\sup \{	\gamma>0 :\, R_\gamma(0)>0	\}.
		\end{align*}
		Finally the monotonicity of $R_\g(0)$ implies the uniqueness of $\gamma^*(0)$. 
	\end{proof} 
	
	\begin{remark}[Regularity of $\gamma^*(x)$]
		Let $R(\g,x) \coloneqq R_{\g}(x)$. Since $R(\g^*(x),x)=0$ and $\pp_\g R(\g,x)<0$ for all $x\in \Omega$, the implicit function theorem implies that $\gamma^*(x)\in C^1(\Omega)$ with
		\begin{align*}
			\nabla_x \gamma^*(x)=-\frac{\nabla_x R(\g,x)}{\pp_\g R(\g,x)}.
		\end{align*}
	\end{remark}
	\begin{remark}[radial case]
		We compute $\g(0)$ in case $\Omega=B_1(0)$. We look for a radial solution to
		\begin{align*}
			&\Delta H_\g+\g H_\g =\frac{\at}{\abs{x}}\inn B_1,\\
			&H_\g(x,0)=\frac{\at}{\abs{x}}\onn \pp B_1.
		\end{align*}
		We define $l_0(\abs{x}){\coloneqq}H_\g(x,0)$ for a function $l_0:[0,1]\to \RR$. Then $l_0$ solves
		\begin{align*}
			&\pp_{rr} l_0 +\frac{2}{r}\pp_r l_0 +\g l_0=\g \frac{\at}{r}\inn [0,1],\\
			&l_0(1)=\at, \quad l_0(r) \quad \text{bounded at } r=0.
		\end{align*}
		We write $l_0(r)=\at l(r)/r$, where $l(r)$ solves
		\begin{align*}
			&\pp_{rr}l+\g l= \g \inn [0,1],\\
			&l(1)=1, \quad l(r)=O(r)\quad \text{for } r\to 0.
		\end{align*}
		The solution to this problem is given by
		\begin{align*}
			l(r)=1-\cos(\sqrt{\g}r)+\cot(\sqrt{\g})\sin(\sqrt{\g}r),
		\end{align*}
		and we conclude that
		\begin{align*}%\label{Hgradial}
			H_\g(r,0)=\at \qty[ \frac{1-\cos(\sqrt{\g}r)}{r}+\frac{\sin(\sqrt{\g}r)}{r\tan(\sqrt{\g})}].
		\end{align*}
		In particular, for $r=0$ we find
		\begin{align*}
			R_\g(0)=H_\g(0,0)=\at \sqrt{\g}\cot(\sqrt{\g}).
		\end{align*}
		Asking for $R_\g(0)=0$
		\begin{align*}
			\gamma=\qty(\frac{\pi}{2}+k\pi)^2\quad \text{for} \quad k\in \mathbb{N},
		\end{align*}
		and, recalling that $\lambda_1(B_1)=\pi^2$, the unique value in $(0,\lambda_1)$ is
		\begin{align*}
			\gamma^*=\frac{\pi^2}{4},
		\end{align*}
		as predicted in the analysis of Galaktionov and King \cite{gk2}. 
	\end{remark}
	For the sake of simplicity we continue to use $\g=\g(0)$ to denote the selected number $\g^*(0)$. Since $R_\g(x)\in C^\infty(\Om)$ we expand
	\begin{align*}
		R_\g(\xi)=R_\g(0)+\xi \cdot \nabla_x R_\g(0)+\frac{1}{2}\xi^{\intercal} D_{xx}^2 R_\g\qty(\xi^*) \xi,
	\end{align*}
	for some $\xi^*\in \overline{[0,\xi]}$. Assuming $\xi=O(\mu)$ we conclude
	\begin{align*}
		\mu^{-3/2}5U(0)^4 R_\g(\xi)=O\qty(\mu^{-1/2}).
	\end{align*}
	
	\subsection{Local improvement and computation of the final error}%\label{constr:u_3}
	In this section we make a further improvement and we obtain the final ansatz.
	We still need to remove from (\ref{errorS[u_1]}) the main order of the terms
	\begin{align*}
		\mu^{-3/2}\dot \xi \cdot \nabla_y U + \mu^{-3/2}5U(y)^4 h_\g(x,\xi).
	\end{align*}
	We define the final ansatz
	\begin{align*}
		u_3(x,t)\coloneqq u_2(x,t)+\mu(t)^{-1/2}\phi_3\qty(\frac{x-\xi(t)}{\mu(t)},t)\eta_{l(t)}\qty(\abs{\frac{x-\xi(t)}{\mu(t)}}).
	\end{align*}
	The function $\eta: [0,\infty)\to [0,1]$ denotes a smooth cut-off function such that $\eta(s)\equiv 1$ for $s<1$ and $\supp \eta \subset [0,2]$, and we define 
	\begin{align}\label{defetaL}
		\eta_{l(t)}\qty(\abs{y})\coloneqq\eta\qty(\frac{\abs{y}}{l(t)}),\qquad l(t)\coloneqq\frac{1}{k_2\mu},
	\end{align}
	where $k_2$ is a constant such that $B_{\frac{2}{k_2}}(0)\subset \Omega$, to ensure that $\supp \eta_{l} \Subset \Om$. Also we define the variable 
	$$z_3(x,t)\coloneqq \frac{y(x,t)}{l(t)}=\frac{x-\xi(t)}{\mu(t) l(t)}.$$
	We compute
	\begin{align*}
		\pp_t \qty(\mu^{-1/2}\phi_3 \eta_{l(t)})=&-\frac{\dot \mu}{2\mu} \mu^{-1/2}\phi_3 \eta_{l(t)}+\mu^{-1/2}\eta_{l(t)}\qty[\pp_t \phi_3+\nabla_y \phi_3 \cdot \qty(-\frac{\dot \mu}{\mu}y-\frac{\dot \xi}{\mu})]\\&+\mu^{-1/2}\phi_3 \pp_t \eta_{l(t)},
	\end{align*}
	and
	\begin{align*}
		\Delta_x\qty(\mu^{-1/2}\phi_3 \eta_{l})=&\mu^{-5/2}\eta_{l(t)}\Delta_y \phi_3 + 2\mu^{-3/2}\nabla_y \phi_3 \cdot\frac{y}{\abs{y}} \qty(\frac{\eta'\qty(\abs{z_3})}{\mu l})\\&+\mu^{-1/2}\phi_3 \qty(\frac{2}{\abs{z_3}}\frac{\eta'(\abs{z_3})}{\mu^2 l^2}+\frac{\eta''\qty(\abs{z_3})}{\mu^2 l^2}).
	\end{align*}
	We define
	\begin{align*}
		\mathcal{N}_3(y,t){\coloneqq}&\qty(U(y)^5- \mu H_\g(\mu y+\xi,\xi)+\mu \qty(\frac{\mu_0}{\mu})^{1/2}J[\dot \Lambda](\mu y +\xi,t)+\phi_3(y,t)\eta_l)^5-U(y)^5 \\&- 5U(y)^4 \qty(-\mu H_\g(\mu y +\xi,\xi)+\mu \qty(\frac{\mu_0}{\mu})^{1/2}J[\dot \Lambda](\mu y +\xi,t)+\phi_3 \eta_l)
	\end{align*}
	Thus, using (\ref{errS[u_2]}),
	\begin{align*}
		S[u_3]=&-\pp_t\qty(\mu^{-1/2}\phi_3 \eta_{l})+\Delta_x\qty(\mu^{-1/2}\phi_3 \eta_{l})+u_3^5-u_2^5+S[u_2]\\
		=&\mu^{-3/2}\dot \xi \cdot \nabla_y U+\mu^{1/2}\dot \xi \cdot \nabla_{x_2} H_\g(x,\xi)+\mu^{-3/2}5U(y)^4 h_\g(x,\xi)\\
		&+\mu^{-5/2}\mathcal{N}_3(y,t)+ 5U(y)^4 \mu^{-3/2} \qty(\frac{\mu_0}{\mu})^{1/2} J(x,t)+\mu^{-5/2}\phi_3 \eta_l 5U(y)^4\\
		&-\qty(-\frac{\dot \mu}{2\mu} )\mu^{-1/2}\phi_3 \eta_{l(t)}-\mu^{-1/2}\eta_{l(t)}\qty[\pp_t \phi_3+\nabla_y \phi_3 \cdot \qty(-\frac{\dot \mu}{\mu}y-\frac{\dot \xi}{\mu})]\\
		&-\mu^{-1/2}\phi_3 \pp_t \eta+\mu^{-5/2}\eta_{l(t)}\Delta_y \phi_3 + 2\mu^{-3/2}\nabla_y \phi_3 \cdot\frac{y}{\abs{y}} \qty(\frac{\eta'\qty(\abs{z_3})}{\mu l})\\
		&+\mu^{-1/2}\phi_3 \qty(\frac{2}{\abs{z_3}}\frac{\eta'(\abs{z_3})}{\mu^2 l^2}+\frac{\eta''\qty(\abs{z_3})}{\mu^2 l^2}).
	\end{align*}
	By Taylor expanding $h_\g(x,\xi)$ centered at $x=\xi$ we have
	\begin{align}\label{exphg1}
		h_\g(x,\xi)=R_\g(\xi)+ \mu y\cdot  \nabla_{x_1}h_\g(\xi,\xi) + \frac{1}{2}\mu^2 y^\intercal D_{xx}h_\g(\bar{x},\xi) y
	\end{align}
	for some $\bar{x}\in \overline{[\xi,x]}$. Now, we expand the first terms at $(\xi,\xi)=(0,0)$. By the Chain Rule we have $\nabla_{x_1}h_\g(x,x)=2 \nabla_x R_\g(x)$. Hence, we have
	\begin{align*}
		\nabla_{x_1}h_\g(\xi,\xi)=\frac{1}{2}\nabla_{x}R_\g(\xi)=\frac{1}{2}\nabla_{x}R_\g(0)+\frac{1}{2}\xi^\intercal D_{xx}R_\g\qty(\xi^{**}),
	\end{align*}
	for some $\xi^{**}\in \overline{[0,\xi]}$. Furthermore, since $R_\g(0)=0$, we have
	\begin{align*}
		R_\g(\xi)= \xi \cdot \nabla_{x}R_\g(0)+\frac{1}{2}\xi^\intercal D_{xx}R_\g(\xi^{*})\xi 
	\end{align*}
	for some $\xi^* \in \overline{[0,\xi]}$. Plugging these identities in (\ref{exphg1}) we obtain
	\begin{align}\label{exphgg}
		h_\g(x,\xi)=& \xi \cdot \nabla_{x}R_\g(0)+\frac{1}{2}\mu y\cdot \nabla_{x}R_\g(0)\\\nonumber
		&+\frac{1}{2}\xi^\intercal  D_{xx}R_\g(\xi^{*})\xi +\frac{1}{2} \mu y^\intercal D_{xx}R_\g\qty(\xi^{**}) \xi\\\nonumber
		&+ \frac{1}{2}\mu^2 y^\intercal D_{xx}h_\g(\bar{x},\xi)y.
	\end{align}
	We write
	\begin{align*}
		\xi(t)=\xi_0(t)+\xi_1(t).
	\end{align*}
	Now, we assume the following decay for the parameters $\xi_1,\dot \xi_1, \Lambda,\dot \Lambda$: 
	\begin{align*}
		&\abs{\xi_1(t)}+\abs*{\dot \xi_1(t)}\leq C \mu(t)^{1+k},\\
		&\abs*{\Lambda(t)}\leq C\mu(t)^{l_0},\\
		&\abs*{\dot \Lambda(t)}\leq C \mu(t)^{l_1},
	\end{align*}
	for some positive constants $k,l_0,l_1$ to be chosen (in \S\ref{Choiceofconst}).
	We write the full error
	\begin{align*}
		S[u_3]=
		&\mu^{-3/2}\nabla_y U(y)\cdot \qty[\dot \xi-\mu^{-1}\mu_0 \dot \xi_0 ]\eta_l\\
		&+5U(y)^4 \qty[\mu^{-3/2}h_\g(x,\xi)-\mu^{-5/2}\mu_0 \qty(\frac{1}{2}\mu_0 y \cdot \nabla_x R_\g(0))]\eta_l\\
		&+\qty[\mu^{-3/2}\nabla_y U(y)\cdot \dot \xi +5U(y)^4 \mu^{-3/2}h_\g(x,\xi)](1-\eta_l)\\
		&+\mu^{1/2}\dot \xi \cdot \nabla_{x_2} H_\g(x,\xi)+\mu^{-5/2}\mathcal{N}_3(y,t)+ 5U(y)^4 \mu^{-3/2} \qty(\frac{\mu_0}{\mu})^{1/2} J(x,t) \\
		&-\qty(-\frac{\dot \mu}{2\mu} )\mu^{-1/2}\phi_3 \eta_{l(t)}-\mu^{-1/2}\eta_{l(t)}\qty[\pp_t \phi_3+\nabla_y \phi_3 \cdot \qty(-\frac{\dot \mu}{\mu}y-\frac{\dot \xi}{\mu})]-\mu^{-1/2}\phi_3 \pp_t \eta\\
		&+\mu^{-5/2}\eta_{l(t)}\qty[\Delta_y \phi_3+5U(y)^4\phi_3 +\mathcal{M}[\mu_0,\xi_0]] \\&+ 2\mu^{-3/2}\nabla_y \phi_3 \cdot\frac{y}{\abs{y}} \qty(\frac{\eta'\qty(\abs{z_3})}{\mu l})+\mu^{-1/2}\phi_3 \qty(\frac{2}{\abs{z_3}}\frac{\eta'(\abs{z_3})}{\mu^2 l^2}+\frac{\eta''\qty(\abs{z_3})}{\mu^2 l^2}),
	\end{align*}
	where
	\begin{align}\label{defM}
		\mathcal{M}[\mu_0,\xi_0]{\coloneqq} \mu_0 \dot \xi_0\cdot \nabla_y U(y)- \frac{5}{2}U(y)^4 \mu_0 \qty(\mu_0 y \cdot \nabla_x R_\g(0))
	\end{align}
	For any fixed $t>t_0$, we select $\phi_3(x,t)$ as the bounded solution to the elliptic problem
	\begin{align}\label{phi3problem}
		\Delta_y \phi_3(y,t) + 5U(y)^4 \phi_3(y,t) = - \mathcal{M}[\mu_0,\xi_0](y,t) \inn \RR^3,
	\end{align}
	with the following orthogonality conditions on the right-hand side:
	\begin{align}\label{orthM}
		\int_{\RR^3} \mathcal{M}[\mu_0,\xi_0](y,t)Z_i(y)\ddy =0 \quad \text{for}\quad t>t_0, \quad\text{and}\quad i=1,2,3,4.
	\end{align}
	As we shall see in the proof of Lemma \ref{estphi3}, the conditions (\ref{orthM}) are essential to have $\phi_3$ bounded in space and equivalent to choose $\xi_0(t)$.
	The condition corresponding to the index $i=4$ is satisfied by symmetry. When $i=1,2,3$ the orthogonality condition (\ref{orthM}) is equivalent to
	\begin{align*}
		\mu_0 \dot\xi_{0,i} \qty(\int_{\RR^3}\abs{\pp_{y_i} U(y)}^2\ddy 
		)-\mu_0^2\qty(\int_{\RR^3}5U(y)^4 y_i \pp_{y_i}U(y)\ddy )  \frac{1}{2} \pp_{x,i} R_\g(0)=0.
	\end{align*}
	Hence, we select $\xi_{0,i}$ such that
	\begin{align*}
		\dot \xi_{0,i}(t)= \frac{\pp_{x,i} R_\g(0)\qty(\int_{\RR^3}5U(y)^4 y_i \pp_{y_i}U(y)\ddy) }{2\qty(\int_{\RR^3}\abs{\pp_{y_i} U(y)}^2\ddy)} \mu_0(t)
		.
	\end{align*}
	With the condition $\lim\limits_{t\to \infty}\xi_i(t)=0$ we get
	\begin{align}\label{xi0def}
		\xi_{0,i}(t)= \mathfrak{c}_i e^{-2\gamma t},\quad \mathfrak{c}_i= -\frac{\abs*{\pp_{x,i} R_\g(0)}\abs{\int_{\RR^3}5U(y)^4 y_i \pp_{y_i}U(y)\ddy} }{4\gamma \qty(\int_{\RR^3}\abs{\pp_{y_i} U(y)}^2\ddy)}.
	\end{align}
	Also, we define $\mathfrak{c}{\coloneqq}(\mathfrak{c}_1,\mathfrak{c}_2,\mathfrak{c}_3)$.\\
	\begin{remark}[no local improvement in the radial case]
		In case $\Omega=B_1(0)$, searching $h_\g(r,0)$ solution to (\ref{eqhg}) in the radial form, we see that 
		$$
		\nabla_x R_\g(0)=2\nabla_{x_1} h_\g(0,0)=0,
		$$ 
		hence conditions (\ref{orthM}) imply $\xi_0={0}$, as expected. Thus, the local improvement $\phi_3$, which in fact involves only non-zero modes, is null in the radial case. 
	\end{remark}
	With these choices for $\phi_3$ and $\xi_0$ we conclude with the following expression of the error associated to the final ansatz $u_3$:
	\begin{align*}
		S[u_3]=
		&\mu^{-3/2}\nabla_y U(y)\cdot \qty[\dot \xi_1+\qty(1-\mu^{-1}\mu_0) \dot \xi_0]\eta_l\\
		&+5U(y)^4 \qty[\mu^{-3/2}h_\g(x,\xi)-\mu^{-5/2}\mu_0 \qty(\frac{1}{2}\mu_0 y \cdot \nabla_x R_\g(0))]\eta_l\\
		&+\qty[\mu^{-3/2}\nabla_y U(y)\cdot \dot \xi +5U(y)^4 \mu^{-3/2}h_\g(x,\xi)](1-\eta_l)\\
		&+\mu^{1/2}\dot \xi \cdot \nabla_{x_2} H_\g(x,\xi)+\mu^{-5/2}\mathcal{N}_3(y,t)+ 5U(y)^4 \mu^{-3/2} \qty(\frac{\mu_0}{\mu})^{1/2} J(x,t) \\
		&-\qty(-\frac{\dot \mu}{2\mu} )\mu^{-1/2}\phi_3 \eta_{l(t)}-\mu^{-1/2}\eta_{l(t)}\qty[\pp_t \phi_3+\nabla_y \phi_3 \cdot \qty(-\frac{\dot \mu}{\mu}y-\frac{\dot \xi}{\mu})]-\mu^{-1/2}\phi_3 \pp_t \eta\\
		&+ 2\mu^{-3/2}\nabla_y \phi_3 \cdot\frac{y}{\abs{y}} \qty(\frac{\eta'\qty(\abs{z_3})}{\mu l})+\mu^{-1/2}\phi_3 \qty(\frac{2}{\abs{z_3}}\frac{\eta'(\abs{z_3})}{\mu^2 l^2}+\frac{\eta''\qty(\abs{z_3})}{\mu^2 l^2}).
	\end{align*}
	\subsection{Estimate of the inner and outer error}
	For later purpose, we split $S[u_3]$ in inner and outer error. At this stage, it is important to treat the terms involving directly $\dl$ as part of the outer error, since, as we shall see, a priori those are the terms with less regularity. Let
	\begin{align*}
		S[u_3]&=\Sinn+S_{\text{out}}\\
			  &=\Sinn \eta_{R(t)}(y) + (1-\eta_{R(t)}(y))\Sinn +S_{\text{out}},
	\end{align*}
	where we define the inner error
	\begin{align}\label{innS}
		\Sinn{\coloneqq}& \mu^{-3/2} \qty(\frac{\mu_0}{\mu})^{1/2} 5U(y)^4 J(x,t)+\mu^{-5/2}\mathcal{N}_3 \\\nonumber
		&+\mu^{-3/2}\eta_{l}\qty(\dot \xi_1+\qty(1-\mu^{-1}\mu_0) \dot \xi_0)\cdot \nabla_y U(y)\\\nonumber
		&+\mu^{-3/2}\eta_l 5U(y)^4 \qty(h_\g(x,\xi)-\qty(\frac{\mu_0}{\mu})\qty(\frac{1}{2}\mu_0 y\cdot \nabla_x R_\g(0)))	,
	\end{align}
	the outer error
	\begin{align}\label{outS}
		\Sout{\coloneqq}
		&\mu^{-3/2}\qty[\nabla_y U(y)\cdot \dot \xi +5U(y)^4 h_\g(x,\xi)](1-\eta_l)\\\nonumber
		&+\mu^{1/2}\dot \xi \cdot \nabla_{x_2} H_\g(x,\xi)\\\nonumber
		&-\mu^{-1/2}\qty[(\g-\dl)\eta_l \qty(\phi_3+2y \cdot \nabla_y \phi_3)+\eta_l \qty(\pp_t \phi_3 -\mu^{-1}\dot \xi \cdot \nabla_y \phi_3)+				
		\phi_3 \frac{\eta'(\abs{z_3})}{\mu l}\dot \xi\cdot \frac{z_3}{\abs{z_3}}]\\\nonumber
		&+ 2\mu^{-3/2}\nabla_y \phi_3 \cdot\frac{y}{\abs{y}} \qty(\frac{\eta'\qty(\abs{z_3})}{\mu l})+\mu^{-1/2}\phi_3 \qty(\frac{2}{\abs{z_3}}\frac{\eta'(\abs{z_3})}{\mu^2 l^2}+\frac{\eta''\qty(\abs{z_3})}{\mu^2 l^2}),
	\end{align}
	and the radius
	\begin{align}\label{Radius}
		R(t):=\mu(t)^{-\delta},
	\end{align}
	for some constant $\delta>0$ which will be chosen in (\ref{deltal1}) to make both the errors $\Sinn \eta_R$ and $\Sinn (1-\eta_R)+\Sout$ suitably small for a final contraction. 
	
	\textbf{Size of $\Sinn \eta_{R}$}. We proceed with the estimate of $\Sinn \eta_{R}$. 
	More precisely, we need the following conditions on $\delta,l_0,l_1,k$:
	\begin{align}\label{condouter1}
		&\delta+l_1<1\\\label{condcontractionNLandHolder}
		&\delta \in \qty(\frac{1-l_1}{2},\frac{1+l_1}{6})\\ \label{l1minl0}
		&l_1\leq l_0,\\\label{condk}
		&k+1\geq 2\delta+l_1,
	\end{align}
	The condition (\ref{condouter1}) is used to get the estimate in the linear outer problem, and it is due to the fact that both the heat kernel $p_t^\Om$ and the parameter $\mu_0(t)$ have an exponential decay for $t$ large. To make the quadratic term $U^3 \tilde \phi^2$ smaller than $\Sinn$ in the inner problem we need the upper bound in (\ref{condcontractionNLandHolder}). The lower bound is necessary to get a positive H\"older exponent in the regularity of $\dot \Lambda$. The last two conditions (\ref{l1minl0})-(\ref{condk}) insure that $\Sinn$ is controlled by the first term in (\ref{innS}). Thus, we fix the following values satisfying (\ref{condouter1})-(\ref{condcontractionNLandHolder}):
	\begin{align}\label{deltal1}
		\delta=\frac{2}{9},\quad l_1=\frac{2}{3}.
	\end{align}
	Here and in what follows, we write $a\lesssim b$ if there exists a constant $C$, independent of $t_0$, such that $a\leq Cb$. If both the inequalities $a \lesssim b$ and $b\lesssim a$ hold we write $a \sim b$. Using (\ref{estJ_{1,0}-Linf}) and (\ref{estonJ2}) we estimate
	\begin{align*}
		\abs{\eta_R 5U(y)^4 \qty(\frac{\mu_0}{\mu})^{1/2}\mu^{-3/2}J[\dot \Lambda](x,t)}&\lesssim \frac{\mu^{-3/2}}{1+\abs{y}^4}\qty(\mu^{l_1}+\frac{\mu}{1+\abs{y}^{1-\varepsilon}}),\\
		&\lesssim \frac{\mu^{-3/2+l_1}}{1+\abs{y}^4}
	\end{align*}
	and, since we are in the region where $\eta_R \neq 0$, using (\ref{bphi3}), we obtain
	\begin{align*}
		\abs{\eta_R{\mu^{-5/2}\mathcal{N}_3}}&\lesssim \mu^{-1/2}U(y)^3\qty(\abs{H_\g(x,\xi)}+\abs{J(x,t)}+\mu^{-1}\phi_3 \eta_l)^2\\
		&\lesssim \frac{\mu^{-1/2}}{1+\abs{y}^3}\qty(\mu R+\mu^{l_1}+\frac{\mu}{1+\abs{y}^{1-\varepsilon}}+\mu)^2\\
		&\lesssim \frac{\mu^{-1/2}}{1+\abs{y}^{4}}\mu^{-\delta}\qty(\mu^{1-\delta}+\mu^{l_1})^2\\
		&\lesssim \frac{\mu^{-1/2-\delta +2\min\{1-\delta,l_1\}}}{1+\abs{y}^{4}}
	\end{align*}
	Also,
	\begin{align*}
		\abs{\eta_R\mu^{-3/2}\eta_l \qty(\dot \xi_1+(1-\mu^{-1}\mu_0)\dot \xi_0)\cdot \nabla_y U}
		&\lesssim \eta_l \frac{\mu^{-3/2}}{1+\abs{y}^{4}}R^{2}\qty(\abs*{\dot \xi_1}+\mu^2)\\
		&\lesssim  \frac{\mu^{-3/2-2\delta  +\min\{1+k,2\}}}{1+\abs{y}^{4}}
	\end{align*}
	Now, we estimate the last term of $\Sinn\eta_R$ using expansion (\ref{exphgg}) and $\mu/\mu_0=e^{2\Lambda}$ we get
	\begin{align*}
		\abs{\eta_R\mu^{-3/2}\eta_l U(y)^4 \qty(h_\g(x,\xi)-\qty(\frac{\mu_0}{\mu})\qty(\frac{1}{2}\mu_0 y \cdot \nabla_x R_\g(0)))}\lesssim \frac{\mu^{-3/2+\min\{1,l_0\}}}{1+\abs{y}^4}.
	\end{align*} 
	Combining these estimates we obtain
	\begin{align*}
		\abs{\eta_R\Sinn}\lesssim \frac{1}{1+\abs{y}^4}\qty[\mu^{-1/2-\delta+2\min\{1-\delta,l_1\}}+\mu^{-3/2-2\delta +\min\{1+k,2\}}+\mu^{-3/2+\min\{1,l_0,l_1\}}],
	\end{align*}
	and using the values (\ref{deltal1}) we get
	\begin{align*}
		\abs{\eta_R\Sinn}&\lesssim \frac{\mu^{-\frac{3}{2}+l_1}}{1+\abs{y}^4}\lesssim \frac{\mu^{-5/6}}{1+\abs{y}^4}
	\end{align*} 
	\textbf{Size of $\Sout$.} For the first term in $\Sout$ we have
	\begin{align*}
		&\abs{(1-\eta_l)\nabla_y U \cdot \dot \xi} \lesssim \mu^{3/2}(1-\eta_l)\\
		&\abs{\mu^{-3/2}5U^4 h_\g (1-\eta_l)}\lesssim \mu^{5/2}(1-\eta_l)\\
		&\abs{\mu^{1/2}\dot \xi \cdot \nabla_{x_1} H_\g} \lesssim \mu^{3/2}
	\end{align*}
	and using the estimates given by Lemma \ref{estphi3} on $\phi_3,\nabla_y \phi_3$ and $\pp_t \phi_3$ we get
	\begin{align*}
		\bigg\lvert \mu^{-1/2}\bigg[&(\g-\dot \Lambda) \eta_l \qty(\phi_3+2y \cdot \nabla_y \phi_3)\\
		&+\eta_l \qty(\pp_t \phi_3 -\mu^{-1}\dot \xi \cdot \nabla_y \phi_3)+				
		\phi_3 \frac{\eta'(\abs{z_3})}{\mu l}\dot \xi\cdot \frac{z_3}{\abs{z_3}}\bigg] \bigg\rvert\lesssim
		\mu^{3/2}
	\end{align*}
	Finally,
	\begin{align*}
		\abs{2\mu^{-3/2}\nabla_y \phi_3 \cdot\frac{y}{\abs{y}} \qty(\frac{\eta'\qty(\abs{z_3})}{\mu l})+\mu^{-1/2}\phi_3 \qty(\frac{2}{\abs{z_3}}\frac{\eta'(\abs{z_3})}{\mu^2 l^2}+\frac{\eta''\qty(\abs{z_3})}{\mu^2 l^2})}\lesssim \mu^{3/2}.
	\end{align*}
	We conclude that
	\begin{align*}
		\abs{\Sout}\lesssim \mu^{\frac{3}{2}}.
	\end{align*}
	\textbf{Size of $\Sinn(1-\eta_R)$.} It remains to estimate the size of $\Sinn(1-\eta_R)$. We have
	\begin{align}\label{mterm}
		\abs{(1-\eta_R)5U^4 \mu^{-3/2} J[\dot \Lambda](x,t)}&\lesssim \frac{\mu^{-\frac{3}{2}+l_1+2\delta}}{1+\abs{y}^2}(1-\eta_{R})
	\end{align}
	Then,
	\begin{align*}
		\abs{(1-\eta_R)\mu^{-5/2}\mathcal{N}_3}
		&\lesssim (1-\eta_R)\mu^{-1/2}\frac{1}{1+\abs{y}^3}\qty(\abs{H_\g(x,\xi)}+\abs{J(x,t)}+\mu^{-1}\phi_3 \eta_l)^2\\
		&\lesssim (1-\eta_R)\frac{\mu^{-1/2}R^{-1}}{1+\abs{y}^2}\qty(1+\mu^{2l_1}+\mu^2)\\
		&\lesssim (1-\eta_R)\frac{1}{1+\abs{y}^2}\mu^{-1/2+\delta}.
	\end{align*}
	In particular we observe that this is smaller than (\ref{mterm}), thanks to (\ref{condcontractionNLandHolder}).
	Also,
	\begin{align*}
		\mu^{-3/2}\eta_l (\dot \xi_1+(1-\mu^{-1}\mu_0)\dot \xi_0)\cdot \nabla_y U(y)(1-\eta_R)
		&\lesssim \frac{\mu^{\frac{1}{2}+\min\{0,k-\frac{1}{2}\}}}{1+\abs{y}^2}(1-\eta_R),
	\end{align*}
	and
	\begin{align*}
		\abs*{(1-\eta_R)\eta_l 5U(y)^4 \mu^{-3/2}\qty(h_\g(x,\xi)-\qty(\frac{\mu_0}{\mu})\qty(\frac{1}{2}\mu_0 y\cdot \nabla_x R_\g(0)))
		}&\lesssim \frac{R^{-2}\mu^{-1/2}}{1+\abs{y}^2}(1-\eta_R)\\
		&\lesssim \frac{\mu^{2\delta-\frac{1}{2}}}{1+\abs{y}^2}(1-\eta_R).
	\end{align*}
	Combining these estimates we find
	\begin{align*}
		\abs{\Sinn (1-\eta_R)}&\lesssim \mu^{-\frac{3}{2}+l_1+2\delta}\frac{1}{1+\abs{y}^2}(1-\eta_R(\abs{y})).
	\end{align*}
	We conclude that
	\begin{align*}
		\abs{\Sinn(1-\eta_R)+\Sout}\lesssim \mu^{-\frac{3}{2}+l_1+2\delta}\frac{1}{1+\abs{y}^2}(1-\eta_R)+\mu^{\frac{3}{2}}.
	\end{align*}
	\subsection{Estimates of $J_1,J_2$ and $\phi_3$}
	The following lemma gives an estimate of $J_1[\dot \Lambda](x,t)$ in terms of $\dot \Lambda$. Observe that
	\begin{align*}
		\lim\limits_{t\to \infty}\qty(\frac{\mu(t)}{\mu_0(t)})^{1/2}\qty[\frac{\mu^2(t)-\abs{x-\xi(t)}^2}{\qty(\mu(t)^2+\abs{x-\xi(t)}^2)^{3/2}}+H_\g(x,\xi(t))]=-\frac{1}{\abs{x}}+H_\g(x,0),
	\end{align*}
	thus, for $t_0$ large, we will approximate $J_1$ with $\JJJ$, that is the solution to 
	\begin{align}\label{ProbJ10}
		&\pp_t \JJJ= \Delta_x \JJJ+\gamma \JJJ -\dot \Lambda(t)G_\g(x,0)\quad \inn \Omega \times [t_0-1,\infty),\\\nonumber
		&\JJJ(x,t)=0 \onn \pp \Omega \times [t_0-1,\infty),\\\nonumber
		&\JJJ(x,t_0-1)=0 \inn \Omega.
	\end{align}
	We define the $L^\infty$-weighted space
	\begin{align*}
		X_c{\coloneqq}\{	f  \in L^\infty(t_0-1,\infty): \norm{f}_{\infty,c}<\infty	\},
	\end{align*}  
	where
	\begin{align*}
		\norm{f}_{\infty,c}{\coloneqq}\sup_{t>t_0-1}\abs{f(t)\mu_0(t)^{-c}}.
	\end{align*}
	
	\begin{lemma}[Estimate of $J_1$]\label{EstJ1fromdotLambda}
		Suppose $2 \g l_1<\lambda_1-\g$ and 
		\begin{align*}
			\norm*{\dot \Lambda}_{\infty,l_1}< \infty.
		\end{align*}
		Then we have
		\begin{align}\label{estJ_{1,0}-Linf}
			\norm{J_1(\cdot,t)}_{L^\infty(\Om)}\lesssim \mu_0(t)^{l_1}\norm*{\dot \Lambda}_{\infty,l_1},
		\end{align} 
		for $t\geq t_0$.
	\end{lemma}
	Since we have selected $l_1<1$ in (\ref{deltal1}), condition (\ref{Assumption1}) guarantees that $2\g l_1<\lambda_1-\g$.
	\begin{proof}
		By parabolic comparison, it is enough to prove the bound for $\JJJ$ defined as the solution to (\ref{ProbJ10}).
		Indeed, we have
		\begin{align*}
			\abs{\qty(\frac{\mu}{\mu_0})^{1/2}\dot \Lambda(t)\qty(\frac{\mu^2-\abs{x-\xi}^2}{\qty(\mu^2+\abs{x-\xi}^2)^{3/2}}+H_\g(x,\xi))}\lesssim \abs*{\dot \Lambda(t)}\abs{-\frac{1}{\abs{x}}+H_\g(x,0)}
		\end{align*}
		We decompose 
		\begin{align*}
			\JJJ(x,t)=\sum_{k=1}^{\infty} b_k(t)w_k(x) \inn L^2(\Omega), \quad \text{for}\quad t\geq t_0-1,
		\end{align*}
		where $w_k$ is the $k$-th eigenfunction of $-\Delta$ on $\Om$. Plugging the decomposition into the equation we find
		\begin{align*}
			b_k=c_k \int_{t_0-1}^{t} e^{-\qty(\lambda_k-\gamma)(t-s)}\dot \Lambda(s) \dds,\quad \text{where}\quad c_k{\coloneqq}-\int_\Omega G_\g(x,0) w_k(x)\ddx.
		\end{align*}
		In particular, we have
		\begin{align*}
			\norm{\JJJ(\cdot,t)}_{L^2(\Omega)} \leq  \norm{G_\g(\cdot,0)}_{L^2(\Omega)} \int_{t_0-1}^{t} e^{-(\lambda_1-\gamma)(t-s)}\abs*{\dot \Lambda(s)}\dds.
		\end{align*}
		Using $\norm*{\dot \Lambda}_{\infty,l_1}<\infty$ and $2\g l_1<\lambda_1-\g$ we obtain
		\begin{align*}
			\norm{\JJJ(\cdot,t)}_{L^2(\Omega)}&\leq \norm{G_\g(\cdot,0)}_{L^2(\Omega)} \norm*{\dot \Lambda}_{\infty,l_1} e^{-\min\{2\g l_1,\lambda_1-\g\} (t-(t_0-1))},\\
			&\lesssim \norm*{\dot \Lambda}_{\infty,{l_1}} e^{-2\g l_1 (t-(t_0-1))}
		\end{align*}
		Finally, from standard parabolic estimates, using the $L^2$-bound and equation (\ref{ProbJ10}), we get for $t\geq t_0$
		\begin{align*}
			\norm{\JJJ(\cdot,t)}_{L^\infty(\Omega')}\lesssim \norm*{\dot \Lambda}_{\infty,{l_1}} e^{-2\g l_1 t},
		\end{align*}
		for any $\Omega' \Subset \Omega$. By boundary regularity estimates this inequality can be extended to $\Omega$ thanks to the smoothness of $\pp \Om$. 
	\end{proof}
	
	\begin{lemma}[Estimate of $J_2$]\label{Lemma:estJ2}
		Let $J_2(x,t)$ be the unique solution to the problem
		\begin{align*}
			\pp_t J_2 =& \Delta_x J_2+ \gamma J_2-\qty(\frac{\mu}{\mu_0})^{\frac{1}{2}}\bigg[\gamma \qty(\mu^{-1}2Z_4\qty(\frac{x-\xi}{\mu})+\frac{\alpha_3}{\abs{x-\xi}})\\
			&+\mu^{-2}5U\qty(\frac{x-\xi}{\mu})^4 \theta_\g(x-\xi)\bigg] \inn \Omega \times [t_0,\infty),\\\nonumber
			J_2(x,t)&=0 \onn \pp \Omega \times [t_0,\infty),\\\nonumber
			J_2(x,t_0&)=0 \inn \Omega.
		\end{align*} 
		Suppose that $3\g<\lambda_1$. Then, there exists $t_0$ large such that 
		\begin{align}\label{estonJ2}
			\abs{J_2(\mu y+\xi,t)}\lesssim \mu(t)\frac{1}{1+\abs{y}^{1-\ve}},
		\end{align}	 
		for any $\ve>0$ and for all $(x,t)\in \Om \times [t_0,\infty)$ where $y=(x-\xi)/\mu$.
	\end{lemma}	
	\begin{proof}
		Firstly, we observe that
		\begin{align*}
			\abs{ \frac{1-\abs{y}^2}{(1+\abs{y}^2)^{3/2}}+\frac{1}{\abs{y}}}\lesssim \frac{1}{\abs{y}\qty(1+\abs{y}^{2-\ve})}.
		\end{align*}
		Also, by Taylor expanding the function $\theta_\g$ in (\ref{defthetag}) near the origin, we see that
		\begin{align*}
			\abs{\mu^{-2}5U(y)^4\theta_\g(\mu y)}&\lesssim \frac{\mu^{-1}}{1+\abs{y}^4} \abs{y}\\
			%&\lesssim \frac{\mu^{-1}}{\abs{y}\qty(1+\abs{y}^2)}\\
			&\lesssim \frac{\mu^{-1}}{\abs{y}\qty(1+\abs{y}^{2-\epsilon})},
		\end{align*}
		where $\epsilon>0$ can be taken arbitrarily small.  
		Thus, by parabolic comparison, it is enough to find a supersolution to the problem
		\begin{align*}
			&\pp_t u = \Delta_x u + \g u + \mu^{-1} \frac{1}{\abs{y(x,t)}(1+\abs{y(x,t)}^{2-\ve})} \inn \cyt,\\
			&u(x,t)=0 \onn \pp \cyt,\\
			&u(x,t_0)=0 \inn \Omega.
		\end{align*}	
		Let $v(x,t){\coloneqq} \mu(t)^{-1} u(x,t)$, which satisfies
		\begin{align*}%\label{prJ2vv}
			&\pp_t v %= \pp_t (\mu^{-1}u)
			= \Delta_x v+ (3\g -2\dl)v + \frac{\mu^{-2}}{\abs{y(x,t)}\qty(1+\abs{y(x,t)}^{2-\ve})} \inn \cyt,\\\nonumber
			&v= 0 \onn \pp \Omega \times [t_0,\infty),\\\nonumber
			&v(x,t_0)= 0 \inn \Omega.
		\end{align*}
		We look for a supersolution $\bar v$ of the form
		\begin{align*}
			\bar v(x,t)=v_0\qty(\frac{x-\xi}{\mu})\eta\qty(\frac{x-\xi}{C_0})+v_1(x,t).
		\end{align*}
		We need
		\begin{align}\label{supersolutionv1v2}
			\pp_t v_1-\Delta_x v_1-(3\g-\dot \Lambda)v_1\geq \eta \bigg[&	-\pp_t v_0+\mu^{-2}\Delta_y v_0+(3\g-\dot \Lambda)v_0\\\nonumber
			&+\frac{\mu^{-2}}{\abs{y}(1+\abs{y}^{2-\ve})}	\bigg]\\\nonumber
			&+(1-\eta)\frac{\mu^{-2}}{\abs{y}(1+\abs{y}^{2-\ve})}+\qty(\Delta_x \eta-\pp_t \eta)v_0\\\nonumber
			&+2\mu^{-1}\nabla_x \eta \cdot \nabla_y v_0,
		\end{align}
		with $v_1\geq0$ on $\pp \Om \times [t_0,\infty)$ and $v_0(y(x,t_0))\geq 0$ for $x\in \Om$. Without loss of generality let $\Omega \subset B_1$. Consider the positive radial solution $v_0(\abs{y},t)$ to
		\begin{align*}
			&\Delta_y v_0 +2\frac{1}{\abs{y}\qty(1+\abs{y}^{2-\ve})}=0 \onn B_{\frac{1}{\mu(t)}},\\
			&v_0\equiv 0 \onn \pp B_{\frac{1}{\mu(t)}},
		\end{align*}
		given by the formula of variation of parameters
		\begin{align*}%\label{formulav0}
			v_0(\abs{y},t)&=2\omega_3\int_{\abs{y}}^{\frac{1}{\mu(t)}} \frac{1}{\rho^2} \int_{0}^{\rho}\frac{s}{1+s^{2-\ve}}\dds \, d\rho.
		\end{align*}
		From this formula we obtain the following estimates in $(x,t)\in\Om \times [t_0,\infty)$:
		\begin{align*}
			&\abs{v_0(\abs{y},t)}+\abs{\pp_t v_0(\abs{y},t)}\lesssim \frac{1}{1+\abs{y}^{1-\ve}},
		\end{align*}
		Thus, if $\abs{x-\xi}<C_0$, for $C_0$ sufficiently small, then
		\begin{align*}
			-\pp_t v_0+&\mu^{-2}\Delta_y v_0+(3\g-\dot \Lambda)v_0+\frac{\mu^{-2}}{\ay (1+\ay^{2-\ve})}
			\\&=-\frac{\mu^{-2}}{\ay (1+\ay^{2-\ve})}+O\qty(\frac{1}{1+\abs{y}^{1-\ve}})\leq 0.
		\end{align*}
		Then, let $v_1$ be the solution to
		\begin{align*}
			\pp_t v_1-\Delta_x v_1-(3\g-\dot \Lambda)v_1 =&(1-\eta)\frac{\mu^{-2}}{\abs{y}(1+\abs{y}^{2-\ve})}+(\Delta_x \eta-\pp_t \eta)v_0\\&+2\mu^{-1}\nabla_x \eta \cdot \nabla_y v_0 \inn \cyt,
		\end{align*}
		with
		\begin{align*}
			&v_1(x,t)=0 \onn \pp \Om \times [t_0,\infty),\\
			&v_1(x,t_0)=0 \inn \Om.
		\end{align*}
		In the right-hand side we have
		\begin{align*}
			&(1-\eta)\frac{\mu^{-2}}{\abs{y}(1+\abs{y}^{2-\ve})}\lesssim \mu^{1-\ve},
			\\
			&\abs{(\Delta_x-\pp_t \eta)v_0}\lesssim \mu^{1-\ve},\\
			&\abs{2\mu^{-1}\nabla_x \eta \cdot \nabla_y v_0}\lesssim \mu^{1-\ve}.
		\end{align*}
		Since $3\g-\dot \Lambda(t)<\lambda_1$ provided that $t_0$ is sufficiently large, the comparison principle applies and we get $\abs{v_0}\lesssim \mu^{1-\ve}$. Thus, we verified inequality (\ref{supersolutionv1v2}). Also, we have $v=v_1\geq0$ on $\pp \Om \times [t_0,\infty)$ and $\eta v_0(y(x,t_0))\geq 0$. Thus, $v$ is a supersolution, and going back to the original function $u=\mu v$ we get estimate (\ref{estonJ2}) for $J_2$. 
	\end{proof}
	
	\begin{lemma}[Estimate on $\phi_3$]\label{estphi3}
		Let $\mathcal{M}[\xi_0,\mu_0]$ be defined as in (\ref{defM}). If the orthogonality conditions (\ref{orthM}) on $\mathcal{M}[\xi_0,\mu_0]$ hold, then there exists a bounded solution to the problem 
		\begin{align}\label{eqphi3}
			\Delta_y \phi_3 + 5U(y)^4 \phi_3(y,t)=-\mathcal{M}[\xi_0,\mu_0](y,t) \inn \RR^3.
		\end{align}
		We have the following estimates on $\phi_3$ and its derivatives:
		\begin{align}\label{bphi3}
			\abs{\phi_3(y,t)}+(1+\abs{y})\abs{\nabla_y \phi_3(y,t)}+\abs{\pp_t \phi_3(y,t)} \lesssim \mu^2(t) f(y,t),
		\end{align}
		where $f$ is a smooth bounded function.
	\end{lemma}
	\begin{proof}
		From the explicit form of the function $\mathcal{M}$ given in (\ref{defM}) we estimate its size by
		\begin{align*}
			\abs{\mathcal{M}[\mu_0,\xi_0](y,t)}
			%&\leq \mu^2 \frac{1}{(1+\abs{y}^2)}+\mu^4 +\frac{\mu^2}{(1+\abs{y}^3)}\\
			&\leq \mu^2 \frac{1}{1+\abs{y}^2},
		\end{align*}
		and we observe that $\mathcal{M}$ has only modes $i=1,2,3$. Thus, we decompose $\phi_3$ in such modes:
		\begin{align*}
			\phi_3(y)=\sum_{i=1}^3 \phi_{3,i}(r) \vartheta_i(y/r),\quad r{\coloneqq}\abs{y},\quad \phi_{3,i}(r){\coloneqq}\int_{S^2}\phi_3(r\theta)\vartheta_i(\theta)\, d\theta.
		\end{align*}
		Similarly, we define $$z_i(r){\coloneqq}\int_{S^2}Z_i(r\theta)\vartheta_i(\theta) \,d\theta.$$ 
		The formula of variation of constants gives
		\begin{align*}
			\phi_{3,i}(r)=z_i(r)\int_0^r \frac{1}{\rho^2 z_i(\rho)^2}\mathcal{I}_i(\rho) \,d\rho,
		\end{align*}
		where
		\begin{align*}
			\mathcal{I}_i(\rho){\coloneqq}\int_{0}^{\rho} M_i(s)z_i(s) s^2 \dds,
		\end{align*}
		and 
		$$
		M_i(r){\coloneqq}\int_{S^2}\mathcal{M}(r\theta)\vartheta_i(\theta) \,d\theta.
		$$ 
		Since 
		$$
		\abs{M_i(r)}\lesssim \frac{1}{1+r^2},
		$$ 
		and 
		$$
		\abs{z_i(r)}\lesssim \frac{r}{(1+r^3)},
		$$ 
		we deduce
		\begin{align*}
			\abs{\mathcal{I}_i(\rho)}\lesssim \rho^4 \ass \rho \to 0.
		\end{align*}
		Also, by the orthogonality conditions (\ref{orthM}) we have
		\begin{align*}
			\abs{\mathcal{I}_i(\rho)}&=\abs{\int_{\rho}^{\infty} M_i(s)z_i(s)s^2 \dds }\\
			&\lesssim \frac{1}{\rho} \ass \rho \to \infty.
		\end{align*}
		With these estimates we conclude 
		\begin{align*}
			\abs{\phi_3(r)}&\lesssim \frac{r}{1+r^3} \int_{0}^{r}\frac{\qty(1+\rho^2)^3}{\rho^4}\abs{\mathcal{I}(\rho)}\,d\rho\\
			&\lesssim \frac{r}{1+r^3} \int_{0}^{r}\frac{\qty(1+\rho^2)^3}{\rho^4} \frac{\rho^4}{1+\rho^5}\,d\rho\\
			&\lesssim 1.
		\end{align*}
		Similarly, taking the space and time derivatives of equation (\ref{eqphi3}), we deduce the bounds on $\nabla_y \phi_3$ and $\pp_t \phi_3$. 
	\end{proof}
	We conclude this section with summarizing the estimates of the error $S[u_3]$.
	\begin{lemma}\label{Lemma:FinalAnsatzError}
		Let $3\g<\lambda_1$, $\mu=\mu_0 e^{2\Lambda}$ and $\xi=\xi_0+\xi_1$, where $\mu_0,\xi_0$ are given by (\ref{mu0def}) and (\ref{xi0def}) respectively. Assume 
		\begin{align*}
			&\abs{\Lambda(t)}\lesssim \mu_0(t)^{l_0}, \quad \abs{\dot \Lambda(t)}\lesssim \mu_0(t)^{l_1},\\&R(t)=\mu^{-\delta},\quad \abs{\dot \xi_1(t)}\lesssim \mu_0^{1+k},
		\end{align*}
		for positive constant $\delta,l_0,l_1,k$ satisfying (\ref{condouter1}),(\ref{condcontractionNLandHolder}), (\ref{l1minl0}) and (\ref{condk}). Then, setting $x=\mu y+\xi$, for $t_0$ sufficiently large the following estimate on the error function $S[u_3]$ holds:
		\begin{align*}
			S[u_3](y,t)=\Sinn(y,t) \eta_{R(t)}\qty(\abs{y}) + \Sinn(y,t) (1-\eta_{R(t)}(\abs{y}))+\Sout(y,t),
		\end{align*}
		where
		\begin{align*}
			&\abs{\Sinn(y,t)\eta_{R(t)}}\lesssim \mu^{-\frac{3}{2}+l_1}\frac{1}{1+\abs{y}^4},\\%\lesssim \mu^{-\frac{5}{6}}\frac{1}{1+\abs{y}^4}\\
			&\abs{\Sout(y,t)}\lesssim \mu^{\frac{3}{2}},\\
			&\abs{\Sinn(y,t)(1-\eta_{R(t)})}\lesssim \mu^{-\frac{3}{2}+l_1+2\delta}\frac{1}{1+\abs{y}^2}.%(1-\eta_R) %  \lesssim (1-\eta_R)\mu^{-\frac{7}{18}}\frac{1}{1+\abs{y}^2}.
		\end{align*}
	\end{lemma}
	
	\section{The Inner-outer scheme}\label{sec:innout}
	\medskip 
	We recall that our final purpose is to find an unbounded global in time solution $u$ to (\ref{eqCHE}) of the form
	\begin{align}\label{exactSol}
		u=u_3+\tilde\phi,
	\end{align}
	for a small perturbation $\tilde\phi$. The latter is constructed by means of the inner-outer gluing method. This consists in looking for a perturbation of the form
	\begin{align}\label{deftildephi}
		\tilde \phi(x,t)=\mu_0(t)^{1/2}\psi(x,t)+\eta_{R(t)}\qty(\abs{y})\mu(t)^{-1/2}\phi\qty(y,t),
	\end{align}
	where
	\begin{align*}%\label{defetaR}
		\eta_{R(t)}\qty(\abs{y})=\qty(\frac{\abs{y}}{R(t)}),\quad y{\coloneqq}y(x,t){\coloneqq}\frac{x-\xi(t)}{\mu(t)},
	\end{align*}
	and $\eta(s)$ is a cut-off function with $\supp \eta\subset [0,2]$ and $\eta\equiv 1$ in $[0,1]$. We have already chosen $R=R(t)$ in (\ref{Radius}).
	In terms of $\tilde \phi$ the equation reads as 
	\begin{align*}
		0=S[u]&=-\pp_t u+ \Delta_x u +u^5\\
		&= \qty(-\pp_t u_3+\Delta_x u_3+u_3^5)-\pp_t \tilde \phi +\Delta_x \tilde \phi +(u_3+\tilde \phi)^4-u_3^5\\
		&=S[u_3]-\pp_t\tilde \phi+\Delta_x \tilde \phi + 5u_3^4 \tilde \phi +\mathcal{N}(u_3,\tilde \phi)
	\end{align*}
	where 
	\begin{align}\label{DefNLterm}
		\mathcal{N}(u_3,\tilde \phi){\coloneqq}(u_3+\tilde \phi)^5-u_3^5 -5u_3^4 \tilde \phi.
	\end{align}
	Hence the problem for $\tilde \phi$ is
	\begin{align*}
		&\pp_t \tilde \phi= \Delta_x \tilde \phi+ 5u_3^4 \tilde \phi +S[u_3] +\mathcal{N}(u_3,\tilde{\phi}) \inn \Omega \times [t_0,\infty),\\
		&\tilde \phi = -u_3 \onn \pp \Omega \times [t_0,\infty).
	\end{align*}
	Now, the main idea is to split the problem for $\tilde \phi$ in a system for $(\psi,\phi)$, localizing the inner regime. 
	We divide the error in
	\begin{align*}
		S[u_3]=&S_{\text{in}}\eta_R+S_{\text{in}}(1-\eta_R)+S_{\text{out}},
	\end{align*}
	where $\Sinn, S_{\text{out}}$ are defined in (\ref{innS}) and (\ref{outS}) respectively.
	Considering $\tilde \phi$ as in (\ref{deftildephi}) we compute
	\begin{align*}
		\pp_t \tilde \phi  =& \frac{\dot \mu_0}{2\mu_0}\mu_0^{1/2}\psi +\mu_0^{1/2}\pp_t \psi  + \mu^{-1/2} \phi \pp_t \eta\qty(\frac{y(x,t)}{R(t)})-\frac{\dot \mu}{2\mu}\mu^{-1/2}\phi \eta_R \\
		&+\mu^{-1/2}\qty(\pp_t \phi+ \nabla_y \phi \cdot \pp_t y(x,t))\eta_R\\
		=&-\g\mu_0^{1/2}\psi +\mu_0^{1/2}\pp_t \psi + \mu^{-1/2} \phi \qty[ \nabla_z \eta \qty(\frac{y}{R})\cdot \qty(-\frac{\dot R}{R}\frac{y}{R}-\frac{\dot \mu}{\mu}\frac{y}{R}-\frac{\dot \xi}{\mu R})]\\
		&+\qty(-\frac{\dot \mu}{2\mu})\mu^{-1/2}\phi \eta_R +\mu^{-1/2}\eta_R \qty(\pp_t \phi+ \nabla_y \phi\cdot \qty(-\frac{\dot \mu}{\mu}y-\frac{\dot \xi}{\mu})),
	\end{align*}
	and 
	\begin{align*}
		\Delta_x \tilde \phi =& \mu_0^{1/2}\Delta_x \psi + \mu^{-1/2}\Delta_x \qty(\phi(y(x,t),t)\eta_{R(t)}(y(x,t)))\\
		=&\mu_0^{1/2}\Delta_x \psi + \mu^{-5/2}\eta_R(y)\Delta_y \phi(y,t)  + \mu^{-1/2}\phi\qty(\frac{2}{\abs{z}}\frac{\eta'(\abs{z})}{\mu^2 R^2}+\frac{\eta''\qty(\abs{z})}{\mu^2 R^2})\\
		&+ 2\mu^{-1/2}\frac{1}{\mu}\nabla_y \phi(y,t) \cdot \frac{z}{\abs{z}} \frac{\eta'(\abs{z})}{\mu R},
	\end{align*}
	where $z{\coloneqq}y/R$. We split
	\begin{align*}
		5u_3^4 \tilde \phi= 5u_3^4 \mu_0^{1/2}\psi \eta_R + 5 u_3^4 \mu_0^{1/2}\psi(1-\eta_R)+5u_3^4 \mu^{-1/2} \phi \eta_R.
	\end{align*}
	Hence, the full equation becomes
	\begin{align*}
		-\g \mu_0^{1/2}\psi +\mu_0^{1/2}\pp_t \psi &+\mu^{-1/2}\phi \pp_t \eta_R +\eta_R \mu^{-1/2}\pp_t \phi 
		\\&+\eta_R\qty{(\g-\dot \Lambda)\mu^{-1/2}(\phi+2 \nabla_y \phi \cdot y)  -\mu^{-1/2}\nabla_y \phi\cdot \qty(\frac{\dot \xi}{\mu})}\\
		=&\mu_0^{1/2}\Delta_x \psi +\mu^{-5/2}\eta_R\Delta_y \phi +\mu^{-1/2}\phi \qty(\frac{2}{\abs{z}}\frac{\eta'(\abs{z})}{\mu^2 R^2}+\frac{\eta''\qty(\abs{z})}{\mu^2 R^2}) \\
		&+2\mu^{-1/2}\frac{1}{\mu}\nabla_y \phi \cdot \frac{z}{\abs{z}}\frac{\eta'(\abs{z})}{\mu R} \\
		&+5u_3^4 \mu_0^{1/2}\psi \eta_R + 5u_3^4 \mu_0^{1/2}\psi (1-\eta_R) + 5u_3^4 \mu^{-1/2}\phi \eta_R\\
		&+S_{\text{in}}\eta_R+S_{\text{in}}\qty(1-\eta_R)+S_{\text{out}}\\
		&+\mathcal{N}(u_3,\tilde \phi)(1-\eta_R) +\mathcal{N}(u_3,\tilde \phi)\eta_R.
	\end{align*}
	We divide the full problem in a system.
	Firstly, we look for a solution $\psi$ to
	\begin{align*}
		\mu_0^{1/2}\pp_t \psi =&\mu_0^{1/2}\Delta_x \psi+\g \mu_0^{1/2}\psi+5u_3^4 \mu_0^{1/2}\psi (1-\eta_R)+\mu^{-1/2}\phi \pp_t \eta_R \\\nonumber&
		+ \eta_R\qty{(\g-\dot \Lambda)\mu^{-1/2}(\phi+2 \nabla_y \phi \cdot y)  -\mu^{-1/2}\nabla_y \phi\cdot \qty(\frac{\dot \xi}{\mu})}\\\nonumber
		&+\mu^{-1/2}\phi \qty(\frac{2}{\abs{z}}\frac{\eta'(\abs{z})}{\mu^2 R^2}+\frac{\eta''\qty(\abs{z})}{\mu^2 R^2}) +2\mu^{-1/2}\frac{1}{\mu}\nabla_y \phi \cdot \frac{z}{\abs{z}}\frac{\eta'(\abs{z})}{\mu R} \\\nonumber
		&+S_{\text{in}}\qty(1-\eta_R)+S_{\text{out}}+\mathcal{N}(u_3,\tilde \phi)(1-\eta_R),\quad\quad\quad \inn \Omega \times [t_0,\infty)\\\nonumber
		\psi(x,t)=&-\mu_0^{-1/2}u_3(x,t)\quad \onn \pp \Omega \times [t_0,\infty).
	\end{align*}
	Thus, after dividing by $\mu_0^{1/2}$, $\psi$ solves the \textbf{outer problem} 
	\begin{align}\label{OutPr}
		\pp_t \psi =&\Delta_x \psi+\g \psi+5u_3^4 \psi (1-\eta_R)+\mu^{-1}\qty(\frac{\mu}{\mu_0})^{1/2}\phi \pp_t \eta_R \\\nonumber&
		+ \mu^{-1}\qty(\frac{\mu}{\mu_0})^{1/2}\eta_R\qty{(\g-\dot \Lambda)(\phi+2 \nabla_y \phi \cdot y)  -\nabla_y \phi\cdot \qty(\frac{\dot \xi}{\mu})}\\\nonumber
		&+\mu^{-1}\qty(\frac{\mu}{\mu_0})^{1/2}\qty(\phi \qty(\frac{2}{\abs{z}}\frac{\eta'(\abs{z})}{\mu^2 R^2}+\frac{\eta''\qty(\abs{z})}{\mu^2 R^2}) +2\frac{\nabla_y \phi}{\mu} \cdot \frac{z}{\abs{z}}\frac{\eta'(\abs{z})}{\mu R}) \\\nonumber
		&+\mu_0^{-1/2}S_{\text{in}}\qty(1-\eta_R)+\mu_0^{-1/2}S_{\text{out}}+\mu_0^{-1/2}\mathcal{N}(u_3,\tilde \phi)(1-\eta_R),\quad\quad\quad \inn \Omega \times [t_0,\infty)\\\nonumber
		\psi(x,t)=&-\mu_0^{-1/2}u_3(x,t)\quad \onn \pp \Omega \times [t_0,\infty),
	\end{align}
	Then, $\phi$ has to solve the problem
	\begin{align*}
		&\mu^{-1/2}\pp_t \phi= \mu^{-5/2}\Delta_y \phi + 5u_3^4 \mu^{-1/2}\phi+ 5u_3^4 \mu_0^{1/2}\psi +S_{\text{in}}+\mathcal{N}(u_3,\tilde \phi) \qquad \inn B_{2R}(0)\times [t_0,\infty).
	\end{align*}
	%We emphasize that such equation must be solved in $B_{2R}(0)$ without any boundary conditions.\\
	Equivalently, multiplying by $\mu^{5/2}$, $\phi$ solves
	\begin{align}\label{InnEq}
		\mu^2 \pp_t \phi =&  \Delta_y \phi + 5U^4 \phi+5U^4 \qty(\frac{\mu_{0}}{\mu})^{1/2}\mu\psi(\mu y+\xi,t)+ B_0\qty[\phi+\mu\psi](\mu y+\xi,t)\\\nonumber
		&+\mu^{5/2}S_{\text{in}}(\mu y+\xi,t)+\mathcal{N}(\mu^{1/2}u_3,\mu^{1/2}\tilde \phi)(\mu y+\xi,t) \inn B_{2R}(0)\times [t_0,\infty),
	\end{align}
	where $B_0$ is the linear operator
	\begin{align}\label{DefB0}
		B_0[f]{\coloneqq}5\qty[\qty(U-\mu H_\gamma+\mu J[\dot \Lambda]+\mu^{-1/2}\phi_3(y,t)\eta_3)^4-U^4]f,
	\end{align}
	\subsubsection{General strategy for solving the inner-outer system}
	We now describe the method we use to solve system (\ref{OutPr})-(\ref{InnEq}). 
	Firstly, for fixed parameters $\Lambda,\dot \Lambda,\xi,\dot \xi$ and inner function $\phi$ in suitable weighted spaces, we solve problem (\ref{OutPr}) in $\psi=\psi[\Lambda,\dot \Lambda,\xi,\dot \xi,\phi]$. This is done in \S\ref{sec:out}. 
	We insert such $\psi$ in the inner problem. At this point we need to find $\Lambda,\dot \Lambda,\xi,\dot \xi$ and $\phi$.
	We make the change of variable $t(\tau)$ defined by the ODE
	\begin{align*}
		&\frac{dt(\tau)}{d\,\tau}=\mu^2(t(\tau))\\
		&t(\tau_0)=t_0,
	\end{align*}
	which explicitly gives
	\begin{align}\label{tauoft}
		\tau-\tau_0
		&=\int_{t_0}^t\frac{ds}{\mu(s)^2}\dds\\
		&=\int_{t_0}^t\frac{ds}{\mu_0(s)^2}(1+o(1))\dds\\
		%&=\frac{e^{4\g t}}{4\g }(1+o(1))\\
		&=\frac{1}{4\gamma}\mu_0(t)^{-2}(1+o(1)).
	\end{align}
	Expressing equation (\ref{InnEq}) in the new variables $(y,\tau)$ we get the \textbf{inner problem}
	\begin{align}\label{InnPr}
		\pp_\tau \phi =  \Delta_y \phi+ 5U^4 \phi+H[\phi,\psi,\Lambda,\dot \Lambda,\xi,\dot \xi](y,\tau) \inn B_{2R}\times[\tau_0,\infty),
	\end{align}
	where
	\begin{align}\label{RHSInnPr}
		H[\phi,\psi,\Lambda,\dot \Lambda,\xi,\dot \xi](y,\tau){\coloneqq}
		&5\Uy^4 \mu \qty(\frac{\mu_0}{\mu})^{1/2} \psi(\mu y +\xi,t(\tau))\\\nonumber
		&+B_0\qty[\phi+\mu\psi](\mu y+\xi,t(\tau))+\mu^{5/2}S_{\text{in}}(\mu y+\xi,t(\tau))\\\nonumber
		&+\mathcal{N}(\mu^{1/2}u_3,\mu^{1/2}\tilde \phi)(\mu y+\xi,t(\tau)).
	\end{align}
	Let $Z_0$ be the positive radially symmetric bounded eigenfunction associated to the only negative eigenvalue $\lambda_0$ of the problem 
	\begin{align*}
		-\Delta_y \phi- 5U(y)^4 \phi=\lambda_0\phi \quad \text{for}\quad \phi \in L^\infty(\RR^3).  
	\end{align*}	
	It is known 
	that $\lambda_0$ is simple and 
	\begin{align*}
		Z_0(y)\sim \frac{e^{-\sqrt{\abs{\lambda_0}}\abs{y}}}{\abs{y}}\ass \abs{y}\to \infty.
	\end{align*}
	We solve (\ref{InnPr}) with a multiple of $Z_0(y)$ as initial datum, namely
	\begin{align}\label{InnPrInCond}
		\phi(\tau_0,y)=e_0 Z_0(y) \inn B_{2R},
	\end{align}
	for some constant $e_0=e_0[H]$ to be found. 
	Formally, this initial datum (\ref{InnPrInCond}) allows $\phi$ to remain small along its trajectory. Indeed, multiplying (\ref{InnEq}) by $Z_0$ and integrating we obtain
	\begin{align*}
		\mu^2\pp_t p(t)+\lambda_0 p(t)=q(t),
	\end{align*}
	where
	\begin{align*}
		p(t)\coloneqq\int_{\RR^3} \phi(y,t)Z_0(y) \ddy,\quad q(t)\coloneqq\int_{\RR^3}h(y,t)Z_0\ddy. 
	\end{align*}
	The general solution $p(t)$ is given by
	\begin{align*}
		p(t)=e^{\abs{\lambda_0}\int_0^t \mu(s)^{-2}\dds } \qty(p(t_0)+\int_{t_0}^{t}\mu(s)^{-2} q(s) e^{-\abs{\lambda_0}\mu(s)^{-2}}\dds ).
	\end{align*}
	This shows that in order to get a decaying solution $p(t)$ (and hence $\phi(y,t)$), the following initial conditions should hold:
	\begin{align*}
		p(t_0)= \int_{\RR^3}\phi(y,t_0)Z_0(y)\ddy = -\int_{t_0}^{\infty} \mu(s)^{-2} q(s) e^{-\abs{\lambda_0}\mu(s)^{-2}}\dds.
	\end{align*}
	This argument formally suggests that, to avoid the instability caused by $Z_0$, the small initial value $\phi(y,t_0)$ needs to be constrained along $Z_0$. \\
	Another important observation is that, in order to solve the problem (\ref{InnPr})-(\ref{InnPrInCond}) we need to constrain the right-hand side $H$ to be orthogonal to $\{Z_i\}_{i=1}^4$. Namely we need
	\begin{align}\label{InnPrOrth}
		\int_{B_{2R}} H(y,\tau)Z_i(y)\ddy =0 \quad \text{for} \quad \tau\in [\tau_0,\infty)  \quad \text{and}\quad  i=1,2,3,4.
	\end{align}
	Indeed, the elliptic kernel generated by $\{Z_i\}_{i=1}^4$ is a subset of the kernel of the parabolic operator
	\begin{align*}
		\mu^2 \pp_t \phi = \Delta_y \phi +5 U(y)^4 \phi.
	\end{align*}
	Hence, we expect to have solvability of the inhomogeneous problem (\ref{InnPr}) with suitable space-time decay if the orthogonality conditions (\ref{InnPrOrth}) are satisfied. \\
	As we shall see in \S\ref{sec:Choiceparameters}, condition (\ref{InnPrOrth}) with index $i=4$ is equivalent to a nonlocal problem in $\Lambda$, for fixed $\phi,\xi$. Such operator is similar to an half-derivative in the sense of Caputo \cite{caputo}, and we develop an invertibility theory in \S\ref{sec:invj}. In \S\ref{sec:Choiceparameters} we solve (\ref{InnPrOrth}) by fixed-point argument and hence we find $\Lambda,\xi$. A main ingredient of the full proof is the linear theory for the inner problem developed in \cite{cdm} and adapted in dimension 3 in \cite{dmw1}.
	\subsubsection{Statement of the linear estimate for the inner problem}
	We recall the result on the linear theory in dimension 3, proved in \cite{dmw1}. To state the result, we decompose a general function $h(\cdot,\tau)\in L^2(B_{2R})$ for any $\tau \in [\tau_0,\infty)$ in spherical modes. Let $\{\vartheta_m\}_{m=0}^\infty$ the orthonormal basis of $L^2(S^2)$ made up of spherical harmonics, namely the eigenfunctions of the problem
	\begin{align*}
		\Delta_{S^2}\vartheta_m +\lambda_m \vartheta_m=0\inn S^2,
	\end{align*}
	where $0=\lambda_0<\lambda_1=\lambda_2=\lambda_3=2<\lambda_4\leq \dots$. We decompose $h$ into the form
	\begin{align*}
		h(y,\tau)=\sum_{m=1}^{\infty} h_m(\abs{y},\tau)\vartheta_m\qty(\frac{y}{\abs{y}}),\quad h_j(\abs{y},\tau)=\int_{S^2}h(r\theta,\tau)\vartheta_m(\theta)\, d \theta.
	\end{align*}
	Furthermore, we write $	h=h^0+h^1+h^\perp$ where
	\begin{align*}
		h^0=h_0\qty(\abs{y},\tau),\quad h^1=\sum_{m=1}^{3}h_m\qty(\abs{y},\tau) \vartheta_m\qty(\frac{y}{\abs{y}}), \quad h^\perp = \sum_{m=4}^{\infty} h_m\qty(\abs{y},\tau)\vartheta_m\qty(\frac{y}{\abs{y}}).
	\end{align*}
	We solve the inner problem (\ref{InnProbGeneral}) for functions $h$ in the space $X_{\nu,2+a}$ defined by
	\begin{align}\label{DefXnu4space}
		X_{\nu,2+a}{\coloneqq}\{ h\in L^\infty\qty(B_{2R}\times [\tau_0,\infty)): \norm{h}_{\nu,2+a}<\infty 		\},
	\end{align}
	where
	\begin{align*}%\label{DefNormhInnerpr}
		\norm{h}_{\nu,2+a}{\coloneqq}\sup_{\tau>\tau_0, y \in B_{2R}}\tau^{\nu} (1+\abs{y}^{2+a})\abs{h(y,\tau)}.
	\end{align*}
	
	\begin{proposition}\label{propInnerlineartheory}
		Let $\nu,a$ be positive constants. Then for all sufficiently large $R>0$ and any $h(y,\tau)$ with $\norm{h}_{\nu,2+a}<\infty$ such that
		\begin{align*}
			\int_{B_{2R}} h(y,\tau)Z_j(y)\ddy =0 \inn [\tau_0,\infty), \quad\text{for}\quad i=1,2,3,4,
		\end{align*}
		there exist $\phi[h]$ and $e_0[h]$ which solves 
		\begin{align}\label{InnProbGeneral}
			&\pp_\tau \phi= \Delta_ y\phi + 5U(y)^4 \phi +h(y,\tau)\inn B_{2R}\times (\tau_0,\infty)\\\nonumber
			&\phi(y,\tau_0)=e_0 Z_0(y) \inn B_{2R}.
		\end{align}
		They define linear operators of $h$ that satisfy the estimates
		\begin{align}\label{estinn}
			\abs{\phi(y,\tau)}+(1+\abs{y})\abs{\nabla_y \phi(y,\tau)}\lesssim \tau^{-\nu}\bigg[&	\frac{R^2 \theta_0(R,a)}{1+\abs{y}^3}\norm*{h^0}_{\nu,2+a}\\\nonumber
			&+ \frac{R^3 \theta_1(R,a)}{1+\abs{y}^4}\norm*{h^1}_{\nu,2+a}+\frac{1}{1+\abs{y}^a}\norm*{h^{\perp}}_{\nu,2+a}		\bigg],
		\end{align}
		and 
		$$
		\abs{e_0[h]}\lesssim \norm{h}_{\nu,2+a},
		$$
		where
		\begin{equation*}
			\theta_0(R,a) {\coloneqq} \left \{ 
			\begin{matrix}  
				1   & \hbox{\rm  if } a> 2, \\   \log R  & \hbox{\rm if } a= 2, \\ R^{2-a}   & \hbox{ \rm if } a < 2, 
			\end{matrix}\right. , \quad 
			\theta_1(R,a) {\coloneqq} \left \{ 
			\begin{matrix}
				1   & \hbox{\rm  if } a> 1, \\   \log R  & \hbox{\rm if } a= 1, \\ R^{1-a}   & \hbox{ \rm if } a < 1.
			\end{matrix}\right.	
		\end{equation*}
	\end{proposition}
	As we said in \S \ref{constr:u_2}, in order to make the system for $(\phi,\psi)$ weakly coupled, $\phi$ needs to be small at distance $y\sim R$. For this reason, we need to take $a>1$ in the statement of Proposition \ref{propInnerlineartheory}. This makes clear why we need to improve ansatz $u_1$ to $u_3$ in \S\ref{sec:ansatz}.
	Since in our problem $h=H$ as in (\ref{InnPr}) decays as 
	$$\abs{h}\lesssim \mu^{1+l_1} \frac{1}{1+\abs{y}^{4}}=\tau^{-\nu} \frac{1}{1+\abs{y}^{4}},$$
	where $\tau$ is given in (\ref{tauoft}),
	we apply estimate (\ref{estinn}) with constants 
	$$
	a=2,\quad \nu=\frac{1+l_1}{2},
	$$ 
	in the simplified form
	\begin{align}\label{estinnSimplified}
		\abs{\phi}+(1+\abs{y})\abs{\nabla_y \phi(y,\tau)}\lesssim \norm{h}_{\nu,4}\tau^{-\nu}\qty[\frac{R^2 \log(R)}{1+\abs{y}^3}+\frac{R^3}{1+\abs{y}^4}],
	\end{align} 
	and observe that
	\begin{align*}
		\qty[\frac{R^2 \log(R)}{1+\abs{y}^3}+\frac{R^3}{1+\abs{y}^4}]\lesssim  \left\{ 
		\begin{matrix}  
			 R^{-1}\log R   & \hbox{\rm  if } \abs{y}\sim R, \\   R^3  & \hbox{\rm if } \abs{y} \sim 0.
		\end{matrix}\right.
	\end{align*}
	We look for $\phi$ in the space of functions
	\begin{align*}
		X_*{\coloneqq}\{ \phi(y,t)\in L^\infty(\Om \times [t_0,\infty)): \norm{\phi}_*<\infty\},
	\end{align*}
	where
	\begin{align}\label{NormphiDef}
		\norm{\phi}_{*}{\coloneqq}&\sup_{\tau \in [\tau_0,\infty), y \in B_{2R} } \tau^{\nu}\qty[\frac{R^2 \log(R)}{1+\abs{y}^3}+\frac{R^3}{1+\abs{y}^4}]^{-1} \qty[\abs{\phi(y,\tau)}+(1+\abs{y})\abs{\nabla_y \phi(y,\tau)}]\\\nonumber
		&+\sup_{\substack{\tau \in [\tau_0,\infty), y \in B_{2R} \\ \tau_1,\tau_2 \in [\tau,\tau+1]}} \tau^{\nu}\qty[\frac{R^2 \log(R)}{1+\abs{y}^3}+\frac{R^3}{1+\abs{y}^4}]^{-1}\frac{\abs{\phi(y,\tau_1)-\phi(y,\tau_2)}}{\abs{\tau_1-\tau_2}^{\frac{1}{2}+\ve}}\\\nonumber
		&+\sup_{\substack{\tau \in [\tau_0,\infty), y \in B_{2R} \\ \tau_1,\tau_2 \in [\tau,\tau+1]}}\tau^{\nu}\qty[\frac{R^2 \log(R)}{1+\abs{y}^3}+\frac{R^3}{1+\abs{y}^4}]^{-1} (1+\abs{y})\frac{\abs{\nabla_y\phi(y,\tau_1)-\nabla_y\phi(y,\tau_2)}}{\abs{\tau_1-\tau_2}^{\frac{1}{2}+\ve}},
	\end{align}
	for $\ve>0$ fixed small (as in \S\ref{Choiceofconst}).

	We notice that, by standard parabolic estimates, from (\ref{estinnSimplified}) we also get the bound on the H\"older seminorms in (\ref{NormphiDef}), thus
	\begin{align}\label{UsedLinearEst}
		\norm{\phi}_*\leq C \norm{h}_{\nu,4}.
	\end{align}
	\subsubsection{Spaces for the parameters}
	We introduce weighted H\"older spaces for the parameters $\Lambda,\xi$. Let
	\begin{align*}
		X_{\sharp,a,b,\sigma}{\coloneqq}\{\Lambda \in C(t_0,\infty): \norm{\Lambda}_{\sharp,a,b,\sigma}<\infty\},
	\end{align*}
	where
	\begin{align*}
		\norm{\Lambda}_{\sharp,a,b,\sigma}{\coloneqq} \sup_{t>t_0}\qty{\mu(t)^{-a}\norm{\Lambda}_{\infty,[t,t+1]}}+\sup_{t>t_0}\qty{\mu(t)^{-b}[\Lambda]_{0,\sigma,[t,t+1]}},
	\end{align*}
	and
	\begin{align*}
		&\norm{\Lambda}_{\infty,[t,t+1]}=\sup_{s\in[t,t+1]}\abs{\Lambda(s)},\\
		&[\Lambda]_{0,\sigma,[t,t+1]}{\coloneqq}\sup_{\substack{s_1,s_2 \in [t,t+1]\\ s_1 \neq s_2}} \frac{\abs{\Lambda(s_1)-\Lambda(s_2)}}{\abs{s_1-s_2}^{\sigma}}.
	\end{align*}
	We look for $\Lambda$ such that
	\begin{align}\label{NormsLambda,dotLambda}
		&\norm{\Lambda}_{\sharp,l_0,\delta_0,\frac{1}{2}+\ve}+\norm*{\dot \Lambda}_{\sharp,l_1,\delta_1,\ve}<\mathfrak{b}_1,
	\end{align}
	for some positive constant $\ve,\delta_0,\delta_1,l_0,l_1$ to be chosen (see \S\ref{Choiceofconst}).
	We also define $X_{\sharp,c,\sigma}{\coloneqq}X_{\sharp,c,c,\sigma}$ and
	\begin{align*}
		\norm{h}_{\sharp,c,\sigma}{\coloneqq}\sup_{t>t_0} \mu(t)^{-c} \qty[\norm{h}_{\infty,[t,t+1]}+[h]_{0,\sigma,[t,t+1]}].
	\end{align*}
	We consider $\xi_1$ such that 
	\begin{align}\label{NormsXi,dotXi}
		&\norm{\xi_1}_{\sharp,1+k,\frac{1}{2}+\ve}+\norm*{\dot \xi_1}_{\sharp,1+k,\ve}<\mathfrak{b}_2,
	\end{align}
	for some $k>0$ (see \S\ref{Choiceofconst}). The positive constants $\mathfrak{b}_{1},\mathfrak{b}_{2}$ will be selected as small as needed. 
	\subsubsection{Choice of constants} \label{Choiceofconst}
	Here we select the constants 
	\begin{align*}
		l_0,l_1,\delta_0,\delta_1,\ve,\delta,k,\alpha,\beta,\sigma,\kappa,
	\end{align*}
	which are sufficient to find the perturbation $\tilde \phi$ in (\ref{exactSol}) by the inner-outer gluing scheme. 
	Firstly, we indicate where the constants appear in the scheme:
	\begin{itemize}
		\item $l_0,l_1,\delta_0,\delta_1,\ve$ appear in the definition (\ref{NormsLambda,dotLambda});
		\item $k$ is used in the norm (\ref{NormsXi,dotXi}) for $\xi$;
		\item $\delta$ appears in $R(t)=\mu^{-\delta}$, that is the radius of the inner regime;
		\item $\alpha,\beta$ is used in the norms for the outer problem, see (\ref{defFouterNorm}) and (\ref{normPsi});
		\item $\sigma>0$ appears in the choice of $\beta=l_1+\delta+\sigma$ in the outer problem;
		\item $\kappa >0$ is the constant appearing in Proposition \ref{PropOuterNonlinear}.
	\end{itemize}
	We fix the following values:
	\begin{itemize}
		\item $\delta=\frac{2}{9}$;
		\item $l_1=k=\frac{2}{3}$;
		\item $l_0=l_1+\frac{\delta}{2}=\frac{7}{9}$;
		\item $\sigma=2\alpha=\ve=\frac{1}{100}$;
		\item $\delta_1=l_1+\delta-\sigma -(1-\delta)(1+\alpha/2)(1+2\ve)$;
		\item $\delta_0 = l_1 + \delta - \sigma - (1-\delta)(1+\alpha/2)2\ve$;
		\item $\beta=\frac{1}{2}+l_1+\delta-\sigma$;
		\item $\kappa = \gamma (\sigma-\alpha \delta)$
	\end{itemize}
	These choices are dictated by the following constraints, based on the estimate of the approximate solution, the linear theory for inner (Proposition \ref{propInnerlineartheory}) and outer problem (Lemma \ref{LinearLemmapsi}), the characterization of the orthogonality conditions (\ref{SystemEquivalent}) and the estimates in Proposition \ref{Proposition:invj}:
	\begin{itemize}
		\item $l_1+\delta<1$ to make $\beta<3/2$ and apply the outer linear estimate (\ref{LinEstPsi1inf}); 
		\item we need $$\delta \in \qty(\frac{1-l_1}{2}+\hat \ve,\frac{1+l_1}{6}), \quad\text{where}\quad \hat \ve=\frac{(1+\alpha/2)(1+2\ve)-l_1+\sigma}{1+(1+\alpha/2)(1+2\ve)}-\frac{1-l_1}{2}.$$ 
		Up to choosing $\sigma>\alpha>0$ and $\ve>0$ small enough, these range is equivalent to (\ref{condouter1}), which, together with the previous restriction, impose a range for $\delta$ and $l_1$ leading (for instance) to the choice (\ref{deltal1});
		\item $l_0\geq l_1$ and $k+1\geq 2\delta +l_1$ to get $\mu^{5/2}\Sinn$ controlled by the term $\mu(t) 5U(y)^4 J(x,t)$;
		\item $\sigma>\alpha \delta>0$, $\ve>0$ and $\kappa \in (0,2\gamma (\sigma-\alpha \delta))$. This allows to estimate the $R^{\alpha}\log R\lesssim e^{-\kappa t}\mu^{-\sigma}$ when we need to control the term $\mu^{-1} \phi \Delta_x \eta_R$ in the outer error;
		\item $k=l_1$. From (\ref{Eqi=123Char}) we need  $\abs{\xi_1}+\abs*{\dot \xi_1} \lesssim \mu^{1+l_1}$, thus the choice of $k$, which is consistent with (\ref{condk});
		\item in the outer problem we obtain $\abs{\psi(x,t)}\lesssim \frac{\mu^{l_1-\sigma}R^{-1}}{1+\abs{y}^\alpha}$. The nonlocal equation (\ref{i=4Equivalent}) and the estimate (\ref{l1minl0}) asks for $\abs{\Lambda(t)}\lesssim \abs{\psi(\xi(t),t)}$. Thus, this leads to the a choice of $l_0\in [l_1,l_1+\delta-\sigma]$;
		\item from estimate (\ref{LinEstLambda<h}), equation (\ref{i=4Equivalent}) and the bound on the $\ve$-H\"older seminorm of $\psi$ we get
		\begin{align*}
			[\Lambda]_{0,\frac{1}{2}+\ve,[t,t+1]}\lesssim [\psi(\xi(\cdot),\cdot)]_{0,\ve,[t,t+1]}\lesssim  \frac{\mu^{l_1+\delta-\sigma}}{(\mu R)^{(1+\frac{\alpha}{2})2\ve}}= \mu^{\delta_0},
		\end{align*}
		which gives $\delta_0$;
		\item similarly, from (\ref{estpsi}) the H\"older estimate on the outer solution gives $$[\psi(0,\cdot)]_{0,\frac{1}{2}+\ve,[t,t+1]}\lesssim \mu^{l_1+\delta-\sigma}(\mu R)^{-(1+\frac{\alpha}{2})(1+2\ve)},$$ 
		and by equation (\ref{i=4Equivalent}) and estimate (\ref{LinEstdotLambda<h}) we need $$[\dot \Lambda]_{0,\ve,[t,t+1]}\lesssim [\psi(0,\cdot)]_{0,\frac{1}{2}+\ve,[t,t+1]}.$$ This leads to the choice of $\delta_1$;
		\item after choosing $\sigma=2\alpha>0$ small so that $\delta>\sigma$, the constant $\ve$ is chosen small enough to make $\delta_1$ positive (any choice of $\alpha \in (0,\delta/2)$ and $\ve$ such that $\delta_1=\delta_1(\alpha,\ve)>0$ is sufficient).
	\end{itemize}

	\section{Solving the outer problem}\label{sec:out}
	\medskip
	We devote this section to solve the outer problem (\ref{OutPr})
	\begin{align*}
		&\pp_t \psi =\Delta_x \psi+\g \psi+V \psi +f[\psi,\phi,\Lambda,\dot \Lambda,\xi,\dot \xi](x,t), \inn \Omega \times [t_0,\infty),\\\nonumber
		&\psi(x,t)=-\mu_0^{-1/2}u_3(x,t)\quad \onn \pp \Omega \times [t_0,\infty),\\
		&\psi(x,t_0)=\psi_0(x)\inn \Om,
	\end{align*}
	where $\psi_0(x)$ is any suitable small initial condition, 
	\begin{align}\label{deff(x,t)}
		f(x,t)=&\mu^{-1}\qty(\frac{\mu}{\mu_0})^{1/2}\phi \pp_t \eta_R \\\nonumber&
		+ \mu^{-1}\qty(\frac{\mu}{\mu_0})^{1/2}\eta_R\qty{(\g-\dot \Lambda)(\phi+2 \nabla_y \phi \cdot y)  -\nabla_y \phi\cdot \qty(\frac{\dot \xi}{\mu})}\\\nonumber
		&+\mu^{-1}\qty(\frac{\mu}{\mu_0})^{1/2}\qty(\phi \qty(\frac{2}{\abs{z}}\frac{\eta'(\abs{z})}{\mu^2 R^2}+\frac{\eta''\qty(\abs{z})}{\mu^2 R^2}) +2\frac{\nabla_y \phi}{\mu} \cdot \frac{z}{\abs{z}}\frac{\eta'(\abs{z})}{\mu R}) \\\nonumber
		&+\mu_0^{-1/2}S_{\text{in}}\qty(1-\eta_R)+\mu_0^{-1/2}S_{\text{out}}+\mu_0^{-1/2}\mathcal{N}(u_3,\tilde \phi)(1-\eta_R)
	\end{align}
	and potential
	\begin{align*}
		V(x,t)&=5u_3^4 (1-\eta_R),
	\end{align*}
	which, by the definition of $u_3$, using again the bounds on $H_\gamma,J,\phi_3$ and the support of $(1-\eta_R)$, satisfies
	\begin{align}\label{VBound}
		\abs{V}&\lesssim \mu^{-2}U(y)^4 \abs{\qty(1-\mu \frac{H_\gamma+J+\mu^{-1}\phi_3 \eta_l}{U})}^4 \\\nonumber
		&\lesssim (1-\eta_R)\frac{\mu^{-2}}{1+\abs{y}^4}[1+\mu(1+\abs{y})]^4\\\nonumber
		&\lesssim (1-\eta_R)\frac{\mu^{-2}}{1+\abs{y}^4}\\\nonumber
		&\lesssim \frac{\mu^{-2}}{1+\abs{y}^2}R^{-2}.
	\end{align}
	Let 
	$$
	\psi_1(x,t){\coloneqq}\mu_0(t)^{1/2}\psi(x,t).
	$$ 
	Then, the problem for $\psi_1$ becomes
	\begin{align}\label{probpsi1Real}
		&\pp_t \psi_1 = \Delta_x \psi_1+V \psi_1 + F[\psi,\phi,\Lambda,\dot \Lambda, \xi,\dot \xi](x,t)\inn  \Om \times [t_0,\infty),\\\nonumber
		&\psi_1(x,t)=g(x,t)  \onn \Om \times [t_0,\infty),\\\nonumber
		&\psi_1(x,t_0)=\psi_{1,0}(x) \inn \Omega
	\end{align}
	where
	\begin{align*}
		&F(x,t){\coloneqq}\mu_0(t)^{\frac{1}{2}}f(x,t)\\
		&g(x,t){\coloneqq}-u_3(x,t)\\
		&\psi_{1,0}(x){\coloneqq}\mu_0(t_0)^{\frac{1}{2}}\psi_0(x).
	\end{align*}
	In particular, in the proof of Proposition \ref{PropOuterNonlinear} we prove that for any $\alpha>0$  
	\begin{align}\label{RelF}
		&\abs{F(x,t)}\lesssim e^{-\kappa t_0}\frac{\mu^{l_1+\delta-\sigma}\mu^{-2}}{1+\abs{y}^{2+\alpha}}.
	\end{align}
	Also, using the definition of $u_3$, in (\ref{sizeboundary}) we prove
	\begin{align}\label{Relg}
		&\abs{g(x,t)}\lesssim \mu^{\frac{5}{2}}.
	\end{align}
	Firstly, we consider the linear version of (\ref{probpsi1Real}). Let
	\begin{align}\label{defFouterNorm}
		\abs{F(x,t)}\leq \norm{F}_{\beta-2,\alpha+2} \frac{\mu^{\beta}\mu^{-2}}{1+\abs{y}^{\alpha+2}},
	\end{align}
	for some $\beta>0,\alpha>0$, where $\norm{F}_{\beta-2,\alpha+2}$ is the best constant for such inequality. 
	Also, for $\delta \in (0,1/2)$ and $\sigma \in (0,1)$ we define the H\"older norms
	\begin{align*}
		&[f]_{0,2\delta,\delta, \Omega \times [t,t+1]}:=\sup_{\substack{x_1\neq x_2 \in \Omega \\ t_1\neq t_2 \in [t,t+1]}}\frac{\abs{f(x_1,t_1)-f(x_2,t_2)}}{\abs{x_1-x_2}^{2\delta}+\abs{t-t_1}^{\delta}}\\
		&[f(x,\cdot)]_{\sigma,[t,t+1]}:= \sup_{t_1\neq t_2 \in [t,t+1]}\frac{\abs{f(x,t)-f(x,t)}}{\abs{t_1-t_2}^{\sigma}}\\
		&[f(\cdot,t)]_{0,\sigma, \Omega}:=\sup_{x_1\neq x_2 \in \Omega}\frac{\abs{f(x_1,t)-f(x_2,t)}}{\abs{t_1-t_2}^{\sigma}}
	\end{align*}
	\begin{lemma}\label{LinearLemmapsi}
		Let $F$ such that $\norm{F}_{\beta-2,\alpha+2}<\infty$ for some constants $\beta<3/2 $ and $\alpha\in (0,1)$. Furthermore, assume that $\norm{e^{as}g(s)}_{L^{\infty}(\pp  \Omega \times (t_0,\infty))}<\infty$ for some $a>0$ and $\norm{h}_{L^\infty(\Omega)}<\infty$. Let $\psi_1[F,g,h]$ be the unique solution to 
		\begin{align}\label{psi1linearproblem}
			&\pp_t \psi_1 = \Delta_x \psi_1 + V \psi_1 + F(x,t)\inn \cyt,\\\nonumber
			&\psi_1(x,t) = g(x,t) \onn \pp \cyt,\\\nonumber
			&\psi_1(x,t_0)=h(x) \inn \Om.
		\end{align}
		Then, for $b \in (0,\lambda_1)$ and $\tilde a \in (0,\min\{a,\lambda_1-\ve\}]$ for $\ve>0$ arbitrary small, we have
		\begin{align}\label{LinEstPsi1inf}
			\abs{\psi_1(x,t)}\lesssim &\norm{F}_{\beta-2,\alpha+2} \frac{\mu^{\beta}}{1+\abs{y}^{\alpha}}+e^{-b(t-t_0)}\norm{h}_{L^\infty(\Omega)}\\\nonumber
			&+e^{-\tilde a(t-t_0)}\norm{e^{as}g}_{L^\infty(\pp \Omega \times (t_0,\infty))}
		\end{align}
		for all $x=\mu y+\xi \in \Om$ and $t>t_0$. Furthermore, the following local estimate on the gradient holds:
		\begin{align}\label{LinEstPsi1nabla}
			&\abs{\nabla_x\psi_1(x,t)}\lesssim \norm{F}_{\beta-2,\alpha+2} \frac{
				\mu^{\beta-1}}{1+\abs{y}^{\alpha+1}}\quad \text{for}\quad \abs{y}<R,\\\nonumber
			&[\nabla_x \psi_1(\cdot,t)]_{0,2\ve,B_{\mu R}(\xi)}\lesssim \norm{F}_{\beta-2,\alpha+2} {\mu^{\beta-1-2\ve}} R^{-1-2\ve}
		\end{align}	
		where $R\leq \delta \mu^{-1}$ for sufficiently small $\delta>0$.
		Also, one has
		\begin{align}\label{LinEstPsi1Calpha}
			&\qty[\qty(\mu R)^{2+\alpha}]^{\frac{1}{2}+\ve}\sup_{x \in B_{R\mu}(\xi)}{[\psi_1(x,\cdot)]_{0,\frac{1}{2}+\ve,[t,t+1]}}\\\nonumber
			&+\qty[\qty(\mu R)^{2+\alpha}]^{\ve}\sup_{x\in B_{R\mu(\xi)}}[\psi_1(x,\cdot)]_{0,\ve,[t,t+1]} \lesssim \norm{F}_{\beta-2,\alpha+2}\mu^{\beta}.
		\end{align}
	\end{lemma}
	
	\begin{proof}
		To prove the result is enough to find a supersolution to the problem
		\begin{align*}
			&\pp_t \psi_2 =\Delta \psi_2+F\inn \cyt,\\
			&\psi_2=g \onn \pp \cyt,\\
			&\psi_2=h \inn \Om.
		\end{align*}
		We use the notation $\psi_2=\psi_2[F,g,h]$. Indeed, suppose that $\bar{\psi}_2$ is a supersolution to this problem. By (\ref{VBound}) we have
		\begin{align*}
			\abs*{V\bar \psi_2}\lesssim \frac{\mu^{\beta-2}}{1+\abs{y}^{2+\alpha}}R(t_0)^{-2},
		\end{align*}
		and hence $\norm*{V \bar \psi_2}_{\beta-2,2+\alpha}<R(t_0)^{-2}$ for $t_0$ sufficiently large. Thus, we find that a large multiple of $\bar \psi_2$ is a supersolution of (\ref{psi1linearproblem}).
		Firstly, let $F,g\equiv 0$ and consider $\psi_2[0,0,h]$. Let $v_0(x)$ be the solution to 
		\begin{align*}
			&-\Delta_x v_0-b v_0=0 \inn \Om,\\
			&v_0=1\onn \pp \Om,
		\end{align*}
		for $b \in (0,\lambda_1)$ and define
		\begin{align*}
			\bar \psi_2 =\norm{h}_\infty e^{-b (t-t_0)}v_0(x).
		\end{align*}
		We claim that $\bar \psi_2$ is a supersolution for $\psi_2[0,0,h]$. Indeed, we have
		\begin{align*}
			&\pp_t \bar \psi_2 - \Delta_x \bar \psi_2 =\norm{h}_\infty e^{-b(t-t_0)}\qty(-b v_0-\Delta_x v_0)=0\inn \cyt,\\
			&\bar \psi_2(x,t)=\norm{h}_\infty e^{-b(t-t_0)}\geq 0 \onn \pp \cyt,\\
			&\bar \psi_2(x,t_0)=\norm{h}_\infty v_0(x)\geq h(x) \inn \Om, 
		\end{align*}
		where the last inequality is a consequence of the maximum principle applied to $v_0$.
		Secondly, we look for a supersolution to $\psi_2[0,g,0]$. Let $v_1(x)$ to be the solution to
		\begin{align*}
			&-\Delta_x v_1 - \tilde a v_1 =0 \inn \Om,\\
			&v_1(x)=1 \onn \pp \Om,
		\end{align*}
		where $\tilde a \in (0,\min\{a,\lambda_1-\ve\}] $ and consider
		\begin{align*}
			\bar \psi_2(x,t)= \norm{e^{as}g}_{L^\infty(\pp\Om\times (t_0,\infty))} e^{-\tilde a(t-t_0)}v_1(x).
		\end{align*}
		We verify that
		\begin{align*}
			&\pp_t \bar \psi_2- \Delta \bar \psi_2  = \norm{e^{as}g}_\infty e^{-\tilde a(t-t_0)}\qty(-\tilde a v_1-\Delta v_1)=0 \inn \cyt,\\
			&\bar \psi_2(x,t)=\norm{e^{as}g}_\infty e^{-\tilde a(t-t_0)}\geq g(x,t)\onn \pp \cyt,\\
			&\bar \psi_2(x,t_0)= \norm{e^{as}g}_{L^\infty(\pp\Om\times (t_0,\infty))} v_1(x)\geq 0 \inn \Om,
		\end{align*}
		where we used $\tilde a\leq a$ to get the second inequality and $\tilde a<\lambda_1$ to get the third one by the maximum  principle.
		It remains to find a supersolution for $\psi_2[F,0,0]$. Let $\psi_2[F,0,0]=e^{-c(t-t_0)}\psi_3$, where $c=2\g\beta $ so that
		\begin{align*}
			\pp_t \psi_3
			= \Delta_x \psi_3 + c \psi_3+\frac{\mu^{-2}}{1+\abs{y}^{2+\alpha}}. 
		\end{align*}
		We find a bounded $\bar \psi_3$ supersolution in case $c<\lambda_1$, that is $3\g <\lambda_1$. Consider 
		\begin{align*}
			\bar \psi_3 = \ppsi_0\qty(\frac{x-\xi}{\mu})\eta\qty(\frac{x-\xi}{d})+\ppsi_1(x,t).
		\end{align*}
		We need 
		\begin{align}\label{condpsi2supers}
			\pp_t \ppsi_1-\Delta_x \ppsi_1-c\ppsi_1 \geq& \eta \qty[-\pp_t \ppsi_0+ \mu^{-2}\Delta_y \ppsi_0 +c \ppsi_0+ \frac{\mu^{-2}}{1+\abs{y}^{2+\alpha}}]\\\nonumber
			&+(1-\eta)\frac{\mu^{-2}}{1+\abs{y}^{2+\alpha}}+(\Delta_x \eta-\pp_t \eta)\ppsi_0+2 \mu^{-1}\nabla_x \eta\cdot \nabla_y \ppsi_0,
		\end{align}
		with $\bar \psi_3(x,t)\geq 0$ on $\pp \cyt$ and initial datum $\bar \psi_3(x,t_0)\geq 0$.
		Suppose without loss of generality that $\Omega\subset B_1$ and take $\ppsi_0$ as the solution to 
		\begin{align*}
			&\Delta_y \ppsi_0=-\frac{2}{1+\ay^{2+\alpha}}\inn B_{\mu^{-1}}\\
			&\ppsi_0 = 0\onn \pp B_{\mu^{-1}}.
		\end{align*}
		From the variation of parameters formula
		\begin{align}\label{VarPar}
			\ppsi_0(\abs{y})=2\omega_3\int_{\abs{y}}^{\mu^{-1}}\frac{1}{\rho^2}\int_{0}^{\rho}\frac{s^2}{1+s^{2+\alpha}}\dds \,d\rho,
		\end{align}
		we find
		\begin{align*}
			&\abs{\ppsi_0}\lesssim\frac{1}{1+\abs{y}^{\alpha}},
		\end{align*}
		and
		\begin{align*}
			\abs{\pp_t \ppsi_0} &\lesssim  \pp_t\qty(\frac{1}{\mu}) \mu^2 \int_0^{\mu^{-1}}\frac{s^2}{1+s^{2+\alpha}}\dds + \pp_t\qty(\abs{y(x,t)})\frac{1}{\abs{y}^2}\int_{0}^{\abs{y}}\frac{s^2}{1+s^{2+\alpha}}\dds \\&\lesssim \frac{1}{1+\abs{y}^\alpha}+\frac{\abs{y}^2}{1+\abs{y}^{2+\alpha}}\\
			&\lesssim \frac{1}{1+\abs{y}^{\alpha}}
		\end{align*}
		Also, if $\abs{x-\xi}<d$ for $d$ fixed sufficiently small, we obtain
		\begin{align*}
			-\pp_t \ppsi_0+ \mu^{-2}\Delta_y \ppsi_0 +c \ppsi_0+ \frac{\mu^{-2}}{1+\abs{y}^{2+\alpha}}=-\frac{\mu^{-2}}{1+\ay^{2+\alpha}}+O\qty(\frac{1}{1+\abs{y}^{\alpha}}) <0.
		\end{align*}
		Now, we take $\ppsi_1$ as the solution to
		\begin{align*}
			&\pp_t \ppsi_1 -\Delta_x \ppsi_1-c\ppsi_1=(1-\eta) \frac{\mu^{-2}}{1+\abs{y}^{2+\alpha}} + \qty(\Delta_x \eta-\pp_t \eta)\ppsi_0+2 \mu^{-1}\nabla_x \eta \cdot \nabla_y \ppsi_0,\\
			&\ppsi_1=0 \onn \pp \cyt,\\
			&\ppsi_1(x,t_0)=0 \inn \Om.
		\end{align*}
		We estimate the right-hand side by
		\begin{align*}
			&(1-\eta)\frac{\mu^{-2}}{1+\abs{y}^{2+\alpha}}\lesssim  \mu^{\alpha},\\
			& \abs{\qty(\Delta_x \eta-\pp_t \eta)\ppsi_0} \lesssim \mu^{\alpha},\\
			&\abs{2 \mu^{-1}\nabla_x \eta \cdot \nabla_y \ppsi_0}\lesssim \mu^{\alpha}.
		\end{align*}
		Hence, by comparison principle using $c<\lambda_1$ we obtain a solution $\abs*{\bar \psi_3}\lesssim \mu^{\alpha}$. Thus, inequality (\ref{condpsi2supers}) is satisfied. Also, $\bar \psi_3=0$ on $\pp \cyt$ and $\psi_3(x,t_0)=\eta \ppsi_0(x,t_0)\geq 0$. We conclude that $\bar \psi_3$ is a supersolution and the bound (\ref{LinEstPsi1inf}) is proven. Now, we prove the gradient estimate (\ref{LinEstPsi1nabla}).
		Let 
		$$\psi_1(x,t)=:\tilde \psi\qty(z(x,t),\tau(t)),\quad\text{where}\quad z\coloneqq\frac{x-\xi(t)}{R(t)\mu(t)}$$
		and $\dot \tau(t)=(R(t)\mu(t))^{-2}$, that gives $\tau(t)\sim \mu^{-2}$. We can take $\tau(t_0)=2$. %Suppose that $\abs{z}<R<\delta \mu^{-1}$ for $\delta$ sufficiently small. 
		The equation for $\tilde \psi$ becomes
		\begin{align*}
			\pp_\tau \tilde \psi = \Delta_z \tilde \psi +a(z,\tau) \cdot \nabla_z \tilde \psi + b(z,\tau)\tilde \psi + \tilde F(z,\tau),
		\end{align*}
		where $\tilde F(z,\tau(t))=(R\mu)^2 F(\mu R z+\xi, \tau(t))$, and the coefficients
		\begin{align*}
			&a(z,\tau)\coloneqq (\mu_0 R)\qty[z \pp_t(\mu_0)+\dot \xi]  ,\quad b(z,\tau)\coloneqq (R\mu)^2 V(\mu R z +\xi,\tau(t))\lesssim \frac{1}{1+R^2\abs{z}^2},
		\end{align*}
		are uniformly bounded. Since $\norm{F}_{\beta-2,\alpha+2}<\infty$, we have
		\begin{align*}
			\tilde F(z,\tau(t))=(R\mu)^2 F(\mu Rz+\xi, \tau(t))\lesssim \mu^{\beta}\frac{\norm{F}_{\beta-2,\alpha+2}}{1+\abs{Rz}^{2+\alpha}}.
		\end{align*}
		We have already proved the $L^\infty$-bound 
		$$
		\norm{\psi_1}_{\beta,\alpha}\lesssim \norm{F}_{\beta-2,\alpha+2}.
		$$
		We apply standard local parabolic estimates for the gradient: let $\sigma \in (0,1)$ and $\tau_1\geq \tau(t_0)+2$, then
		\begin{align*}
			[\nabla_z \tilde \psi_1(\cdot,\tau_1)]_{0,\sigma,B_1(0)}+\norm*{\nabla_z \tilde \psi_1(\cdot,\tau_1)}_{L^\infty\qty(B_1(0))}&\lesssim \norm*{\tilde \psi}_{L^\infty(B_2(0))\times (\tau_1-1,\tau_1)}+\norm*{\tilde F}_{L^\infty\qty(B_{2}(0)\times (\tau_1-1,\tau_1))}\\
			&\lesssim \mu(t(\tau_1-1))^{\beta} \norm{F}_{\beta-2,\alpha+2}\\
			&\lesssim \mu(t(\tau_1))^{\beta} \norm{F}_{\beta-2,\alpha+2}.
		\end{align*}
		In the original variables, for any $t\geq t_0+2$ we find
		\begin{align}\label{gradestPsifinal}
			(R\mu)^{1+\sigma}[\nabla_x \psi_1(\cdot,t)]_{0,\sigma, B_{\mu R}(\xi)}+R\mu \norm{\nabla_x \psi_1(\cdot,t)}_{L^\infty(B_{\mu R}(\xi))}\lesssim \mu^{\beta}\norm{F}_{\beta-2,\alpha+2}.
		\end{align}
		By similar parabolic estimates using $\norm{\nabla_x\psi_0}_\infty<\infty$ we can extend estimate (\ref{gradestPsifinal}) up to $t=t_0$, thus, the proof of (\ref{LinEstPsi1nabla}) is complete. Now, we prove estimate (\ref{LinEstPsi1Calpha}). We consider the H\"older seminorms.
		We perform the change of variable
		\begin{align*}
			\psi_1(x,t)=\hat{\psi}\qty(z,\tau),
		\end{align*}
		where $z{\coloneqq}(x-\xi)/({R\mu_0})^{1+\frac{\alpha}{2}}$ and $\tau$ satisfies
		\begin{align*}
			\frac{d\tau}{dt}=\frac{1}{\qty(\mu_0(t) R(t))^{2+\alpha}},
		\end{align*}
		that is
		\begin{align*}
			\tau-\tau_0&=\int_{t_0}^\infty \frac{\dds }{\qty(\mu_0(t) R(t))^{2+\alpha}}\dds \\
			&=C \qty(\mu_0 R)^{-(2+\alpha)}(1+o(1)).
		\end{align*}
		The equation for $\hat\psi$ is
		\begin{align*}
			\pp_\tau \hat  \psi = \Delta_z \hat \psi +\hat a(z,\tau)\cdot \nabla_z \hat \psi+\hat b(z,\tau) \hat  \psi + \hat f(z,\tau),
		\end{align*}
		where
		\begin{align*}
			&\hat a(z,\tau)=   (\mu_0 R)^{1+\frac{\alpha}{2}}\qty[z  \pp_t(\mu_0 R)^{1+\frac{\alpha}{2}}+\dot \xi] ,\quad \hat b(z,\tau) = (\mu_0 R)^{2+\alpha}V((\mu_0 R)^{1+\frac{\alpha}{2}}z+\xi,t(\tau)),\\
			&\hat F = (\mu_0 R)^{2+\alpha} F((\mu_0 R)^{1+\frac{\alpha}{2}}z+\xi,t(\tau))
		\end{align*}
		Then, applying local parabolic estimates on $\tilde \psi$, we get
		\begin{align*}
			[\psi_1(x,\cdot)]_{0,\frac{1+2\ve}{2},[t,t+1]}=&\sup_{{t_1 \neq t_2 \in [t,t+1]}}\frac{\abs{\psi_1(x,t_1)-\psi_1(x,t_2)}}{\abs{t_1-t_2}^{\frac{1+2\ve}{2}}}
			\\\lesssim & \sup_{\tau_1,\tau_2 \in [\tau,\tau+1]}\frac{\abs*{\hat \psi(z_1,\tau_1)-\hat \psi(z_1,\tau_2)}}{\abs{\tau_1-\tau_2}^{\frac{1+2\ve}{2}}}\frac{\abs{\tau_1-\tau_2}}{\abs{t_1-t_2}}^{\frac{1+2\ve }{2}}\\
			\lesssim& [\hat \psi(z,\cdot)]_{0,\frac{1+2\ve}{2},[\tau_1,\tau_2]}  \frac{1}{[\qty(\mu R)^{2+\alpha}]^{\frac{1+2\ve}{2}}}\\
			\lesssim& \norm{F}_{\beta-2,\alpha+2} \mu^{\beta}\frac{1}{[\qty(\mu R)^{2+\alpha}]^{\frac{1+2\ve}{2}}},
		\end{align*}
		where $\tau_i=\tau(t_i)$ and $z=z(x,t_i)$ for $i=1,2$.
		Similarly, using the H\"older coefficient $(2\ve,\ve)$, we get
		\begin{align*}
			[\psi_1(x,\cdot)]_{ 0,\ve, [t,t+1]}\lesssim \mu^{\beta}\frac{1}{[\qty(\mu R)^{2+\alpha}]^{\ve}}.
		\end{align*}
	\end{proof}

	We introduce the following weighted norms for $\psi$:
	\begin{align}\label{normPsi}
		\norm{\psi}_{**}:= &\sup_{x \in \Omega, t>t_0} \qty{g_1(x,t)^{-1}\abs{\psi(x,t)}}+\sup_{t>t_0}\qty{g_2(t)^{-1}\sup_{x \in B_{\mu R}(\xi)}[\psi(x,\cdot)]_{0,\ve,[t,t+1]}} \\\nonumber
		&+\sup_{t>t_0}\sup_{x\in B_{\mu R}(\xi)}\qty{g_3(x,t)^{-1}\abs{\nabla_x \psi(x,t)}}+\sup_{t>t_0} g_4(t)^{-1}[\nabla_x \psi]_{0,2\ve,\ve,B_{R\mu}(\xi)\times [t,t+1] }\\\nonumber
		&+\sup_{t>t_0}\qty{g_5(t)^{-1}\sup_{x\in B_{\mu R}(\xi)}[\psi(x,\cdot)]_{0,\frac{1}{2}+\ve,[t,t+1]}}
	\end{align}
	where
	\begin{align*}
		&g_1(x,t)=\frac{\mu^{\beta}}{1+\abs{y}^{\alpha}},\quad g_2(t)=\mu^{\beta}\qty[(\mu R)^{2+\alpha}]^{-\ve},\quad g_3(x,t)=\frac{\mu^{\beta-1}}{1+\abs{y}^{\alpha+1}},\\
		&g_4(t)=\mu^{\beta}[\mu R]^{-1-2\ve},\quad g_5(t)={\mu^{\beta}}{\qty[(\mu R)^{2+\alpha}]^{-(\frac{1}{2}+\ve)}}
	\end{align*}
	and define the space of functions 
	\begin{align*}
		X_{**}=\{\psi \in L^\infty(\cyt): \norm{\psi}_{**}<\infty\}.
	\end{align*}
	Now, we are ready to solve the outer problem (\ref{OutPr}) for $\phi$ such that
	\begin{align}\label{Normphi}
		\norm{\phi}_{*}<\bold b,
	\end{align} 
	for parameters satisfying (\ref{NormsLambda,dotLambda}) and (\ref{NormsXi,dotXi}).
	\begin{proposition}\label{PropOuterNonlinear}
		Assume that $\Lambda,\xi_1,\phi$ satisfy (\ref{NormsLambda,dotLambda}), (\ref{NormsXi,dotXi}) and (\ref{Normphi}) respectively. Also, suppose $\psi_0 \in C^2(\bar \Om)$ such that
		\begin{align*}
			\norm{\psi_0}_{L^\infty}+\norm{\nabla \psi_0}<e^{-\kappa t_0},
		\end{align*}
		for some $\kappa\in (0,2\gamma(\sigma-\alpha \delta))$. Then, there exists $t_0$ large so that problem (\ref{OutPr}) has a unique solution $\psi=\Psi[\Lambda,\dot \Lambda,\xi,\dot \xi,\phi]$ and given $\alpha>0$, there exists $C_{**}$ such that
		\begin{align}\label{estpsi}
			\norm{\psi}_{**}\leq e^{-\kappa t_0}C_{**},
		\end{align}
		where $C_{**}=C_{**}(\bold b,\mathfrak{b}_1,\mathfrak{b}_2)$ and $\bold b,\mathfrak{b}_1,\mathfrak{b}_2$ are the constants in (\ref{Normphi}),(\ref{NormsLambda,dotLambda}) and (\ref{NormsXi,dotXi}) respectively.
	\end{proposition} 
	\begin{proof}
	Let $T_1$ the linear operator, defined by Lemma \ref{LinearLemmapsi}, such that, given $\beta<3/2,\alpha>0$ and functions $f,g,h$ with bounded norms $\norm{f}_{\beta-2,\alpha+2},\norm{e^{as}g}_{\infty},\norm{h}_\infty$ respectively, $T[f,g,h]$ is the solution to (\ref{psi1linearproblem}).\\
	Let 
	\begin{align*}
		\mu_0^{1/2}\psi=\psi_A+\psi_B,
	\end{align*}
	where we define $\psi_B{\coloneqq}T(0,-u_3,\mu_0(t_0)^{1/2}\psi_0)$. From the definition of $u_3(x,t)$ we expand for $x\in \pp \Om$ and $t_0$ large, to get
	\begin{align}\label{sizeboundary}
		u_3(x,t)
		=&\frac{\alpha_3\mu^{1/2}}{\qty(\mu^2+\abs{x-\xi}^2)^{1/2}}-\mu^{1/2}\frac{\alpha_3}{\abs{x-\xi}}\\\nonumber
		=&\alpha_3 \mu^{1/2}\abs{x-\xi}^{-1}\qty[\qty(\frac{\mu^2+\abs{x-\xi}^2}{\abs{x-\xi}^2})^{-1/2}-1]\\\nonumber
		=&\mu^{5/2}f_B(x,t),
	\end{align}
	for a smooth bounded function $f_B(x,t)$ on $\pp \Omega \times [t_0,\infty)$. Hence, Lemma \ref{LinearLemmapsi} gives the bound
	\begin{align*}
		\abs{\psi_B}\lesssim e^{-b (t-t_0)}\norm{\psi_0}_{L^\infty} + e^{-a(t-t_0)}\norm{e^{as}u_3}_{L^\infty(\pp \Om \times [t_0,\infty))},
	\end{align*}
	for any $b<\lambda_1$ and $a<\min\{5\g,\lambda_1-\ve\}$ for any $\ve>0$.\\
	Now, we apply the fixed point theorem to find $\psi_A$ such that $\psi$ satisfies (\ref{estpsi}). We obtain a solution $\psi$ if $\psi_A$ satisfies
	\begin{align*}
		\psi_A=\mathcal{A}(\psi_A),\quad \mathcal{A}(\psi_A)\coloneqq T[F(\psi_A),0,0],
	\end{align*}
	where $F(x,t)=\mu_0(t)^{1/2}f(x,t)$ and $f$ is given by  (\ref{deff(x,t)}).
	We look for $\psi_A$ in 
	\begin{align*}
		\mathcal{B}=\{	\psi_A: \quad \norm*{\mu_0^{-1/2}\psi_A}_{**}<M e^{-\kappa t_0}	\},
	\end{align*}
	where $M$ is a fixed large constant, independent of $t$ and $t_0$.
	We prove that $\AAA(\psi_A) \in \BBB$  for any $\psi_A \in \BBB$.  
	Firstly, we estimate the $L^\infty$ norm of $F(\psi_A)$. From (\ref{normPsi}) we apply Lemma \ref{LinearLemmapsi} with $\beta=1/2+l_1+\delta<3/2$.
	We recall that $F=\mu_0^{1/2}f$ where
	\begin{align*}
		f(x,t)=&\mu_0^{-1/2}\mathcal{N}(u_3,\tilde \phi)(1-\eta_R) \\\nonumber&
		+\mu^{-1}\qty(\frac{\mu}{\mu_0})^{1/2}\phi \pp_t \eta_R
		+ \mu^{-1}\qty(\frac{\mu}{\mu_0})^{1/2}\eta_R\qty{(\g-\dot \Lambda)(\phi+2 \nabla_y \phi \cdot y)  -\nabla_y \phi\cdot \qty(\frac{\dot \xi}{\mu})}\\\nonumber
		&+\mu^{-1}\qty(\frac{\mu}{\mu_0})^{1/2}\qty(\phi \qty(\frac{2}{\abs{z}}\frac{\eta'(\abs{z})}{\mu^2 R^2}+\frac{\eta''\qty(\abs{z})}{\mu^2 R^2}) +2\frac{\nabla_y \phi}{\mu} \cdot \frac{z}{\abs{z}}\frac{\eta'(\abs{z})}{\mu R}) \\\nonumber
		&+\mu_0^{-1/2}S_{\text{in}}\qty(1-\eta_R)+\mu_0^{-1/2}S_{\text{out}}.
	\end{align*}
	We have $\eta''(y/R)\neq 0$ and $\eta'(y/R)\neq 0$ only if $\abs{y}\sim R$, hence we estimate
	\begin{align}\label{termphiDeltaeta}
		\abs{\mu^{-1}\qty(\frac{\mu}{\mu_0})^{1/2}\phi \Delta \eta}
		&\lesssim \mu^{-1}\norm{\phi}_{*}\mu^{1+l_1}\qty[\frac{R^2 \log(R)}{1+\abs{y}^3}+\frac{R^3}{1+\abs{y}^4}] \frac{\mu^{-2}}{1+\abs{y}^2}\eta''\qty(\abs{\frac{x-\xi}{\mu R}})\\\nonumber
		&\lesssim \mu^{l_1+\delta}\log(R)\frac{\mu^{-2}}{1+\abs{y}^2}\\\nonumber
		&\lesssim e^{-\kappa t_0}\mu^{l_1+\delta-\sigma}\frac{\mu^{-2}}{1+\abs{y}^{2+\alpha}}.
	\end{align}
	Using the bound on the gradient given in the definition of $\norm{\phi}_{*}$ we obtain
	\begin{align*}
		\abs{\mu^{-1}\qty(\frac{\mu}{\mu_0})^{1/2}\qty( 2\frac{\nabla_y \phi}{\mu} \cdot \frac{z}{\abs{z}}\frac{\eta'(\abs{z})}{\mu R}) }
		&\lesssim \mu^{-1} \mu^{1+l_1}\qty[\frac{R^2 \log(R)}{1+\abs{y}^4}+\frac{R^3}{1+\abs{y}^5}]\qty(\frac{\abs{\eta'(\abs{z})}}{\mu^2 R})\\
		&\lesssim \norm{\phi}_{*}\mu^{l_1+\delta}\log(R)\frac{\mu^{-2}}{1+\abs{y}^2}\\
		&\lesssim e^{-\kappa t_0}\mu^{l_1+\delta-\sigma}\frac{\mu^{-2}}{1+\abs{y}^{2+\alpha}}.
	\end{align*}
	Similarly, also using the bounds on $\dot \Lambda,\dot \xi$ we have
	\begin{align*}
		\abs{\mu^{-1}\qty(\frac{\mu}{\mu_0})^{1/2}\phi \pp_t \eta_R}
		&\lesssim \norm{\phi}_{*} \mu^{l_1+\delta}\log(R) \abs{\eta'(\abs{z})\frac{z}{\abs{z}}\cdot \qty(-\frac{\dot \xi}{\mu R}-z \frac{\pp_t(\mu R)}{\mu R})}\\
		&\lesssim  \mu^{l_1+\delta}\log(R) \frac{\abs{\eta'(\abs{z})}}{\mu^2 R^2}\\
		&\lesssim e^{-\kappa t_0}\mu^{l_1+\delta-\sigma}\frac{\mu^{-2}}{1+\abs{y}^{2+\alpha}}.
	\end{align*}
	Also, since $\delta<1/3$, we have 
	\begin{align*}
		\abs{(\mu \mu_0)^{-1/2}\eta_R\qty{\qty(\g-\dot \Lambda)(\phi+2 \nabla_y \phi \cdot y)  -\nabla_y \phi\cdot \qty(\frac{\dot \xi}{\mu})}}
		&\lesssim \mu^{-1} \mu^{1+l_1}\qty[\frac{R^2 \log(R)}{1+\abs{y}^4}+\frac{R^3}{1+\abs{y}^5}] \\
		&\lesssim\mu^{l_1}\norm{\phi}_* R^3\\
		&\lesssim \mu^{l_1+\delta}\frac{\mu^{-2}}{R^2}\\
		&\lesssim e^{-\kappa t_0}\mu^{l_1+\delta-\sigma}\frac{\mu^{-2}}{1+\abs{y}^{2+\alpha}}.
	\end{align*}
	Furthermore, using Lemma \ref{Lemma:FinalAnsatzError} we estimate
	\begin{align*}
		\abs{\mu^{-1/2}\Sout}&\lesssim \mu\\
		&\lesssim \mu^{l_1+\delta}\frac{\mu^{-2}}{R^2}\\
		&\lesssim e^{-\kappa t_0}\mu^{l_1+\delta-\sigma}\frac{\mu^{-2}}{1+\abs{y}^{2+\alpha}}.
	\end{align*}
	and
	\begin{align*}
		\abs{\mu^{-1/2}\Sinn (1-\eta_R)}&\lesssim \mu^{l_1+2\delta} \frac{\mu^{-2}}{1+\abs{y}^2}\\
		&\lesssim e^{-\kappa t_0}\mu^{l_1+\delta-\sigma}\frac{\mu^{-2}}{1+\abs{y}^{2+\alpha}}.
	\end{align*}
	Finally, since $\norm*{\mu_0^{-1/2}\psi_A}_{**}$ is bounded we get
	\begin{align}\label{term-nlouter}
		\abs{\mu_0^{-1/2}\mathcal{N}_3(u_3,\tilde \phi)(1-\eta_R)}\lesssim& \mu^{-1/2} u_3^3 \qty(\mu^{-1/2}\phi \eta_R +\mu^{1/2}\psi)^2(1-\eta_R)\\\nonumber
		\lesssim& \frac{\mu^{-2}}{1+\abs{y}^3}\qty(\mu^{-1}\abs{\phi}^2 \eta_R^2 +\mu_0\abs*{\mu_0^{-1/2}\psi_A}^2)(1-\eta_R)\\\nonumber
		\lesssim& \frac{\mu^{-2}}{1+\abs{y}^{2+\alpha}}\bigg[	\mu^{-1}\norm{\phi}_*^2  \mu^{2(1+l_1)}\qty(R^{-1}\log(R))^2\\\nonumber&\quad \quad \quad \quad\quad+\norm*{\mu_0^{-1/2}\psi_A}_{**}^2 \mu \frac{\mu^{2(l_1+\delta-\sigma)}}{R^{2\alpha}} 	\bigg]\\\nonumber
		&\lesssim e^{-\kappa t_0}\mu^{l_1+\delta-\sigma}\frac{\mu^{-2}}{1+\abs{y}^{2+\alpha}}.\nonumber
	\end{align}
	Summing up these estimates we conclude that 
	\begin{align*}
		\abs{f(x,t)}\lesssim   e^{-\kappa t_0}\mu^{l_1+\delta-\sigma}\frac{\mu^{-2}}{1+\abs{y}^{2+\alpha}}.
	\end{align*}
	Hence, we have
	\begin{align*}
		\abs{F(x,t)}\lesssim   e^{-\kappa t_0}\mu^{\frac{1}{2}+l_1+\delta-\sigma}\frac{\mu^{-2}}{1+\abs{y}^{2+\alpha}},
	\end{align*}
	and Lemma \ref{LinearLemmapsi} gives
	\begin{align*}
		\abs{T[F(\psi_A),0,0]}\lesssim e^{-\kappa t_0}\mu^{\frac{1}{2}+l_1+\delta-\sigma}\frac{1}{1+\abs{y}^{\alpha}}.
	\end{align*}
	Since $F\in L^\infty(\Om \times [t_0,\infty))$, classic parabolic estimates give $\psi\in C^{1+\hat \sigma,\frac{1+\hat \sigma}{2}}(\Om \times [t_0,\infty))$ for any $\hat \sigma<1$ and from Lemma \ref{LinearLemmapsi} we get
	\begin{align}\label{bddphinorm}
		\norm*{\mu_0^{-1/2}\psi_A}_{**}\leq M e^{-k t_0}
	\end{align}
	for sufficiently large $M$ (independent of $t_0$). This proves $\AAA(\psi_A)\in \BBB$.
	Now, we claim that the map $\AAA(\psi)$ is a contraction, that is: there exists $\bold  c<1$ such that, for any $\psi_A^{(1)},\psi_A^{(2)} \in \BBB$,
	\begin{align*}
		\norm{\mu_0^{-1/2}\AAA(\psi_A^{(1)})-\mu_0^{-1/2}\AAA(\psi_A^{(2)})}_{**}\leq \bold c \norm{\mu_0^{-1/2}\psi_A^{(1)}-\mu_0^{-1/2}\psi_A^{(2)}}_{**}.
	\end{align*}
	Since $\psi$ appears in $F(\psi)$ only in the nonlinear term $\NNN$, we get
	\begin{align*}
		\AAA(\psi_A^{(1)})-\AAA(\psi_A^{(2)})=T\bigg[&\NNN\qty(u_3,\psi_A^{(1)}+\psi_{B}+\mu^{-1/2}\phi\eta_R)\\
		&-\NNN\qty(u_3,\psi_A^{(2)}+ \psi_{B}+\mu^{-1/2}\phi\eta_R),0,0\bigg].
	\end{align*}
	From definition (\ref{DefNLterm}) we write
	\begin{align*}
		&\NNN\qty(u_3,\psi_A^{(1)}+\psi_{B}+\mu^{-1/2}\phi\eta_R)- \NNN\qty(u_3,\psi_A^{(2)}+ \psi_{B}+\mu^{-1/2}\phi\eta_R)		\\
		=&\bigg[\qty(u_3+\psi_A^{(1)}+\psi_{B}+\mu^{-1/2}\phi\eta)^5 - (u_3 +\psi_A^{(2)}+\psi_{B} + \mu^{-1/2}\phi\eta)^5 \\\nonumber
		 &- 5u_3^{4}(\psi^{(1)}-\psi^{(2)})\bigg]\\
		=&\bigg[\qty(u_3+\psi_A^{(1)}+\psi_{B}+\mu^{-1/2}\phi\eta)^5-\qty(u_3+\psi_A^{(2)}+\psi_{B}+\mu^{-1/2}\phi\eta)^5 \\& -5(u_3+\mu^{-1/2}\phi\eta)^4\mu_0^{1/2}\qty(\psi^{(1)}-\psi^{(2)})\bigg]\\
		&+ 5 \qty[(u_3+\mu^{-1/2}\phi\eta)^4-u_3^4](\psi^{(1)}-\psi^{(2)})\\
		=:&N_1+N_2.
	\end{align*}
	We estimate 
	\begin{align*}
		\abs{N_1(x,t)}&\lesssim \mu^{-3/2}U^3\mu_0 \abs{\mu_0^{-1/2}(\psi_A^{(1)}-\psi_A^{(2)})}^2\\
		&\lesssim \mu^{-1/2}U^3 \abs{\mu_0^{-1/2}(\psi_A^{(1)}-\psi_A^{(2)})}^2\\
		&\lesssim \frac{\mu^{-1/2}}{1+\abs{y}^{3}}\frac{\mu^{2\qty(l_1+\delta-\sigma)}}{1+\abs{y}^{2\alpha}}\norm{\mu_0^{-1/2}(\psi_A^{(1)}-\psi_A^{(2)})}_{**}^2\\
		&\lesssim e^{-\kappa t_0} \mu^{\frac{1}{2}+l_1+\delta-\sigma}\frac{\mu^{-2}}{1+\abs{y}^{2+\alpha}}\norm{\mu_0^{-1/2}(\psi_A^{(1)}-\psi_A^{(2)})}_{**}
	\end{align*}
	and
	\begin{align*}
		\abs{N_2(x,t)}&\lesssim u_3^3 \mu_0^{-1/2}\phi \eta \mu_0^{\frac{1}{2}}\abs{\mu_0^{-1/2}(\psi_A^{(1)}-\psi_A^{(2)})}\\
		&\lesssim \frac{\mu^{-3/2}}{1+\abs{y}^3} \norm{\phi}_* \mu^{1+l_1}R^3 \frac{\mu^{l_1+\delta-\sigma}}{1+\abs{y}^\alpha}\norm{\mu_0^{-1/2}(\psi_A^{(1)}-\psi_A^{(2)})}_{**}\\
		&\lesssim e^{-\kappa t_0} \mu^{\frac{1}{2}+l_1+\delta-\sigma}\frac{\mu^{-2}}{1+\abs{y}^{2+\alpha}}\norm{\mu_0^{1/2}(\psi_A^{(1)}-\psi_A^{(2)})}_{**}
	\end{align*}
	Finally, using $\beta=1/2+l_1+\delta-\sigma<3/2$ we apply $T[\cdot,0,0]$ to $F(\psi)$ we obtain
	\begin{align}\label{AAAcontr}
		\abs{\AAA[\psi_A^{(1)}]-\AAA[\psi_A^{(2)}]}\lesssim e^{-\kappa t_0} \mu^{\frac{1}{2}+l_1+\delta-\sigma}\frac{\mu^{-2}}{1+\abs{y}^{\alpha}}\norm{\mu_0^{-1/2}(\psi^{(1)}-\psi^{(2)})}_{**}.
	\end{align}
	Arguing as in (\ref{bddphinorm}), from (\ref{AAAcontr}) and standard parabolic estimates we obtain
	\begin{align*}
		\norm{\mu_0^{-1/2}(\AAA[\psi_A^{(1)}]-\AAA[\psi_A^{(2)}])}_{**}\leq \bold c  \norm{\mu_0^{-1/2}(\psi_A^{(1)}-\psi_A^{(2)})}_{**},
	\end{align*}
	with $\bold c<1$ if $t_0$ is taken sufficiently large. Applying the Banach fixed point theorem we get existence and uniqueness of $\psi_A$ and hence of $\psi=\mu_0^{-1/2}(\psi_A+\psi_B)$ with estimate (\ref{estpsi}) that is a consequence of estimates (\ref{LinEstPsi1inf})-(\ref{LinEstPsi1Calpha}).
	\end{proof}
	
	\begin{remark}[Continuity with respect to the initial condition $\psi_0$]\label{rem:continuityPsi}
		Given an initial datum $\psi_0$ Proposition \ref{PropOuterNonlinear} defines a solution $\psi=\Psi[\psi_0]$ to (\ref{OutPr}),
		from a small neighborhood of $0$ in the $L^\infty(\Om)$ space with the $C^1$-norm $\norm{\psi}_\infty+\norm{\nabla\psi_0}_\infty$ into the Banach space $L^\infty$ with norm $\norm{\psi}_{**}$ defined in (\ref{normPsi}). In fact, from the proof of Proposition \ref{PropOuterNonlinear} and the implicit function theorem, $\psi_0 \mapsto \Psi[\psi_0]$ is a diffeomorphism and hence
		\begin{align*}
			\norm{\Psi[\psi_0^1]-\Psi[\psi_0^2]}_{**}\leq c \qty[\norm{\psi_0^1-\psi_0^2}_\infty+\norm{\nabla_x\psi_0^1-\nabla_x\psi_0^2}_\infty],
		\end{align*}
		for some positive constant $c$.
	\end{remark}
	The function $\psi=\Psi[\Lambda,\dot \Lambda,\xi,\dot \xi,\phi]$ depends continuously on the parameters $\Lambda,\dot \Lambda,\xi,\dot \xi,\phi$. To see this we argue similarly to \cite[Proposition 4.3]{dmw1}. For example, fix $\dot \Lambda,\xi,\dot \xi,\phi$ and consider 
	$$
	\bar \psi{\coloneqq}\psi^{(1)}-\psi^{(2)}\quad \text{where}\quad \psi^{(i)}=\Psi[\Lambda_i,\dot \Lambda,\xi,\dot \xi,\phi],\quad \text{for}\quad  i=1,2
	$$ 
	for $\Lambda_1,\Lambda_2$ satisfying (\ref{NormsLambda,dotLambda}). Then $\bar \psi$ solves
	\begin{align*}
		\pp_t \bar \psi = \Delta \bar \psi+\gamma \bar \psi+ V[\Lambda_1]\bar \psi+\qty(V[\Lambda_1]-V[\Lambda_2])\psi^{(2)}+F[\Lambda_1]-F[\Lambda_2].
	\end{align*}
	One can easily check each term in $F$ and obtain
	\begin{align*}
		\norm{F[\Lambda_1]-F[\Lambda_2]}_{\beta-2,\alpha+2}\leq \bold c \norm{\Lambda_1-\Lambda_2}_{l_0,\infty},
	\end{align*}
	with $\bold c<1$ if $t_0$ is large enough. Also, using (\ref{VBound}) we find that
	\begin{align*}
		\norm{V[\Lambda_1]-V[\Lambda_2]}_{\beta-2,\alpha+2}\leq \bold  c \norm{\Lambda_1-\Lambda_2}_{l_0,\infty}.
	\end{align*}
	Then, arguing as in the proof of (\ref{LinEstPsi1inf}), a multiple of $\norm{\Lambda_1-\Lambda_2}_{\sharp,l_0,\delta_0,\frac{1}{2}+\ve} \ppsi$, where $\ppsi$ is the supersolution constructed in Lemma \ref{LinearLemmapsi}, is a supersolution for $\bar \psi$.
	Similarly, one obtain analogue estimates fixing $\xi,\dot \Lambda, \dot \xi$. Let us consider all the parameters fixed. We define $\bar \psi\coloneqq\psi[\Lambda,\dot \Lambda,\xi,\dot \xi,\phi_1]-\psi[\Lambda,\dot \Lambda,\xi,\dot \xi,\phi_2]$, which satisfies the equation
	\begin{align*}
		\pp_t \bar \psi = \Delta \bar \psi +V\bar \psi + F[\phi_1]-F[\phi_2].
	\end{align*}
	For instance, we estimate
	\begin{align*}
		\abs{	\mu^{-1}\qty(\frac{\mu}{\mu_0})^{1/2}\qty(\phi_1-\phi_2) \pp_t \eta_R }\lesssim & \norm{\phi_1-\phi_2}_*\mu^{l_1}R^{-1}\log(R)  \mu^2 R^{2+\alpha}\frac{\mu^{-2}}{1+\abs{y}^{2+\alpha}}\\
		\lesssim& \,\bold c\norm{\phi_1-\phi_2}_{*} \mu^{l_1+\delta-\sigma}\frac{\mu^{-2}}{1+\abs{y}^{2+\alpha}},
	\end{align*}
	with $\bold c<1$ when $t_0$ is fixed large enough, and arguing as in (\ref{termphiDeltaeta})-(\ref{term-nlouter}), we obtain similar estimate on the other terms of $F[\phi_1]-F[\phi_2]$. 
	Having the $L^\infty$-bound, the estimate for the gradient and the H\"older norms of $\bar \psi$ follow as in the proof of Lemma \ref{PropOuterNonlinear}. We summarize the continuity of $\psi[\Lambda,\dot \Lambda,\xi,\dot \xi,\phi]$ with respect to the parameters in the following Proposition.
	\begin{proposition}\label{PropertiesPsi}
		Under the same assumption of Proposition \ref{PropOuterNonlinear}, the function $\psi=\Psi[\Lambda,\dot\Lambda,\xi,\dot \xi,\phi]$ is continuous with respect to the parameters $\Lambda,\dot\Lambda,\xi,\dot \xi,\phi$. Moreover the following estimate holds:
		\begin{align*}
			\lVert \Psi[\Lambda^{(1)},{\dot \Lambda}^{(1)},&\xi^{(1)},{\dot \xi}^{(1)}, \phi^{(1)}]-\Psi[\Lambda^{(2)},{\dot \Lambda}^{(2)},\xi^{(2)},{\dot \xi}^{(2)}, \phi^{(2)}]\rVert_{**}\\
			\leq \bold c &\bigg\{\norm*{\Lambda^{(1)}-\Lambda^{(2)}}	_{\sharp,l_0,\delta_0,\frac{1}{2}+\ve}
			+\norm*{\dot \Lambda^{(1)}-\dot \Lambda^{(2)}}_{\sharp,l_1,\delta_1,\ve}\\
			&+\norm*{\xi_1^{(1)}-\xi_1^{(2)}}_{\sharp,1+l_1,\frac{1}{2}+\ve}	
			+\norm*{{\dot \xi}_1^{(1)}-{\dot\xi}_1^{(2)}}_{\sharp,1+l_1,\ve}
			+\norm*{\phi^{(1)}-\phi^{(2)}}_*\bigg\}	
		\end{align*}
		where $\bold c<1$ provided that $t_0$ is sufficiently large and the constants $\mathfrak{b}_1,\mathfrak{b}_2$ in (\ref{NormsLambda,dotLambda}),(\ref{NormsXi,dotXi}) are sufficiently small.
	\end{proposition}

	\section{Characterization of the orthogonality conditions (\ref{InnPr})}\label{sec:Choiceparameters}
	Given the function $\psi=\Psi[\Lambda,\dot\Lambda,\xi,\dot \xi,\phi]$ provided by Proposition \ref{PropOuterNonlinear}, we plug it in the inner problem for $\phi$. From the linear theory stated in Proposition \ref{propInnerlineartheory}, the inner problem (\ref{InnPr}) with initial datum (\ref{InnProbGeneral}) can be solved if the orthogonality conditions
	\begin{align}\label{OrthCondFin}
		\int_{B_{2R}} H[\Lambda,\dot \Lambda,\xi,\dot \xi,\phi](y,t(\tau))Z_i(y)\ddy =0 \quad \text{for}\quad t>t_0,\quad \text{and}\quad i=1,2,3,4,
	\end{align} 
	are satisfied.
	The aim of this section is to characterize this set of conditions as an nonlocal system in $\Lambda,\xi$ for fixed $\phi\in X_*$. 
	The next lemma shows that the orthogonality condition with index $i=4$ is equivalent to a nonlocal equation in the variable $ \Lambda$, for fixed $\phi,\xi$.
	\begin{lemma}\label{Lemma:Chari=4}
		Assume that $\Lambda,\xi,\phi$ satisfy (\ref{NormsLambda,dotLambda}),(\ref{NormsXi,dotXi}) and (\ref{NormphiDef}) respectively. Let $\psi=\Psi[\Lambda,\dot \Lambda,\xi,\dot \xi,\phi]$ be the solution to problem (\ref{OutPr}) given by Proposition \ref{PropOuterNonlinear}. Then, the condition (\ref{OrthCondFin}) with index $i=4$ is equivalent to
		\begin{align}\label{i=4Equivalent}
			(1+ a[\dot \Lambda,\xi](t))\JJJ[\dot \Lambda](0,t)=g(t)+G[\Lambda,\dot \Lambda,\xi,\dot \xi,\phi](t) \quad \text{for}\quad t \in [t_0,\infty),
		\end{align}
		where $\JJJ$ is the solution to
		\begin{align*}
			&\pp_t \JJJ=\Delta \JJJ+\g \JJJ -\dot \Lambda(t) G_\g(x,0) \inn \Om \times [t_0-1,\infty),\\
			&\JJJ(x,t)=0 \onn \pp \Om \times [t_0-1,\infty),\\
			&\JJJ(x,t_0-1)=0 \inn \Om.
		\end{align*}
		The function $a$ is smooth, decays as $t\to 0$ and $a[0,0]\equiv0$. Then, for $\kappa\in (0,2\gamma(\sigma-\alpha \delta))$, the following estimates on $g$ and $G$ hold:
		\begin{align*}
			\norm{g}_{\sharp,l_0,\delta_1,\frac{1}{2}+\ve}+
			\norm{g}_{\sharp,l_0,\delta_0,\ve}\leq C_0 e^{-\kappa t_0} ,
		\end{align*}
		and
		\begin{align}\label{CharPropG1}
			\lVert G[\Lambda,\dot \Lambda,\xi,\dot \xi,\phi] \rVert_{\sharp,l_0,\delta_1,\frac{1}{2}+\ve}+&\lVert G[\Lambda,\dot \Lambda,\xi,\dot \xi,\phi] \rVert_{\sharp,l_0,\delta_0,\ve}\leq	
			e^{-\kappa t_0}\big\{\norm{\Lambda}_{\sharp,l_0,\delta_0,\ve}\\\nonumber
			&+\norm*{\dot \Lambda}_{\sharp,l_1,\delta_1,\ve}+\norm{\xi_1}_{\sharp,1+l_1,\frac{1}{2}+\ve}+\norm*{\dot \xi_1}_{\sharp,1+l_1,\ve}+\norm{\phi}_*\big\}.
		\end{align}
		Furthermore, we have
		\begin{align}\label{CharPropG2}
			\lVert G[\Lambda^{(1)},\dot \Lambda^{(1)},\xi^{(1)},\dot \xi^{(1)},\phi^{(1)}]&-G[\Lambda^{(2)},\dot \Lambda^{(2)},\xi^{(2)},\dot \xi^{(2)},\phi^{(2)}] \rVert_{\sharp,l_0,\delta_1,\frac{1+2\ve}{2}}\\\nonumber
			\leq \bold c \big\{
			&\norm*{ \Lambda^{(1)}- \Lambda^{(2)}}_{\sharp,l_0,\delta_0,\frac{1+2\ve}{2}}+\norm*{ \dot \Lambda^{(1)}- \dot \Lambda^{(2)}}_{\sharp,l_1,\delta_1,\ve}\\\nonumber
			& +\norm*{\xi_{1}^{(1)}- \xi_{1}^{(2)}}_{\sharp,1+l_0,\frac{1}{2}+\ve}+ \norm*{\dot \xi_{1}^{(1)}- \dot \xi_{1}^{(2)}}_{\sharp,1+l_0,\ve}\\\nonumber
			&+ \norm*{\phi^{(1)}-\phi^{(2)}}_{*}\big\}
		\end{align}
		with constant $\bold c<1$ provided that $t_0$ is fixed sufficiently large and $\mathfrak{b}_i$ small for $i=1,2$.
	\end{lemma}
	\begin{proof}
		We recall that 
		\begin{align*}
			H[\phi,\psi,\mu,\dot \mu,\xi,\dot \xi](y,\tau){\coloneqq}
			&5\Uy^4 \mu \qty(\frac{\mu_0}{\mu})^{1/2} \psi(\mu y +\xi,t(\tau))\\\nonumber
			&+B_0\qty[\phi+\mu\psi](\mu y+\xi,t(\tau))+\mu^{5/2}S_{\text{in}}(\mu y+\xi,t(\tau))\\\nonumber
			&+\mathcal{N}(\mu^{1/2}u_3,\mu^{1/2}\tilde \phi)(\mu y+\xi,t(\tau)).
		\end{align*}
		Hence, (\ref{InnPrOrth}) with index $i=4$ becomes
		\begin{align*}
			0=&\mu^{5/2}\int_{B_{2R}}Z_4(y)S_{\text{in}}(y,t)  \ddy+\mu \qty(\frac{\mu_0}{\mu})^{1/2}\int_{B_{2R}} Z_4(y) 5U(y)^4 \psi(\mu y +\xi,t) \ddy \\
			&+\int_{B_{2R}}Z_4(y)B_0\qty[\phi+\mu\psi](\mu y+\xi,t) \ddy 
			+ \int_{B_{2R}}Z_4(y)\mathcal{N}(\mu^{1/2}u_3,\mu^{1/2}\tilde \phi)(\mu y+\xi,t(\tau))\ddy \\
			=&: \sum_{j=1}^{4}i_j(t).
		\end{align*}
		We follow the analogue \cite[Lemma 5.1]{dmw1} to estimate the terms $i_j(t)$. Firstly, we analyze $i_1$. We have
		\begin{align*}
			i_1(t)=&\mu^{5/2}\int_{B_{2R}}
			\Sinn(y,t) Z_4(y)\ddy\\
			=&\mu  \qty(\frac{\mu_0}{\mu})^{1/2} \int_{B_{2R}}5U(y)^4 J(\mu y+\xi,t)Z_4(y)\ddy\\ 
			&+ \int_{B_{2R}}\mathcal{N}_3(y,t) Z_4(y)\ddy   \\\nonumber
			&+\mu \int_{B_{2R}}Z_4(y)5U(y)^4 h_\g(\mu y+\xi,\xi)	\ddy 
			\\=:&a_1(t)+a_2(t)+a_3(t),
		\end{align*}
		where we used that the integral of $Z_4(y)U(y)^4\pp_{y_i} U(y)$ on $B_{2R}(0)$ is null by symmetry for $i=1,2,3$. Also,
		\begin{align*}
			\mu^{-1}  \qty(\frac{\mu_0}{\mu})^{-1/2} a_1(t)=& \intb 5U(y)^4 Z_4(y)J[\dl](\mu y +\xi,t) \ddy \\
			=& \JJJ[\dl](0,t) \intb 5U(y)^4 Z_4(y) \ddy \\
			&+\qty[ J[\dot \Lambda](\xi,t)-\JJJ[\dl](0,t)] \intb 5U(y)^4 Z_4(y) \ddy\\
			&+ \intb 5U(y)^4 Z_4(y)[J[\dl](\mu y +\xi,t)-J[\dl](\xi,t)] \ddy\\
			=:&a_{11}[\dot \Lambda](t)+a_{12}[\dot \Lambda,\xi](t)+a_{13}[\dot \Lambda,\xi](t).
		\end{align*}
		The main term in the left-hand side of (\ref{i=4Equivalent}) is given by
		\begin{align*}
			\mu^{-1}  \qty(\frac{\mu_0}{\mu})^{-1/2}	c_1^{-1}(1+O(R^{-2}))^{-1}a_{11}(t)=\JJJ[\dot \Lambda](0,t),
		\end{align*}
		where $c_1(1+O(R^{-2}))=\int_{B_{2R}} 5U^4 Z_4 \ddy$. To analyze the terms $a_{12}$, we decompose $w[\dot \Lambda](x,t)=J[\dot \Lambda](x,t)-\JJJ[\dot \Lambda](x,t)$ as a sum of a solution in $\RR^3$ and a smooth one in $\Omega$ with more decay. Then, using the Duhamel's formula in $\RR^3$ as in \cite[Proof of (7.5)]{dmw1} we deduce
		\begin{align}\label{estPertJ}
			\norm*{a_{12}[\dot \Lambda,\xi]}_{\sharp,l_0,\delta_1,\frac{1}{2}+\ve}+\norm*{a_{12}[\dot \Lambda,\xi]}_{\sharp,l_0,\delta_0,\ve}\lesssim e^{-\kappa t_0}\big\{\norm*{ \Lambda}_{\sharp,l_0,\delta_0,\frac{1}{2}+\ve}&+\norm*{\dot \Lambda}_{\sharp,l_1,\delta_1,\ve}\\\nonumber&+\norm{\xi_1}_{\sharp,1+l_1,\frac{1}{2}+\ve}\big\}. 
		\end{align}
		We analyze $a_{13}$ by splitting $J$ as a sum of a solution to the same equation in $\RR^3$ and smooth remainder in $\Om$ with more decay, and, proceeding as in \cite[Proof of (5.10)]{dmw1}, again by Duhamel's formula in $\RR^3$ we obtain
		\begin{align*}
			\abs{J[\dl](\mu y +\xi,t)-J[\dl](\xi,t)}= \abs{y\mu }^{\sigma} \Pi[\dot \Lambda,\xi](t)\theta(\abs{y}),
		\end{align*}
		for some $\sigma \in (0,1)$ and bounded smooth function $\theta$, and $\Pi[\dot \Lambda,\xi]$ satisfying the estimate above for $a_{12}$. After integration, $a_{13}[\dot \Lambda,\xi](t)$ satisfies (\ref{estPertJ}). 
		Taking into account the behavior of $J_1,J_2$ and $\phi_3$ given in (\ref{estJ_{1,0}-Linf}), (\ref{estonJ2}) and (\ref{bphi3}) respectively, we have
		\begin{align*}
			a_2&= \int_{B_{2R}} Z_4(y)\mathcal{N}_3(y,t)\ddy\\
			&= \int_{B_{2R}}Z_4 \qty{10 \qty(U(y)+s(-\mu H_\g+\mu  J+\mu^{-1/2}\phi_3 \eta_l))^3\qty(-\mu H_\g+\mu J+\mu^{-1/2}\phi_3 \eta_l)^2}\ddy\\
			&=\mu^2 \int_{B_{2R}} 10 Z_4(y)U(y)^3 Q[\Lambda,\dot \Lambda,\xi](y,t)\ddy ,
		\end{align*}
		for some constant $s\in (0,1)$ and bounded smooth function $Q[\Lambda,\xi ](y,t)$ satisfying (\ref{estPertJ}). \\
		Finally, Taylor expanding $h_\g(x,\xi)$ at $x=\xi$, we get
		\begin{align*}
			a_3=&\mu \intb Z_4(y)5U(y)^4  h_\g(\mu y+\xi,\xi)\ddy \\
			=& \mu R_\g(\xi) \int_{B_{2R}} Z_4(y) 5U(y)^4 \ddy\\
			&+\mu^3 \intb Z_4(y)5U(y)^4 \qty(y\cdot D_{xx}h_\g(\mu y^*(y)+\xi,\xi)\cdot y) \ddy  \\
			=&\mu^2 \Pi_2[\Lambda,\xi](t),
		\end{align*}
		for some $y^* \in \overline{[0,y]}$ and a smooth bounded function $\Pi_2(t)$. The term $\mu^{-1/2}\mu_0^{-1/2}i_1(t)=\sum_{i=1}^{3}a_i(t)$ gives the left-hand side in (\ref{i=4Equivalent}). Now, we look at $i_2$. We decompose
		\begin{align*}
			\mu^{-\frac{1}{2}}\mu_0^{-\frac{1}{2}}i_2(t)=&\intb Z_4(y) 5U(y)^4 \psi[\Lambda,\xi,\dot \Lambda,\dot \xi,\phi](\mu y +\xi,t) \ddy\\
			=&\psi[0,0,0,0,0](0,t)\intb Z_4(y) 5U(y)^4  \ddy\\
			&+\intb Z_4(y) 5U(y)^4\qty{\psi[0,0,0,0,0](\mu y+\xi,t)-\psi[0,0,0,0,0](0,t)}  \ddy\\
			&+\intb Z_4(y) 5U(y)^4\qty{\psi[\Lambda,\xi,\dot \Lambda,\dot \xi,\phi](\mu y +\xi,t)-\psi[0,0,0,0,0](\mu y+\xi,t)}  \ddy\\
			\eqqcolon&b_1(t)+b_2[\Lambda,\xi](t)+b_3[\Lambda,\xi,\dot \Lambda,\dot\xi,\phi](t),
		\end{align*}
		The term
		\begin{align*}
			b_1(t)=\psi[0,0,0,0,0](0,t)\int_{B_{2R}} 5U(y)^4 Z_4(y)\ddy,
		\end{align*}
		is independent of parameters and, as a consequence of Proposition \ref{PropOuterNonlinear}, satisfies the estimate
		\begin{align*}
			&\norm{b_1}_{\sharp,l_0,\delta_0,\ve}+\norm{b_1}_{\sharp,l_0,\delta_1,\frac{1}{2}+\ve}\leq  C e^{-\kappa t_0}.
		\end{align*}
		Applying the mean value theorem to $\psi$ and using the gradient estimate we deduce the same bound for $b_2$. This gives the main term $b_1(t)+b_2(t)=g(t)$ in the right-hand side of (\ref{i=4Equivalent}).
		We analyze $b_3(t)$. By Proposition \ref{PropertiesPsi} applied to
		\begin{align*}
			\psi[ \Lambda,\xi,\dot \Lambda,\dot \xi,\phi]-\psi[0,0,0,0,0]
		\end{align*}
		we obtain
		\begin{align*}
			\norm{b_3}_{\sharp,l_0,\delta_1,\frac{1}{2}+\ve}+\norm{b_3}_{\sharp,l_0,\delta_0,\ve}\lesssim &e^{-\kappa t_0}\bigg\{\norm{\Lambda}_{\sharp,l_0,\delta_0,\frac{1}{2}+\ve}+\norm*{\dot \Lambda}_{\sharp, l_1,\delta_1,\ve}\\
			&+\norm*{\dot \xi_1}_{\sharp,1+l_1,\ve}+\norm{ \xi_1}_{\sharp,1+l_1,\frac{1}{2}+\ve}+\norm{\phi}_*\bigg\}.
		\end{align*}
		Also, again as a consequence of the Lipschitz estimates in $\psi$ we have for example
		\begin{align*}
			\norm*{b_3[\dot \Lambda_1]-b_3[\dot \Lambda_2]}_{\sharp,l_0,\delta_1,\frac{1}{2}+\ve}\leq \bold c \norm*{\dot \Lambda_1- \dot \Lambda_2}_{\sharp, l_1,\delta_1,\ve},
		\end{align*}
		for any $\dot \Lambda_1,\dot \Lambda_2 \in X_{\sharp,l_1,\delta_1,\ve}$ and fixed $\Lambda,\xi,\dot \Lambda,\dot \xi$ in the respective spaces. 
		We analyze $i_3$. We recall that 
		\begin{align*}
			B_0[\phi+\mu  \psi]=5\qty[\qty(U-\mu H_\gamma+\mu J[\dot \Lambda]+\mu^{-1/2}\phi_3(y,t)\eta_3)^4-U^4][\phi+\mu \psi],
		\end{align*}
		which is linear in $\phi,\psi$ and satisfies
		\begin{align*}
			\abs{B_0[\phi+\mu \psi](\mu y+\xi,t)}\lesssim \frac{\mu}{1+\abs{y}^{3}}\mu \abs{\phi+\mu \psi}.
		\end{align*}
		It follows that
		\begin{align*}
			\abs{i_3(t)}\lesssim &  \mu {\norm{\phi}_* \mu^{1+l_1}{R^3 }+\norm{\psi}_{**}\mu^2 \mu^{l_1+\delta-\sigma}}\\
			\lesssim &e^{-\kappa t_0}\mu^{l_0}.
		\end{align*}
		Then, the H\"older bounds on $\psi$ and $\phi$ in the respective norms give estimate (\ref{CharPropG1}) for $i_3$, and using Proposition \ref{PropertiesPsi} we also get the Lipschitz property (\ref{CharPropG1}) for $i_3$. 
		Finally, we have
		\begin{align*}
			\abs{\mathcal{N}(\mu^{1/2}u_3,\mu^{1/2}\tilde \phi)(\mu y+\xi,t(\tau))}&\lesssim \frac{1}{1+\abs{y}^3}\qty(\phi+\mu \psi)^2\\
			&\lesssim \frac{1}{1+\abs{y}^3}\qty(\norm{\phi}_*^2 \mu^{2(1+l_1)}R^6+\norm{\psi}_{**}^2\mu^2 \frac{\mu^{2(l_1+\delta-\sigma)}}{1+\abs{y}^{2\alpha}}	)\\
			&\lesssim e^{-\kappa t_0}\mu^{l_0},
		\end{align*}
		and (\ref{CharPropG1})-(\ref{CharPropG1}) for $i_4$ follows arguing as for $i_3$. Summing up the estimates we obtain $G[\Lambda,\dot \Lambda,\xi,\dot \xi,\phi](t)=b_3+i_3+i_4$ as in (\ref{i=4Equivalent}) with the properties (\ref{CharPropG1}) and (\ref{CharPropG2}).
	\end{proof}
	Now, we characterize the conditions
	\begin{align*}
		\int_{B_{2R}} Z_i(y) H[\Lambda,\dot \Lambda,\xi,\dot \xi,\phi](y,t)\ddy =0, \quad\text{for}\quad t\in (t_0,\infty) \quad \text{and}\quad i=1,2,3.
	\end{align*}
	This characterization is given in the following lemma, whose proof, similar to the one of Lemma \ref{Lemma:Chari=4}, is omitted.
	\begin{lemma}\label{Lemma:Chari=123}
		The relation (\ref{OrthCondFin}) for $i=1,2,3$ is equivalent to
		\begin{align}\label{Eqi=123Char}
			\dot \xi_{1,i}=\mathfrak{c}_i \mu_0(t)^{1+l_1}+\Theta_i[\Lambda,\dot \Lambda,\xi,\phi](t)
		\end{align}
		for smooth bounded function $\Theta$ which satisfies
		\begin{align}\label{smallnessTheta}
			\lVert \Theta[\Lambda,\dot \Lambda,\xi,\phi] \rVert_{\sharp,1+l_1,\frac{1}{2}+\ve}+&\lVert \Theta[\Lambda,\dot \Lambda,\xi,\phi] \rVert_{{\sharp,1+l_1,\ve}}\leq	
			e^{-\kappa t_0}\big\{\norm*{\Lambda}_{\sharp,l_0,\delta_0,\ve}\\\nonumber
			&+\norm*{\dot \Lambda}_{\sharp,l_1,\delta_1,\ve}+\norm{\xi_1}_{\sharp,1+l_1,\frac{1}{2}+\ve}+\norm{\phi}_*\big\}.
		\end{align}
		Furthermore, we have
		\begin{align}\label{CharPropT2}
			\lVert \Theta[\Lambda^{(1)},\dot \Lambda^{(1)},\xi^{(1)},\dot \xi^{(1)},\phi^{(1)}]&-\Theta[\Lambda^{(2)},\dot \Lambda^{(2)},\xi^{(2)},\dot \xi^{(2)},\phi^{(2)}] \rVert_{\sharp,1+l_1,\frac{1}{2}+\ve}\\\nonumber
			\leq \bold c \big\{
			&\norm*{ \Lambda^{(1)}- \Lambda^{(2)}}_{\sharp,l_0,\delta_0,\frac{1}{2}+\ve}+\norm*{ \dot \Lambda^{(1)}- \dot \Lambda^{(2)}}_{\sharp,l_1,\delta_1,\ve}\\\nonumber
			& +\norm*{\xi_{1}^{(1)}- \xi_{1}^{(2)}}_{\sharp,1+l_1,\frac{1}{2}+\ve}+ \norm*{\dot \xi_{1}^{(1)}- \dot \xi_{1}^{(2)}}_{\sharp,1+l_1,\ve}\\\nonumber
			&+ \norm*{\phi^{(1)}-\phi^{(2)}}_{*}\big\},
		\end{align}
		with constant $\bold c<1$ provided that $t_0$ is fixed sufficiently large and $\mathfrak{b}_i$ small for $i=1,2$.
	\end{lemma}
	
	\section{Choice of parameters $\Lambda,\xi$}\label{sec:solveSystem}
	In the previous section we have proved that if $\phi \in X_*$ and $\Lambda,\xi$ satisfy (\ref{NormsLambda,dotLambda}) and (\ref{NormsXi,dotXi}) then the system of orthogonality conditions
	\begin{align*}
		\int_{B_{2R}} H[\Lambda,\dot \Lambda,\xi,\dot \xi,\phi](y,t(\tau))Z_i(y)\ddy=0 \quad \text{for}\quad t\in [t_0,\infty) \quad \text{and}\quad i=1,2,3,4,
	\end{align*}
	is equivalent to the nonlocal system in $[t_0,\infty)$
	\begin{align}\label{SystemEquivalent}
		\begin{cases}
			(1+ a[\dot \Lambda,\xi](t))\JJJ[\dot \Lambda](0,t)=g(t)+G[\Lambda,\dot \Lambda,\xi,\dot \xi,\phi](t),& \\
			\dot \xi_{1,i}=\mathfrak{c}_i\mu_0(t)^2 \qty(1+\Theta_i[\Lambda,\dot \Lambda,\xi,\phi])& \text{for}\quad i=1,2,3,
		\end{cases}
	\end{align}
	with $g,G,a$ as in Lemma \ref{Lemma:Chari=4} and $\Theta_i$ as in Lemma \ref{Lemma:Chari=123}.
	Next, we verify that this system is solvable for $\Lambda,\xi$ satisfying (\ref{NormsLambda,dotLambda}),(\ref{NormsXi,dotXi}) respectively. This relies on the following proposition, proved in \S\ref{sec:invj}, about the solvability of the nonlocal operator $\JJJ[\dot \Lambda](0,t)=g(t)$ for $g$ as in (\ref{i=4Equivalent}).
	
	\begin{proposition}\label{Proposition:invj}
		Let $h:[t_0,\infty)\to \RR$ a function satisfying $\norm{h}_{\sharp,c_1,c_2,\ve}<\infty$ 
		for some constants $\ve>0$ and $c_1,c_2$ such that
		\begin{align}\label{condc1c2lambda1}
			0<c_2\leq c_1<\frac{\lambda_1-\g}{2\g}.
		\end{align}
		Then there exists a function $\Lambda\in C^{\frac{1}{2}+\ve}(t_0-1,\infty)$ satisfying
		\begin{align}\label{LinEstLambda<h}
			\JJJ[\dot \Lambda](0,t)=h(t) \inn (t_0,\infty),
		\end{align}
		where $\JJJ[\dot \Lambda]$ satisfies (\ref{ProbJ10}),
		and there exists a constant $C_1$ such that
		\begin{align}\label{estinvLambda}
			\norm{\Lambda}_{\sharp,c_1,c_2,\ve+\frac{1}{2}}\leq C_1 \norm{h}_{\sharp,c_1,c_2,\ve}.
		\end{align}
		Moreover, if $\norm{h}_{\sharp,c_1,c_2,\frac{1}{2}+\ve} <\infty$ then $\Lambda \in C^{1,\ve}(t_0-1,\infty)$ and there exists a constant $C_2$ such that
		\begin{align}\label{LinEstdotLambda<h}
			\norm*{\dot \Lambda}_{\sharp,c_1,c_2,\ve}\leq C_2 \norm{h}_{\sharp,c_1,c_2,\frac{1}{2}+\ve}.
		\end{align}
		Thus, the linear operators
		\begin{align}\label{T1InverseJLambda}
			T_1: &X_{\sharp,c_1,c_2,\ve}\to X_{\sharp,c_1,c_2,\ve+\frac{1}{2}}\\\nonumber
			&h(t)\mapsto \Lambda[h](t),
		\end{align}
		and
		\begin{align}\label{T2InverseJdotLambda}
			\hat T_1{\coloneqq}\frac{d}{dt}\circ T_1: &X_{\sharp,c_1,c_2,\frac{1}{2}+\ve}\to X_{\sharp,c_1,c_2,\ve}\\\nonumber
			&h(t)\mapsto \dot \Lambda[h](t),
		\end{align}
		are well-defined and continuous.
	\end{proposition}
	We are ready to solve the system (\ref{SystemEquivalent}) in $\Lambda,\xi$ for fixed $\phi \in X_*$.
	\begin{proposition}\label{prop:FindLambdaXi}
		Suppose that $\phi$ satisfies (\ref{Normphi}). Then, there exist $\Lambda=\Lambda[\phi](t)$ and $\xi=\xi[\phi](t)$ to the nonlinear nonlocal system (\ref{SystemEquivalent}) which satisfy (\ref{NormsLambda,dotLambda}) and (\ref{NormsXi,dotXi}) respectively. Moreover, they satisfy
		\begin{align}\label{LambdaofPhiLip}
			&\norm{\Lambda[\phi_1] -  \Lambda[\phi_2]}_{\sharp,l_0,\delta_0,\frac{1}{2}+\ve}\leq \bold c \norm{\phi_1-\phi_2}_{*},\\\nonumber
			&\norm{\dot \Lambda[\phi_1] - \dot \Lambda[\phi_2]}_{\sharp,l_1,\delta_1,\ve}\leq \bold c \norm{\phi_1-\phi_2}_{*},\\\nonumber
			&\norm{\xi[\phi_1]-\xi[\phi_2]}_{\sharp,1+l_1,\frac{1}{2}+\ve}\leq \bold c \norm{\phi_1-\phi_2}_{*},\\\nonumber
			&\norm{\dot \xi[\phi_1]-\dot \xi[\phi_2]}_{\sharp,1+l_1,\ve}\leq \bold c \norm{\phi_1-\phi_2}_{*},
		\end{align}
		with constant $\bold c<1$ provided that $t_0$ is fixed sufficiently large and $\mathfrak{b}_i$ small for $i=1,2$.
	\end{proposition}
	\begin{proof}
		Firstly, we observe that equation (\ref{i=4Equivalent}) can be rewritten as
		\begin{align*}
			\JJJ[\dot \Lambda](0,t)=g_1(t)+G_1[\dot \Lambda,\Lambda,\dot \xi,\xi,\phi](t),
		\end{align*}
		where
		\begin{align}\label{DefG1}
			g_1(t)+G_1[\dot \Lambda,\Lambda,\dot \xi,\xi]=(1+a[\dot \Lambda,\xi])^{-1}[g(t)+G[\dot \Lambda,\Lambda,\dot \xi,\xi,\phi]](t),
		\end{align}
		for new functions $g_1,G_1$ satisfying the same properties of $g,G$ in Lemma \ref{Lemma:Chari=4}. By Proposition \ref{Proposition:invj} we reduce the equation for $ \Lambda$ to a fixed point problem
		\begin{align*}
			\dot \Lambda(t)=\FFF_1[\dot \Lambda](t),\quad \FFF_1[\dot \Lambda](t)=\hat T_1\qty[g_1(t)+G_1[\Lambda,\dot \Lambda,\xi,\dot \xi,\phi]],
		\end{align*}
		where $\hat T_1$ is defined in (\ref{T2InverseJdotLambda}). 
		Let 
		\begin{align*}
			\dot \Lambda_0(t){\coloneqq}\hat T_1[g_1](t)
		\end{align*}
		and define the operator $\LLL_1[h]{\coloneqq}\hat T_1 [h-g_1]$. We use the notation
		\begin{align*}
			\LLL_1[h]=\uplambda[h](t){\coloneqq}\dot \Lambda[h](t)-\dot\Lambda_0(t),
		\end{align*}
		for any $h\in X_{\sharp,l_1,\delta_1,\frac{1}{2}+\ve}$. Observe that
		\begin{align*}
			\abs*{\dot \Lambda[h]}&=\abs*{\dot \Lambda_0}+\abs{\uplambda[h]}\\
			&\lesssim \mu^{l_0}\norm{g}_{\sharp,l_0,\delta_1,\frac{1}{2}+\ve}+\mu^{l_1}\norm{h}_{{\sharp,l_1,\delta_1,\frac{1}{2}+\ve}}.
		\end{align*}
		Given $h_j\in X_{\sharp,1+l_1,\frac{1}{2}+\ve}$ we consider the solution to the ODE
		\begin{align}\label{defXih}
			\dot \xi_{1,j}=\mathfrak{c}_j\mu_0(t)^{1+l_1}+h_j(t),
		\end{align} 
		given explicitly by
		\begin{align*}
			\xi_{1,j}[h](t)=\mathfrak{c}_j\int_t^\infty \mu_0(s)^{1+l_1} \dds+\int_t^\infty h(s)\dds
			{\coloneqq}\Upsilon_{j}+ \int_t^\infty h(s)\dds.
		\end{align*}
		In particular, we have
		\begin{align*}
			&\abs{\xi_{1,j}(t)}\lesssim \mu_0(t)^{1+l_1}+\mu_0(t)^{1+l_1} \norm{h}_{1+l_1,\infty},
			&\abs*{\dot \xi_{1,j}(t)}\lesssim \mu_0(t)^{1+l_1} \norm{h}_{1+l_1,\infty}
		\end{align*}
		We define the vector
		\begin{align*}
			\Xi(t){\coloneqq}\dot \xi -\dot{\Upsilon}=h(t),
		\end{align*}
		where $h=(h_1,h_2,h_3)$ satisfies $\norm{h_i}_{\sharp,1+l_1,\frac{1}{2}+\ve}<\infty$ for $i=1,2,3$. We define
		$$
		\norm{h}_{\sharp,1+l_1,\frac{1}{2}+\ve}{\coloneqq}\max_{i=1,2,3} \norm{h_i}_{\sharp,1+l_1,\frac{1}{2}+\ve}.
		$$ 
		Let $\LLL_2$ the linear operator defined as $\LLL_2[h]=\Xi$ by relation (\ref{defXih}) for $i=1,2,3$. We observe that $(\dot \Lambda,\dot \xi)$ is a solution to (\ref{SystemEquivalent}) if $( \uplambda,\Xi)$ satisfies
		\begin{align*}
			(\uplambda,\Xi)=\AAA(\uplambda,\Xi),
		\end{align*}
		where $\AAA$ is the operator
		\begin{align*}
			\AAA(\uplambda,\Xi){\coloneqq}\qty(\AAA_1[\uplambda,\Xi],\AAA_2(\uplambda,\Xi)){\coloneqq}\qty(\hat T_1[\hat G_1[\uplambda,\Xi,\phi]],\LLL_2[\hat \Theta[\uplambda,\Xi,\phi]]),
		\end{align*}
		and
		\begin{align*}
			&\hat G_1(\uplambda,\Xi,\phi){\coloneqq}G_1\qty[	\Lambda_0(t)+\int_{t}^{\infty}\uplambda(s)\dds ,\dot \Lambda_0-\uplambda,\Upsilon+\int_{t}^{\infty}\Xi(s)\dds,\dot \Upsilon-\Xi,\phi	],\\
			&\hat \Theta (\uplambda,\Xi,\phi){\coloneqq} \Theta\qty[	\Lambda_0(t)+\int_{t}^{\infty}\uplambda(s)\dds,\dot \Lambda_0-\uplambda ,\Upsilon(t)+\int_{t}^{\infty}\Xi(s)\dds,\dot \Upsilon-\Xi,\phi	],
		\end{align*}
		with $G_1$ and $\Theta$ defined in (\ref{DefG1}) and (\ref{Eqi=123Char}). 
		We show that there exists a unique fixed point $\qty(\uplambda,\Xi)=\qty(\uplambda[\phi],\Xi[\phi])$ in
		\begin{align*}
			\BBB=\{	(\uplambda,\Xi) \in (L^\infty(t_0,\infty))^4: \norm{\uplambda}_{\sharp,l_1,\delta_1,\ve}+\norm{\Xi}_{\sharp,1+l_1,\frac{1}{2}+\ve}\leq e^{-\kappa t_0}L	\}
		\end{align*}
		for some $L$ fixed large. Indeed, estimates (\ref{LinEstdotLambda<h}) and (\ref{CharPropG1}) give
		\begin{align*}
			\norm{\AAA_1[\uplambda,\Xi]}_{\sharp,l_1,\delta_1,\ve}
			&\leq C_2 \norm*{\hat G_1[\uplambda,\Xi,\phi]}_{\sharp,l_1,\delta_1,\frac{1}{2}+\ve}\\
			&\leq C_2e^{-\kappa t_0}\qty{\norm{\uplambda}_{\sharp,l_1,\delta_1,\ve}+\norm{\Xi}_{\sharp,1+l_1,\frac{1}{2}+\ve}+\norm{\phi}_*}.
		\end{align*}
		Also, from (\ref{smallnessTheta})
		\begin{align*}
			\norm{\AAA_2[\uplambda,\Xi]}_{\sharp,1+l_1,\frac{1}{2}+\ve}
			&\leq \norm{\Theta[\uplambda,\Xi,\phi]}_{\sharp,1+l_1,\frac{1}{2}+\ve}\\
			&\leq C  e^{-\kappa t_0}\qty{\norm{\uplambda}_{\sharp,l_1,\delta_1,\ve}+\norm{\Xi}_{\sharp,1+l_1,\frac{1}{2}+\ve}+\norm{\phi}_*}.
		\end{align*}
		We have to verify that $\AAA$ is a contraction. For instance, we have
		\begin{align*}
			\norm{\AAA_1[\uplambda_1,\Xi]-\AAA_1[\uplambda_2,\Xi]}_{\sharp,l_1,\delta_1,\ve}&= \norm{\hat T_1[\hat G_1[\uplambda_1,\Theta,\phi]-\hat G_1[\uplambda_2,\Theta,\phi]]}_{\sharp,l_1,\delta_1,\frac{1}{2}+\ve}\\
			&\leq C_2\norm{\hat G_1[\uplambda_1,\Theta,\phi]-\hat G_1[\uplambda_2,\Theta,\phi]}_{\sharp,l_0,\delta_1,\frac{1}{2}+\ve}\\
			&\leq C_2 \bold c \norm{\uplambda_1-\uplambda_2}_{\sharp,l_1,\delta_1,\ve},
		\end{align*}
		where $C_2,\bold c$ is the constant appearing in (\ref{LinEstdotLambda<h}) and (\ref{CharPropG2}) respectively. Since $\bold c$ can be as small as required provided that $t_0$ is fixed sufficiently large, we obtain that $\AAA_1$ is a contraction map. By means of the Lipschitz property of $\hat \Theta$ in (\ref{CharPropT2}) we can estimate $\AAA_2[\uplambda_1,\Xi_1]-\AAA_2[\uplambda_2,\Xi]$ similarly. Finally, using the estimates on $\hat G,\hat \Theta$ with respect to $\Xi$, we get
		\begin{align*}
			\norm{\AAA_1(\uplambda_1,\Xi_1)-\AAA_1(\uplambda_1,\Xi_1)}_{\sharp, l_1,\delta_1,\ve}\leq \bold c \qty[\norm{\uplambda_1-\uplambda_2}_{\sharp,l_1,\delta_1,\ve}+\norm{\Xi_1-\Xi_2}_{\sharp,l_1,\delta_1,\frac{1}{2}+\ve}].
		\end{align*}
		As a consequence of the Banach fixed point theorem, provided that $L$ and $t_0$ are fixed large, the map $\AAA$ has a unique fixed point $(\uplambda,\Xi)$ in the space $\BBB$. Observe that 
		$$
		\Lambda[h](t)=-\int_{t}^\infty \hat T_1[h](s)\dds=T_1[h],
		$$ 
		where $T_1$ is defined in (\ref{T1InverseJLambda}), satisfies (\ref{NormsLambda,dotLambda}) thanks to (\ref{estinvLambda}). Also, the components of vector $\xi_1=\int_{t}^\infty \Xi(s)\dds$ satisfy (\ref{NormsXi,dotXi}). This proves the existence of a solution $(\Lambda,\xi)$ to the system (\ref{SystemEquivalent}) satisfying (\ref{NormsLambda,dotLambda})-(\ref{NormsXi,dotXi}). With similar estimates on $\uplambda[\phi_1]-\uplambda[\phi_2]$ and $\Xi[\phi_1]-\Xi[\phi_2]$, using (\ref{CharPropG2}) and (\ref{CharPropT2}), relations (\ref{LambdaofPhiLip}) follow.
	\end{proof}
	We observe from the proof that $\hat T_1$, like an half-fractional derivative, loses $1/2$-H\"older exponent but we regain it through $g,G_1$ as a consequence of estimates on $\psi$ from Proposition \ref{PropOuterNonlinear}. This is the main reason to put all the terms of $S[u_3]$ involving directly $\dot \mu$ in the outer error (\ref{outS}). Indeed, to get $\dot \Lambda\in C^{\ve}$ it is crucial to allow $H$ in (\ref{RHSInnPr}) (and hence $\Sinn$ in (\ref{innS})) to depend on $\dot \Lambda$ only indirectly through $\psi[\dot \Lambda]$ or $J_1[\dot \Lambda]$. 
	
	\begin{remark}\label{rem:continuitylaxi}
		By remark \ref{rem:continuityPsi} the outer solution $\psi=\Psi[\psi_0]$ is smooth as a function of the initial datum $\psi_0$, provided that $\norm{\psi_0}_\infty+\norm{\nabla \psi_0}_\infty$ is sufficiently small. Thus, also the parameters $\Lambda[\psi_0],\xi[\psi_0]$ found in Proposition \ref{prop:FindLambdaXi} depend smoothly on $\psi_0$, and from the proof we also obtain
		\begin{align*}
			&\norm*{\Lambda[\psi_0^1]-\Lambda[\psi_0^2]}_{\infty}\lesssim \norm*{\psi_0^1-\psi_0^2}_\infty,\\
			&\norm*{\xi_1[\psi_0^1]-\xi_1[\psi_0^2]}_{\infty}\lesssim \norm*{\psi_0^1-\psi_0^2}_\infty.
		\end{align*}
	\end{remark}

	\section{Final argument: solving the inner problem}\label{sec:finalargument}
	This section provides the final step in the proof of Theorem \ref{mainteo}. At this point, given $\phi$ satisfying (\ref{Normphi}), we have a solution $\psi=\Psi[\Lambda[\phi],\xi[\phi],\phi]$ to the outer problem (\ref{OutPr}) and parameters $\Lambda[\phi],\xi[\phi]$ such that the orthogonality conditions (\ref{OrthCondFin}) are satisfied. Thus, to get a solution
	\begin{align*}
		u=u_3+\tilde \phi=u_3+\mu_0^{\frac{1}{2}}\psi+\eta_R \mu^{-\frac{1}{2}} \phi,
	\end{align*} 
	we need to prove the existence of $\phi$ such that $\norm{\phi}_*<\infty$. 
	\begin{proof}[Proof of Theorem \ref{mainteo}]
		We make a fixed point argument using the linear estimate (\ref{estinnSimplified}). Proposition \ref{propInnerlineartheory} defines a linear operator $\TTT:h \mapsto (\phi[h],e[h])$ which is continuous between the $L^\infty$-weighted space described in (\ref{estinnSimplified}). Thus, the solution $\phi$ to the nonlinear inner problem satisfies
		\begin{align}\label{FixedPointphi}
			\phi=\AAA_{\iinn}(\phi),\quad \text{where}\quad \AAA_{\iinn}(\phi){\coloneqq}\TTT(H[\phi]).
		\end{align}		
		We claim that $\AAA_\iinn$ has a unique fixed point in the space 
		\begin{align*}
			\BBB=\{\phi \in L^\infty(B_{2R}): \norm{\phi}_*< \bold b\},
		\end{align*}
		for some fixed constant $\bold b$ large. Firstly, we prove
		\begin{align*}
			\abs*{H[\Lambda,\xi,\dot \Lambda,\dot \xi](y,t)}\lesssim\frac{\mu^{1+l_1}}{1+\abs{y}^4}.
		\end{align*}
		We recall that 
		\begin{align*}
			H[\phi,\psi,\mu,\dot \mu,\xi,\dot \xi](y,\tau){\coloneqq}
			&5\Uy^4 \mu \qty(\frac{\mu_0}{\mu})^{1/2} \psi(\mu y +\xi,t(\tau))\\\nonumber
			&+B_0\qty[\phi+\mu\psi](\mu y+\xi,t(\tau))+\mu^{5/2}S_{\text{in}}(\mu y+\xi,t(\tau))\\\nonumber
			&+\mathcal{N}(\mu^{1/2}u_3,\mu^{1/2}\tilde \phi)(\mu y+\xi,t(\tau)).
		\end{align*}
		Using the estimate on $\psi$ given in Proposition \ref{PropOuterNonlinear}, we have
		\begin{align*}
			\abs{5\Uy^4 \mu \qty(\frac{\mu_0}{\mu})^{1/2} \psi(\mu y +\xi,t(\tau))}\lesssim e^{-\kappa t_0} \frac{\mu^{1+l_1+\delta-\sigma}}{1+\abs{y}^{4+\alpha}}
		\end{align*}
		and from (\ref{DefB0}) we get
		\begin{align*}
			\abs{B_0\qty[\phi+\mu\psi](\mu y+\xi,t(\tau))}
			&\lesssim \frac{1}{1+\abs{y}^3}\abs{\mu H_\g +\mu J + \mu^{-1/2}\phi_3 \eta_3}(\phi+\mu\psi)\\
			&\lesssim \frac{\mu}{1+\abs{y}^3} \qty(\bold b e^{-\kappa t_0}\frac{\mu^{1+l_1}R^3}{1+\abs{y}^4}+\mu \norm{\psi}_{**}\frac{\mu^{l_1+\delta-\sigma}}{1+\abs{y}^{\alpha}})\\
			&\lesssim e^{-\kappa t_0}\frac{\mu^{1+l_1}}{1+\abs{y}^4}.
		\end{align*}
		Recalling the estimates on $\phi$ at $y\sim 0$ and $y\sim R$ given by the norm (\ref{NormphiDef}), using that $R=\mu^{- \delta}$ with $\delta$ satisfying (\ref{condcontractionNLandHolder}) we deduce
		\begin{align*}
			\abs{\mathcal{N}(\mu^{1/2}u_3,\mu^{1/2}\tilde \phi)(\mu y+\xi,t(\tau))}&\lesssim \frac{1}{1+\abs{y}^3}\qty(\phi+\mu \psi)^2\\
			&\lesssim \frac{1}{1+\abs{y}^3}\qty(\bold b \frac{\mu^{2(1+l_1)}R^6}{1+\abs{y}^8}+\mu^2 \frac{\mu^{2(l_1+\delta-\sigma)}}{1+\abs{y}^{2\alpha}}	)\\
			&\lesssim e^{-\kappa t_0}\frac{\mu^{1+l_1}}{1+\abs{y}^4}.
		\end{align*}
		By Lemma \ref{Lemma:FinalAnsatzError} we have the main error
		\begin{align*}
			\abs{\mu^{5/2}S_{\text{in}}(\mu y+\xi,t(\tau))}\lesssim \frac{\mu^{1+l_1}}{1+\abs{y}^4}.
		\end{align*}
		Thanks to the previous section, $H$ satisfies the orthogonality conditions required by Proposition \ref{propInnerlineartheory}. Thus, provided that $t_0$ is large enough, we have 
		$$
		\norm{\TTT[H]}_*\leq C \norm{H}_{\nu,4}< \bold b,
		$$ 
		for $\bold b$ chosen large, where $C$ is the constant in (\ref{UsedLinearEst}). This proves $\AAA_{\iinn}(\phi)\in \BBB$.
		Now, we need to prove that for $\phi^{(1)},\phi^{(2)} \in \BBB$ we have
		\begin{align*}
			\abs*{H[\phi^{(1)}]-H[\phi^{(2)}]}\lesssim \bold c \norm*{\phi^{(1)}-\phi^{(2)}}_{*} \frac{\mu^{1+l_1}}{1+\abs{y}^4},
		\end{align*}
		for some $\bold c<1$. This is a consequence of Proposition \ref{PropertiesPsi} and Proposition \ref{prop:FindLambdaXi}. Indeed, for instance we get
		\begin{align*}
			5\Uy^4 &\mu_0\abs*{e^{\Lambda[\phi^{(1)}]} \psi[\phi^{(1)}]-
				e^{\Lambda[\phi^{(2)}]} \psi[\phi^{(2)}]}\\
			&=5\Uy^4  \mu_0 \abs*{\qty[e^{\Lambda[\phi^{(1)}]}-e^{\Lambda[\phi^{(2)}]}]\psi[\phi^{(1)}]+e^{\Lambda[\phi^{(2)}]}\qty[\psi[\phi^{(1)}]-\psi[\phi^{(2)}]]}\\
			&\lesssim \bold c \norm*{\phi^{(1)}-\phi^{(2)}}_* \frac{\mu^{1+l_1}}{1+\abs{y}^4},
		\end{align*}
		and similarly we get the same control on the other terms of $H[\phi^{(1)}]-H[\phi^{(2)}]$.
		Finally, since the operator $\TTT:X_{\nu,4}\to X_{*}$ is continuous, where $X_{\nu,4}$ is defined in (\ref{DefXnu4space}) for $a=2$, by composition with $H:X_* \to X_{\nu,4}$ we obtain
		\begin{align*}
			\norm*{\AAA_\iinn[\phi^{(1)}]-\AAA_{\iinn}[\phi^{(2)}]}_*\leq  \bold c \norm*{\phi^{(1)}-\phi^{(2)}}_*,
		\end{align*}
		provided that $t_0$ is fixed sufficiently large. Thus, $\AAA_\iinn:\BBB\to \BBB$ is a contraction map and by Banach fixed point theorem we obtain the existence and uniqueness of $\phi \in X_*$ such that (\ref{FixedPointphi}) holds. Finally, we recall that the constant $e_0=e_0[H]$ in the initial condition $\phi(y,t_0)=e_0Z_0(y)$ is a linear operator of $H$. The existence of $\phi$ immediately defines $e_0$. This completes the proof of the existence of $u=u_3 + \tilde \phi$ in Theorem \ref{mainteo}, with the bubbling profile centered in $x=0\in \Om$ and parameters satisfying (\ref{behmuxi}).
	\end{proof}

	\begin{remark}[continuity of $(\phi,e_0)$ with respect to $\psi_0$]\label{remark:1cod}
		We found the inner perturbation $\phi$ and its initial datum $\phi(y,t_0) = e_0Z_0(y)$ based on the existence of the outer solution $\Psi[\phi]$ given by Proposition \ref{PropOuterNonlinear}, which in fact can be found for any initial condition $\psi_0\in C^1(\bar \Om)$. Furthermore, as a consequence of the continuity of $\Psi[\psi_0]$ and $\Lambda[\psi_0],\xi[\psi_0]$ found in remarks \ref{rem:continuityPsi} and \ref{rem:continuitylaxi} we obtain
		\begin{align*}
			\abs*{e_0[\psi_0^1]-e_0[\psi_0^2]}\lesssim \qty[\norm*{\psi_0^1-\psi_0^2}_{L^\infty(\Om)}-\norm*{\nabla \psi_0^1 -\nabla \psi_0^2}_{L^\infty(\Om)}].
		\end{align*}
		Since we know that $\Lambda,\dot \Lambda,\xi,\dot \xi,\psi$ depends smoothly on $\psi_0$, by the implicit function theorem, we deduced that map $\psi_0\mapsto(\phi[\psi_0],e_0[\psi_0])$ is $C^1$ with respect to $\psi_0\in C^1(\bar \Om)$. This allows to prove the 1-codimensional stability of Corollary \ref{CorStability}, under small perturbation. With these ingredients, we can proceed as in \cite[Proof of Corollary 1.1]{cdm}.
	\end{remark}

	\section{Invertibility theory for the nonlocal linear problem}\label{sec:invj}
	In this section we prove Proposition \ref{Proposition:invj}. We deduce the result by Laplace transform method combined with asymptotic estimates of the heat kernel $p_t^{\Om}$ associated to $\Om$. It turns out that the operator $\JJJ[\dot \Lambda]$ is similar to a half-fractional integral of $\dot \Lambda$. Thus, roughly speaking, we expect the inverse operator to behave as a fractional derivative of order $1/2$. In fact, Proposition \ref{Proposition:invj} can be seen as a precise statement of this idea.

	For later purpose we recall some facts about the Dirichlet heat kernel. For the definition and properties we follow  \cite{dod,grigbook}.
	A function $p_t^{\Om}(x,y)$ continuous on $\bar \Om \times \bar \Om \times \RR^+$, $C^2$ in $x$ and $C^1$ in $t$ is called Dirichlet heat kernel for the problem
	\begin{align*}
		&\pp_t u(x,t)=\Delta u(x,t)\inn \Om \times \RR^+,\\
		&u(x,t)=0 \onn \pp \Om \times [0,\infty),\\
		&u(x,0)=u_0(x )\inn \Om,
	\end{align*}
	if, for any $y \in \Om$, satisfies
	\begin{align*}
		&\pp_t p_t^{\Om}(x,y) = \Delta_x p_t^{\Om}(x,y) \inn \Om  \times \RR^+,\\
		&p_t^{\Om}(x,y)=0 \inn \pp\Om,
	\end{align*}
	and
	\begin{align*}
		\lim\limits_{t\to 0^+}\int_\Om p_t^{\Om}(x,y)u_0(y)\ddy =u_0(x),
	\end{align*}
	uniformly for every function $u_0 \in C_0(\bar \Om)$.
	The existence of the Dirichlet heat kernel is a classical result by Levi \cite{Levi}. It has the following basic properties:
	\begin{itemize}
		\item $p_t^{\Om}(x,y)\geq0$, $p_t^{\Om}(x,y)=p_t^{\Om}(y,x)$ and $p_t^{\Om}(x,y)=0$ if $x\in \pp \Om$;
		\item for any $y\in \Om$ the function $p_t^{\Om}(x,y)\in C^{\infty}(\RR^+\times \Om)$;
		\item it satisfies $\pp_t p_t^{\Om}(x,y)=\Delta_x p_t^{\Om}(x,y)$ for $(x,y,t)\in \Om \times \Om \times \RR^+$.
		\end{itemize}
	Also, from \cite[Theorem 10.13]{grigbook} and its proof, the heat kernel $p_t^{\Om}(x,y)$ admits the expansion
	\begin{align}\label{ptseries}
		p_t^\Om(x,y)=\sum_{k\geq 1} e^{-\lambda_k t}\phi_k(x)\phi_k(y),
	\end{align}
	where $\lambda_k$ is the $k$-th Dirichlet eigenvalue of $-\Delta$ on $\Om$ and $\phi_k$ the corresponding eigenfunction and also for $n\geq 1$ (see \cite[Remark 10.15]{grigbook})
	\begin{align}\label{Grigeq}
		\sum_{k=n}^{\infty}\sup_{x,y \in \Om}\abs{\phi_k(x)\phi_k(y)}<\infty
	\end{align}
	The series (\ref{ptseries}) converges absolutely and uniformly in $[\ve,\infty]\times \Om\times \Om$ for any $\ve>0$, as well as in the topology of $C^\infty(\RR^+\times \Om \times \Om)$. 
	
	Before starting the proof of Proposition \ref{Proposition:invj}, we recall an estimate on the short time behavior of the heat kernel $p_\tau^\Om(x,y)$ due to Varadhan \cite[Theorem 4.9]{VaradhDiff}. We will use it in the following form as in Hsu \cite[Corollary 1.6]{Hsu}.
	\begin{lemma}[short time estimate of $p_\tau^{\Om}$]
		Let $\ve>0$ fixed such that $B_\ve(0)\subset \Om$. Then, there exists $\tau_0>0$ such that, for $y\in B_{\ve}(0)$ and $\tau \in (0,\tau_0)$ we have
		\begin{align}\label{boundVaradhan}
			p^{\RR^3}_\tau(0,y)(1-e^{-\frac{\delta^2}{4\tau}})	\leq p^\Omega_\tau(0,y),
		\end{align}
		where $\delta<\delta_0$ is independent of $y$ and
		$$
		\delta_0{\coloneqq}d(\pp \Om, \pp B_{\ve})=\min_{a\in \pp \Om, b \in \pp B_\ve(0)} \abs{a-b}>0.
		$$
	\end{lemma}
	\begin{proof}
		Recall the identities in \cite[p. 675]{VaradhDiff}
		\begin{align}\label{VarId1}
			&\limsup_{\tau\to 0} 4\tau \log(p^{\RR^3}_{\tau}(x,y)-p^{\Om}_{\tau}(x,y))\leq -d_{\pp \Omega}(x,y)^2,\\\label{VarId2}
			&\lim_{\tau\to 0}4 \tau \log(p^{\RR^3}_\tau(x,y))=- d(x,y)^2,
		\end{align}
		where
		\begin{align*}
			d_{\pp \Omega}(x,y){\coloneqq}\inf_{z \in \pp \Omega}\{ d(x,z)+d(z,y) \}.
		\end{align*}
		From (\ref{VarId1}) for $\tau \in (0,\tau_0)$ we have
		\begin{align*}
			p_{\tau}^{\RR^3}(x,y)- e^{-\frac{	d_{\pp \Omega}^2(x,y)-c(\tau_0)}{4\tau}}\leq p_{\tau}^{\Om}(x,y),
		\end{align*}
		for all $x,y \in \Om$, where $0\leq c(\tau_0)=o(1)$ as $\tau_0\to 0$. 
		In particular, fix $x=0$ and consider $y \in B_{\ve}(0)\subset \Om$ for a small $\ve>0$. Then, choosing $\tau_0$ smaller if needed, we have
		\begin{align*}
			d_{\pp \Omega}^2(0,y)-c(\tau_0)\geq \ve^2 +\delta_0^2.
		\end{align*}
		Thus for $y \in B_\ve (0)$ 
		\begin{align*}
			e^{-\frac{	d_{\pp \Omega}^2(0,y)-c(\tau_0)}{4\tau}}\leq e^{-\frac{\ve^2+\delta_0^2}{4 \tau}}\leq e^{-\frac{d(0,y)^2}{4\tau}}e^{-\frac{\delta_0^2}{4\tau}},
		\end{align*}
		and (\ref{VarId2}) says
		\begin{align*}
			p_{\tau}^{\RR^3}(0,y)=e^{-\frac{d^2(0,y)}{4\tau}(1+o(1))} \ass \tau\to 0^+.
		\end{align*}
		Thus, we have for $\tau<\tau_0$ small and $y \in B_\ve(0)$
		\begin{align}\label{VarIdentity}
			p^{\RR^3}_\tau(0,y)(1-e^{-\frac{\delta^2}{4\tau}})	\leq p^\Omega_\tau(0,y),
		\end{align}
		for any $\delta<\delta_0$ independent of $y$. 
	\end{proof}
	We mention that the uniform bound (\ref{VarIdentity}) holds for $y$ ranging in any convex subset of the domain, see \cite[p.374-375]{Hsu}.
	Also, for any $\tau>0$ and $x,y \in \Om$ we have the upper bound
	\begin{align}\label{bound:MPconseq}
		p^\Omega_\tau(x,y)\leq p^{\RR^3}_\tau(x,y),
	\end{align}
	as a consequence of the maximum principle. Thus, Varadhan's estimate (\ref{boundVaradhan}) is a precise statement about the idea that for small times the heat kernel "does not feel the boundary". We refer to Kac \cite{Kac} and Dodziuk \cite{dod} for statements with the same flavor. In the proof of Proposition \ref{Proposition:invj} we need the following lemma.
	\begin{lemma}\label{Lemma:ItauAsymptotic}
		Define the function
		\begin{align*}
			I(\tau){\coloneqq}\int_{\Om} p_{\tau}^{\Om}(0,y)G_\g(y,0)\ddy,
		\end{align*}
		where $p_{\tau}^{\Om}(x,y)$ denotes the Dirichlet heat kernel associated to $\Om$ and $G_\g(x,y)$ the Green function of the operator $-\Delta-\g$ on $\Om$. Then $I(\tau)$ has the following asymptotic behavior:
		\begin{align}\label{asyI(t)}
			I(\tau)=
			\begin{cases}
				O\qty(e^{-\lambda_1 \tau})&\text{for} \quad \tau \to \infty,\\
				c_{1,*} \frac{1}{\sqrt{\tau}}+c_{2,*} \sqrt{\tau}+c_{3,*}\tau +O\qty(\tau^{3/2})  
				&\text{for} \quad  \tau\to 0^+,
			\end{cases}
		\end{align}
		for some constant $c_{i,*}$ and $i=1,2,3$.
	\end{lemma}
	
	\begin{proof}
		\textbf{Step 1} (Asymptotic for $\tau\to \infty$). We recall that the heat kernel $p_\tau^{\Omega}(x,y)$ admits the series expansion (\ref{ptseries}) which converges absolutely and uniformly in the domain $[\ve,\infty)\times \Omega \times \Omega$ for any $\ve>0$, as well as in the topology $C^\infty(\RR^+\times \Omega\times \Omega)$. By the uniform convergence with respect to $y\in \Om$ we obtain for $\tau>0$
		\begin{align}\label{expItau}
			I(\tau)&=\int_\Omega \sum_{k=1}^\infty e^{-\lambda_k \tau}\phi_k(0)\phi_k(y)G_\g(y,0)\ddy \\\nonumber
			&= \sum_{k=1}^\infty e^{-\lambda_k \tau}\phi_k(0)\int_\Omega\phi_k(y)G_\g(y,0)\ddy.
		\end{align}
		Multiplying equation \reff{Ggammaequation} by $\phi_k$ and integrating by parts we get
		\begin{align*}
			-\lambda_k \int_\Om G_\g(x,0)\phi_k(x)\ddx
			&=\int_\Om G_\g(x,0)\Delta \phi_k(x)\\
			&=\int_\Om \phi_k(x)\Delta G_\g(x,0)\ddx  \\
			&= -\g \int_\Om G_\g(x,0)\phi_k(x)\ddx-c_3\int_\Om \phi_k(x) \delta_0(x)\ddx\\
			&=  -\g \int_\Om G_\g(x,0)\phi_k(x)\ddx-c_3 \phi_k(0),
		\end{align*}
		that gives
		\begin{align}\label{coeffGgamma}
			\int_\Om G_\g(x,0)\phi_k(x)\ddx =c_3 \frac{\phi_k(0)^2}{\lambda_k-\g}.
		\end{align}
		We plug (\ref{coeffGgamma}) into (\ref{expItau}). Finally, from (\ref{Grigeq}) we obtain the asymptotic behavior (\ref{asyI(t)}) for $\tau\to \infty$. \\
		\textbf{Step 2} (Asymptotic for $\tau\to 0^+$).
		Firstly, we split 
		\begin{align*}
			I(\tau)=&\int_\Om p_\tau^{\Omega}(0,y) \frac{\alpha_3}{\abs{y}}\ddy+\int_\Om p_\tau^{\Omega}(0,y) H_\g(y,0)\ddy\\
			=&:I_1(\tau)+I_2(\tau).
		\end{align*}
		We analyze $I_1(\tau)$. For the region $B_{\ve}(0)$ we invoke Varadhan's estimate (\ref{VarIdentity}) and we obtain
		\begin{align*}
			\int_{B_\ve(0)} \frac{p_\tau^\Omega(0,y)}{\abs{y}}\ddy 
			&\geq \int_{B_\ve} \frac{e^{-\frac{\abs{y}^2}{4\tau}}}{[4\pi\tau]^{3/2}}\frac{1}{\abs{y}}\ddy (1-e^{-\frac{\ve^2}{\tau}})\\
			&=4\pi \int_0^\ve \frac{e^{-\frac{\rho^2}{4\tau}}}{[4\pi\tau]^{3/2}}\rho \,d\rho (1-e^{-\frac{\ve^2}{\tau}})\\
			&=\frac{1}{\sqrt{4\pi\tau}}\int_{0}^{\frac{\ve}{2\sqrt{\tau}}}e^{-r^2}r \ddr (1-e^{-\frac{\ve^2}{\tau}}) \\
			&=\frac{1}{\sqrt{4\pi \tau}}\qty(\frac{1-e^{-\frac{\ve^2}{4\tau}}}{2})\qty(1-e^{-\frac{\ve^2}{\tau}})\\
			&=\frac{1}{4\sqrt{\pi \tau}}+O\qty(\frac{e^{-\frac{c}{\tau}}}{\sqrt{\tau}})
		\end{align*}
		for some $c>0$, and by (\ref{bound:MPconseq}) we have
		\begin{align*}
			\int_{B_\ve(0)} \frac{p_\tau^\Omega(0,y)}{\abs{y}}\ddy
			&\leq \int_{B_\ve(0)} \frac{p_\tau^{\RR^3}(0,y)}{\abs{y}}\ddy\\
			&\leq \frac{1}{4 \sqrt{\pi\tau}}+O\qty(\frac{e^{-\frac{c}{\tau}}}{\sqrt{\tau}}).
		\end{align*}
		From these bounds we conclude
		\begin{align*}
			\int_{B_\ve(0)} \frac{p_\tau^\Omega(0,y)}{\abs{y}}\ddy=\frac{1}{4 \sqrt{\pi\tau}}+O\qty(\frac{e^{-\frac{c}{\tau}}}{\sqrt{\tau}}).
		\end{align*}
		In the region $\Omega \setminus B_\ve(0)$ by (\ref{bound:MPconseq}) we get
		\begin{align*}
			\int_{\Omega\setminus B_\ve (0)} \frac{p_\tau^\Omega(0,y)}{\abs{y}}\ddy &\leq  \tau^{-3/2}\int_{\ve}^{1} e^{-\frac{\rho^2}{cs}}\rho \, d\rho \\
			&=\tau^{-1/2} \int_{\frac{\ve}{\sqrt{s}}}^{\frac{1}{\sqrt{s}}} e^{-r^2}r\ddr =O\qty(\frac{e^{-\frac{c}{\tau}}}{\sqrt{\tau}}).
		\end{align*}
		We conclude that
		\begin{align*}
			I_1(\tau)&=\alpha_3\int_{\Omega\setminus B_\ve (0)} \frac{p_\tau^\Omega(0,y)}{\abs{y}}\ddy+	\alpha_3\int_{B_\ve(0)} \frac{p_\tau^\Omega(0,y)}{\abs{y}}\ddy\\
			&= \frac{c_{1,*}}{\sqrt{\tau}}+O\qty(\frac{e^{-\frac{c}{\tau}}}{\sqrt{\tau}}) \ass \tau \to 0^+,\quad \text{with}\quad c_{1,*}=\frac{\alpha_3}{4\sqrt{\pi \tau}}.
		\end{align*}
		Now, we estimate the term $I_2(\tau)$. We treat it similarly to $I_1(\tau)$ but we get a lower order term in the expansion since $H_\g(y,0)$ is not singular. We use decomposition (\ref{decHg}) for $H_\g(y,0)$ and we consider the integral over $B_{\ve}(0)$. Using the cosine expansion we get
		\begin{align*}
			\theta_\g(y,0)=\alpha_3\frac{\g}{2}\abs{y}+O(\abs{y}^3).
		\end{align*}
		Thus, we compute the integral associated to the first term with Varadhan's estimate (\ref{boundVaradhan}) and the upper bound (\ref{bound:MPconseq}):
		\begin{align}\label{1-cosInt}
			\int_{B_\ve(0)} p_{\tau}^{\Om}(y,0)\frac{1-\cos(\sqrt{\g}\abs{y})}{\abs{y}} \ddy 
			&= \alpha_3\frac{\g}{2}\int_{B_\ve(0)}  \frac{e^{-\frac{\abs{y}^2}{4\tau}}}{[4\pi\tau]^{3/2}}\abs{y} \ddy \qty(1+o\qty(e^{-\frac{c}{\tau}}))\\\nonumber
			&=4 \pi \alpha_3\frac{\g}{2} \int_{0}^{\ve} \frac{e^{-\frac{\rho^2}{4\tau}}}{[4\pi \tau]^{3/2}} \rho^3 \,d\rho \qty(1+o\qty(e^{-\frac{c}{\tau}}))\\\nonumber
			&=4 \pi \alpha_3 \sqrt{\tau}\frac{\g}{2}\int_0^{\frac{\ve}{2\sqrt{\tau}}} e^{-r^2}r^{3}\ddr \qty(1+o\qty(e^{-\frac{c}{\tau}}))\\\nonumber
			&=c_{2,*} \sqrt{\tau}\qty(1+o\qty(e^{-\frac{c}{\tau}})),
		\end{align}
		for an explicit constant $c_{2,*}$. The same computation on the remainder $O\qty(\abs{y}^3)$ gives a term of order $O\qty(\tau^{3/2})$. Another Taylor expansion at $y=0$ gives
		\begin{align*}\label{exphgy0}
			h_\g(y,0)=\nabla_y h_\g(0,0)\cdot y +\frac{1}{2}y\cdot D_{yy}h_\g(0,0)\cdot y+O(\abs{y}^3),	
		\end{align*}
		where $D_{yy}h_\g(0,0)$ denotes the Hessian of $h_\g(\cdot,0)$ evaluated in
		$y=0$. Integrating the first term on $B_\ve(0)$ against $p_t^\Om(0,y)$ and using (\ref{boundVaradhan})-(\ref{bound:MPconseq}) wee see by symmetry of the integrand $p_t^{\RR^3}(0,y)\nabla_y h_\g(0,0)\cdot y$ that the integral gives an exponentially decaying term. 
		The second term in (\ref{exphgy0}) can be treated similarly to (\ref{1-cosInt}) and gives a term of order $c_{3,*}\tau(1+o(1))$ for some explicit constant $c_{3,*}$. 
		The integral of $p_{\tau}^\Om(y,0) H_\g(y,0)$ on the complement can be treated as before and gives an exponentially decay term for $\tau\to 0$.   
		Thus, we obtain that
		\begin{align*}
			I_2(\tau)=c_{2,*} \sqrt{\tau}+c_{3,*}\tau +O\qty(\tau^{3/2}) \ass \tau \to 0^+.
		\end{align*}
		We conclude that $I(\tau)=I_1(\tau)+I_2(\tau)$ has the asymptotic (\ref{asyI(t)}) for $\tau\to 0^+$.
	\end{proof}
	We start here the main proof of Proposition \ref{Proposition:invj}.
	\subsection*{Proof of Proposition \ref{Proposition:invj}}
	Firstly, we observe that $J(0,t_0)=h(t_0)$ is in general not compatible with a null initial condition. For this reason it is natural to solve the problem for $\JJJ$ starting from $t=t_0-1$.
	We look for $\Lambda(t)$ for $t \in (t_0-1,\infty)$. The function $\JJJ$ is a solution to the problem
	\begin{align*}
		&\pp_t \JJJ= \Delta_x \JJJ + \g \JJJ - \dot \Lambda(t) G_\gamma(x,0) \inn \Omega \times (t_0-1,\infty),\\
		&\JJJ(x,t)\equiv 0 \onn \pp \Omega \times (t_0-1,\infty),
	\end{align*}
	such that
	\begin{align*}
		\JJJ(0,t)=h^*(t) \inn (t_0,\infty),
	\end{align*}
	where
	\begin{align}
		h^*(t)=
		\begin{cases}
			h(t)&t\in [t_0,\infty),\\
			h_{\ext}(t)&t\in [t_0-1,t_0),
		\end{cases}
	\end{align}
	and 
	\begin{align*}
		h_{\ext}(t)=\eta(t) h(t_0),
	\end{align*}
	where $\eta$ is a smooth function such that $\eta(t_0-1)=0$, $\eta(t_0)=1$ and
	$$
	\abs{\eta(t_0-\nu)h(t_0)-h(t_0+\nu)}\leq  [h]_{\ve,[t_0,t_0+1]} \nu^\ve,
	$$
	for any $\nu\leq 1$. This choice gives an extension $h^*(t)\in C^{\ve}$ with
	\begin{align}\label{normh0*}
		\norm{h^*}_{\sharp,c_1,c_2,(t_0-1,\infty)}\lesssim \norm{h}_{\sharp,c_1,c_2,(t_0,\infty)}.
	\end{align}
	Let $s{\coloneqq}t-(t_0-1)$ and for $s\in (0,\infty)$ define
	\begin{align}\label{covbetaLambda}
		&\JJJ_0(x,s){\coloneqq}e^{-\g s} \JJJ(x,s+(t_0-1)), \\\nonumber
		&\beta(s){\coloneqq}- \Lambda(s+(t_0-1)),\\\nonumber
		&h_0^*(s){\coloneqq}h^*(s+(t_0-1)).
	\end{align}
	The function $\JJJ_0$ is a solution to
	\begin{align*}
		&\pp_s \JJJ_0(x,s)= \Delta_x \JJJ_0 + e^{-\g s}\dot \beta(s)G_\g(x,0) \inn \Omega\times (0,\infty)\\\nonumber
		&\JJJ_0(x,s)=0 \onn \pp \Omega\times (0,\infty),\\\nonumber
	\end{align*}
	such that
	\begin{align}\label{eqJ0}
		\JJJ_0[\dot \beta](0,s)=h_{0}^*(s)e^{-\g s} \inn (0,\infty).
	\end{align}
	Imposing the initial condition $\JJJ(x,t_0)\equiv 0$ in $\Om$, that is $\JJJ_0(x,0)\equiv 0$, by Duhamel's formula we have
	\begin{align}\label{Duah}
		\JJJ_0[\dot \beta](0,s)=\int_0^s e^{-\g (s-\tau)}\dot \beta(s-\tau) I(\tau) \, d\tau,
	\end{align}
	where
	\begin{align*}
		I(\tau){\coloneqq}\int_\Omega p_\tau^{\Omega}(0,y) G_\g(y,0)\ddy,
	\end{align*}
	and $p_\tau^{\Omega}(x,y)$ denotes the heat kernel associated to $\Omega$. 
	The asymptotic behavior of $I(\tau)$ is given by Lemma \ref{Lemma:ItauAsymptotic}.
	We denote the Laplace transform of a function $f$ as 
	\begin{align*}
		\tilde f(\xi){\coloneqq}\int_{0}^{\infty}e^{-\xi s}f(s)\dds. 
	\end{align*}
	We refer to the book \cite{Doetsch} by Doetsch for classic properties of the Laplace transform. Applying the Laplace transform to \reff{Duah}, using \reff{eqJ0} and the basic property
	\begin{align*}
		\tilde{\dot f}(\xi)=\xi\tilde f(\xi)-f(0),
	\end{align*} 
	we obtain
	\begin{align*}
		\tilde h_0^*(\xi+\g)&= \tilde{\dot \beta}(\xi+\gamma) \tilde{I}(\xi)\\
		&=\qty[(\xi+\gamma)\tilde \beta(\xi+\gamma)-\beta(0)] \tilde{I}(\xi),
	\end{align*}
	and hence
	\begin{align}\label{formulatildebeta}
		\tilde{\beta}(\xi+\gamma)=\frac{\beta(0)}{\xi+\gamma}+\tilde h_0^*(\xi+\g)\tilde{\sigma}(\xi),
	\end{align}
	where
	\begin{align*}
		\tilde \sigma(\xi){\coloneqq}\frac{1}{(\xi+\gamma)\tilde{I}(\xi)}.
	\end{align*}
	By definition we have
	\begin{align*}
		\tilde{I}(\xi)=\int_{0}^{\infty}e^{-\xi s}I(s)\dds,
	\end{align*}
	that is well defined and analytic in the right-half plane $\Re\xi>-\lambda_1$ thanks to Lemma \ref{Lemma:ItauAsymptotic}.
	By expansion (\ref{asyI(t)}) we have 
	\begin{align*}
		\abs*{e^{-\xi s}I(s)}\lesssim g(s),\quad	g(s)=
		\begin{cases}
			\frac{1}{\sqrt{s}} &\text{for} \quad s \to 0^+,\\
			e^{-(\lambda_1+\Re \xi) s} &\text{for} \quad s \to +\infty,
		\end{cases}
	\end{align*}
	and $g$ is integrable in $\RR^+$ if $\Re{\xi}>-\lambda_1$. Thus, using (\ref{expItau}), in any half plane $\Re \xi\geq c$ where $c>-\lambda_1$ the dominated convergence theorem applies to get
	\begin{align*}
		\tilde I(\xi)&= \int_{0}^{\infty} e^{-\xi s} I(s)\dds \\&= \int_{0}^{\infty} e^{-\xi s} \sum_{k=1}^{\infty} \frac{\phi_k(0)^2}{\lambda_k-\g}e^{-\lambda_k s} \dds\\
		& =  \sum_{k=1}^{\infty}\frac{\phi_k(0)^2}{\lambda_k-\g}\int_{0}^{\infty} e^{-\xi s} e^{-\lambda_k s} \dds\\
		& =  \sum_{k=1}^{\infty}\frac{\phi_k(0)^2}{\lambda_k-\g}\frac{1}{\lambda_k+\xi}
	\end{align*} 
	At this point we can extend $\tilde I(\xi)$ analytically from $\{\xi \in \mathbb{C}: \xi>-\lambda_1\}$ to $\mathbb{C}\setminus \{-\lambda_k\}_{k=1}^\infty$. Let $\xi=a+ib$ and rewrite the series as
	\begin{align*}
		\tilde{I}(\xi)=&\sum_{k=1}^\infty \frac{\phi_k(0)^2}{\lambda_k-\g} \frac{1}{\lambda_k + a+ib}\\
		=&\sum_{k=1}^\infty \frac{\phi_k(0)^2}{\lambda_k-\g} \frac{\lambda_k+a}{(\lambda_k+a)^2+b^2}-ib \sum_{k=1}^\infty  \frac{\phi_k(0)^2}{\lambda_k-\g}\frac{1}{(\lambda_k +a)^2+b^2}.
	\end{align*}
	Since the coefficients of the series are positive, $\tilde I (\xi)=0$ implies $b=0$. Plugging $b=0$ into the first series we obtain that a root $\xi=a$ of $\tilde I$ satisfies
	\begin{align*}
		\sum_{k=1}^\infty \frac{\phi_k(0)^2}{\lambda_k-\g} \frac{1}{\lambda_k+a}=0.
	\end{align*}
	Hence, we deduce that the set of zeros of $\tilde I$ is given by a sequence $\{-a_k\}_{k=1}^{\infty}$ where $a_k\in (\lambda_k,\lambda_{k+1})$. In particular, 
	\begin{align}\label{tildeInonzero}
		\tilde I(\xi)\neq 0 \quad \text{for}\quad \Re \xi>-\lambda_1.
	\end{align}
	By standard argument \cite[Theorem 33.7]{Doetsch} on the Laplace transform, using (\ref{asyI(t)}), we have
	\begin{align*}
		\tilde{I}(\xi)=  c_{1,*}\sqrt{\pi}\xi^{-1/2}+c_{2,*}\frac{\sqrt{\pi}}{2} \xi^{-3/2}+c_{3,*}\xi^{-2}+O(\xi^{-5/2})\ass \abs{\xi}\to \infty,
	\end{align*}
	{in the half-plane} $\Re \xi>-\lambda_1$. 
	Thus, in the same half-plane we have
	\begin{align}\label{expsigmatilde}
		\tilde{\sigma}(\xi)=&\frac{1}{(\xi+\g)\tilde I(\xi)}\\\nonumber
		=&d_{1,*}\xi^{-1/2}+d_{2,*}\xi^{-3/2}+d_{3,*}\xi^{-2}+O(\xi^{-5/2})
		\ass \abs{\xi}\to \infty.
	\end{align}
	As a consequence of (\ref{tildeInonzero}), $\tilde \sigma(\xi)$ has a unique singularity at $\xi=-\g$ in the half-plane of convergence.
	By \cite[Theorem 28.3]{Doetsch} the function $\tilde \sigma(\xi)$
	can be represented as a Laplace transform of a function.\footnote{We cannot have an estimate directly on $\dot \beta$ at this point. Indeed, $(\tilde{I}(\xi))^{-1}$ is not a Laplace transform of a function since diverges as $\abs{\xi}\to \infty$. However, it still can be represented as the Laplace transform of a distribution, see \cite[Theorem 29.3]{Doetsch}.} Finally, we compute the inverse Laplace transform by means of the Residue theorem defining the rectangular contour integral $\mathcal{C}_R$ as in \autoref{contourfig}, which is suggested by the proof of \cite[Theorem 35.1]{Doetsch}. 
	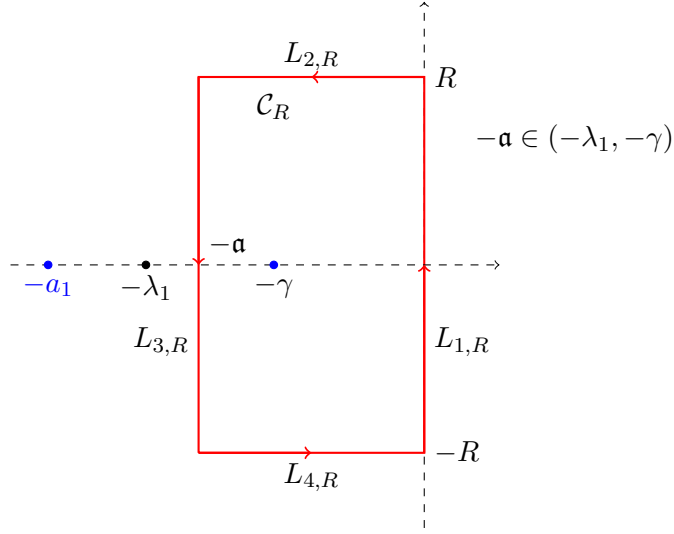
\begin{figure}
		\centering
		\begin{tikzpicture}

			\coordinate (O) at (2,2);
			
			\def\a{3}
			\def\b{5}
			\path (O) --+(0:\a) coordinate (A);
			\path (O) --+(90:\b) coordinate (C);
			\path (C) --+(0:\a) coordinate (B);
			\path (A) --+ (-90:1) coordinate (B1);
			\path (O) --+(-2.5,\b/2) coordinate (C1);
			\path (C1) --+(.5,0) coordinate (a1);
			\path (C1) --+(1.8,0) coordinate (lambda1);
			\path (C1) --+(0.5+\a,0) coordinate (gamma);
			\path (O) --+(0,\b/2) coordinate (beta);
			\path (B) --+(2,-.5) coordinate (anchor);
			\path (B) --+(-2,-0.1) coordinate (anchor2);
			\path (B) --+(0.5,-3.2) coordinate (L1);
			\path (B) --+(-1.5,0.6) coordinate (L2);
			\path (B) --+(-3.5,-3.2) coordinate (L3);
			\path (B) --+(-1.5,-5) coordinate (L4);
			
			\draw[black,->, dashed] (B1) --+ (90:\b+2);
			\draw[red, thick] (O) to (A) to (B) to (C)  to (O);
			\draw[black,->, dashed] (C1) --+ (0:\a+3.5);
			
			\filldraw[blue] (a1) circle (.05);
			\filldraw[black] (lambda1) circle (.05);
			\filldraw[blue] (gamma) circle (.05);

			\draw[red, thick, ->] (O) --+ (0:\a/2);
			\draw[red, thick, ->] (A) --+ (90:\b/2);
			\draw[red, thick, ->] (B) --+ (180:\a/2);
			\draw[red, thick, ->] (C) --+ (-90:\b/2);
			
			\node[below, blue] at (a1) {$-a_1$};
			\node[below] at (lambda1) {$-\lambda_1$};
			\node[below] at (gamma) {$\blue{-\gamma}$};
			\node[above right] at (beta) {$-\mathfrak{a}$};
			\node[right] at (A) {$-R$};
			\node[right] at (B) {$R$};
			\node[below] at (anchor) {$-\mathfrak{a}\in (-\lambda_1, -\gamma)$};
			\node[below] at (anchor2) {$\mathcal{C}_R$};
			\node[below] at (L1) {$L_{1,R}$};
			\node[below] at (L2) {$L_{2,R}$};
			\node[below] at (L3) {$L_{3,R}$};
			\node[below] at (L4) {$L_{4,R}$};
		\end{tikzpicture}
		\caption{Contour integral $\mathcal{C}_R$.}
		\label{contourfig}
	\end{figure}
	For later purpose we observe that, looking at the contour integral $\CCC_R$, the constant $\mathfrak{a}\in (\g,\lambda_1)$ can be taken arbitrarily close to $\lambda_1$. An application of the Riemann-Lebesgue Lemma (as in \cite[p.237]{Doetsch}) implies
	\begin{align*}
		&\lim_{R\to \infty}\int_{L_{2,R}}e^{\xi \tau}\tilde \sigma(\xi)\, d\xi=0,\\
		&\lim_{R\to \infty}\int_{L_{4,R}}e^{\xi \tau}\tilde \sigma(\xi)\, d\xi=0. 
	\end{align*}
	Since 
	\begin{align*}
		\sigma(\tau)=\lim_{R\to \infty}\frac{1}{2\pi i}\int_{L_{1,R}}e^{\xi \tau}\tilde \sigma(\xi)\, d\xi
	\end{align*}
	we obtain
	\begin{align}\label{formulasigma(t)}
		\sigma(t)=\text{Res}\qty(e^{\xi t}\tilde \sigma(\xi),-\gamma)e^{-\g t}+\lim\limits_{R\to \infty}\frac{1}{2\pi i} \int_{-\mathfrak{a}-i R}^{-\mathfrak{a}+iR} e^{\xi t}\tilde \sigma(\xi)\, d\xi.
	\end{align}
	We easily compute
	\begin{align*}
		\text{Res}\qty(e^{\xi \tau}\tilde \sigma(\xi),-\gamma)
		&=\lim\limits_{\xi \to -\g} (\xi+\gamma)\frac{1}{(\xi+\gamma)\tilde{I}(\xi)}=:c_\infty.
	\end{align*}
	Now, we analyze the integral (\ref{formulasigma(t)}). We decompose 
	\begin{align*}%\label{decintegral}
		\lim\limits_{R\to \infty}\int_{-\mathfrak{a}-i R}^{-\mathfrak{a}+iR} e^{\xi \tau}\tilde \sigma(\xi)\, d\xi	=i e^{-\mathfrak{a} \tau}\int_{-R}^{R} e^{iy \tau}\bigg[& \tilde \sigma(-\mathfrak{a}+iy) - \frac{d_{1,*}}{\sqrt{-\mathfrak{a}+iy}}-\frac{d_{2,*}}{(-\mathfrak{a}+iy)^{3/2}}\\\nonumber
		&	-	\frac{d_{3,*}}{(-\mathfrak{a}+iy)^2}	\bigg]\ddy
		\\\nonumber
		&+ie^{-\mathfrak{a}\tau}\int_{-iR}^{iR} e^{\xi t} \frac{d_{1,*}}{\sqrt{-\mathfrak{a}+iy}}\ddy\\\nonumber
		&+ie^{-\mathfrak{a}\tau}\int_{-i R}^{iR}\frac{d_{2,*}}{(-\mathfrak{a}+iy)^{3/2}} \ddy\\\nonumber
		&+ie^{-\mathfrak{a}\tau}\int_{-i R}^{iR}\frac{d_{3,*}}{(-\mathfrak{a}+iy)^{2}} \ddy
	\end{align*}
	It is easy to see (by means of another contour to avoid the standard branch) that, up to constants, the last three integral are respectively the inverse Laplace transform of $\xi^{-1/2},\xi^{-3/2},\xi^{-2}$.
	The integral
	\begin{align*}
		R(\tau){\coloneqq}\int_{-R}^{R} e^{iy \tau}\bigg[ \tilde \sigma(-\mathfrak{a}+iy) - \frac{d_{1,*}}{\sqrt{-\mathfrak{a}+iy}}-\frac{d_{2,*}}{(-\mathfrak{a}+iy)^{3/2}}	-	\frac{d_{3,*}}{(-\mathfrak{a}+iy)^2}	\bigg]\ddy
	\end{align*}
	is absolutely convergent thanks to the second order expansion of $\tilde{\sigma}(\xi)$. In fact, obtaining the absolute convergence of $R(\tau)$ (and $R'(\tau)$) is the main reason to use the sharp Varadhan's estimate on the heat kernel $p_t^{\Omega}$. 
	Thus, from (\ref{formulasigma(t)}) we obtain
	\begin{align*}
		\sigma(\tau)=c_\infty e^{-\g \tau}+e^{-\mathfrak{a} \tau} \qty[\frac{C_{1,*}}{\sqrt{\tau}}+C_{2,*}\sqrt{\tau}+C_{3,*}\tau +R(\tau)],
	\end{align*}
	for some constants $c_\infty, C_{i,*}$ and $i=1,2,3$, where $R(\tau)$ is bounded.
	This gives the asymptotic behavior
	\begin{align*}
		\sigma(\tau)=
		\begin{cases}
			c_\infty e^{-\g \tau} +O(e^{-\mathfrak{a} \tau}) &\text{for} \quad \tau \to \infty,\\
			\frac{C_{1,*}^{-1}}{\sqrt{\tau}}+c_\infty+O\qty(\sqrt{\tau}) &\text{for} \quad \tau \to 0^+,
		\end{cases}
	\end{align*}
	for any $\mathfrak{a} \in (\g,\lambda_1)$. 
	For later purposes, we observe that $\sigma(\tau)$ is differentiable. Indeed, differentiating $R(\tau)$, we still obtain an absolutely convergent integral thanks to the full expansion (\ref{expsigmatilde}), and an application of the dominated convergence theorem gives $\sigma \in C^1$ with 
	\begin{align*}
		\sigma'(\tau)=
		\begin{cases}
			-\g c_\infty e^{- \g\tau} +O(e^{-\mathfrak{a} \tau}) &\text{for} \quad \tau \to \infty,\\
			-(2C_{1,*})^{-1}\tau^{-3/2}(1+O(\tau)) &\text{for} \quad \tau \to 0^+,
		\end{cases}
	\end{align*}
	From (\ref{formulatildebeta}), taking the inverse Laplace transform of both sides, we get
	\begin{align*}
		\beta(s)e^{-\g s}= \beta(0)e^{-\g s} + \int_{0}^{s} e^{-\g (s-\tau)}h_0^*(s-\tau) \sigma(\tau)d\, \tau,
	\end{align*}
	that is
	\begin{align*}
		\beta(s)= \beta(0)+\int_{0}^{s}e^{\g \tau}\sigma(\tau)h_0^*(s-\tau)\, d\tau.
	\end{align*}
	\begin{proof}[Proof of (\ref{estinvLambda})]	
		We rewrite this formula as
		\begin{align*}
			\beta(s)=&\beta(0)+c_\infty \int_{0}^{s}h_0^*(\tau)\,d\tau+ \int_{0}^{s}h_0^*(\tau)\qty[e^{\g (s-\tau)}\sigma(s-\tau)-c_\infty]\,d\tau\\
			=&\qty[\beta(0)+c_\infty \int_0^\infty h_0^*(\tau)\, d\tau] -c_\infty \int_{s}^{\infty} h_0^*(\tau)\, d\tau\\
			&+ \int_{0}^{s}h_0^*(\tau)\qty[e^{\g (s-\tau)}\sigma(s-\tau)-c_\infty]\,d\tau.
		\end{align*}
		We choose $\beta(0)=-c_\infty \int_{0}^{\infty}h_0^*(\tau)d\tau$. It remains to estimate
		\begin{align*}
			&\beta_1(s){\coloneqq}-c_\infty \int_{s}^{\infty}h_0^*(\tau)\, d\tau,\\
			&\beta_2(s){\coloneqq}  \int_{0}^{s}h_0^*(\tau)\qty[e^{\g (s-\tau)}\sigma(s-\tau)-c_\infty]\,d\tau.
		\end{align*}
		We recall that the extension $h_0^*(s)$ has been selected so that (\ref{normh0*}) holds. Here and in what follows, without losing in generality we assume the same value $c=c_i$ for $i=1,2$. When we estimate the $L^\infty$ norm of $\beta$ we will only use the $L^\infty$ norm of $h_0^*$ and hence we get the same $L^\infty$-weight constant $c_1$. Instead, when we estimate the $C^{1/2+\ve}$ we need both the $L^\infty$ and $C^\ve$ norms of $h_0^*$, thus we will get the same $C^\ve$-weight constant $c_2=\min\{c_1,c_2\}$. Thus, conditionally to $c_i<(\lambda_1-\g)/(2\g)$, the weight constant $c_i$ with $i=1,2$ for $\beta$ and $h_0^*$ are respectively the same.
		We proceed with the $L^\infty$ estimate of $\beta$. We have
		\begin{align*}
			\abs{\beta_1(s)}&\lesssim \norm{h_0^*}_{\sharp,c,\ve}\int_{s}^{\infty}e^{-2\g c \tau}\, d\tau\\
			&\lesssim  \norm{h_0^*}_{\sharp,c,\ve} \mu_0(s)^{c},
		\end{align*}
		Using hypothesis (\ref{condc1c2lambda1}) and selecting $\mathfrak{a}$ close enough to $\lambda_1$ so that
		\begin{align}\label{condc,a,lambda_1}
			c<\mathfrak{a}<\frac{\lambda_1-\g}{2\g},
		\end{align} 
		we get
		\begin{align*}
			\abs{\beta_2(s)}&\lesssim \norm{h_0^*}_{\sharp,c,\ve} \int_{0}^{s} e^{-2 \g c\tau}e^{-\mathfrak{a} (s-\tau)}\dds \\
			&\lesssim  \norm{h_0^*}_{\sharp,c,\ve}e^{-\min\{2\g c,\mathfrak{a}\} s}\\
			&\lesssim \norm{h_0^*}_{\sharp,c,\ve}\mu_0(s)^{c}.
		\end{align*}
		Combining the bounds on $\beta_1$ and $\beta_2$ we obtain
		\begin{align}\label{betaLinfEst}
			\abs{\beta(s)}\lesssim \norm{h_0^*}_{\sharp, c,\ve}\mu_0(s)^c.
		\end{align}
		Now we estimate the $\qty(1/2+\ve)$-H\"older seminorm.
		In the following it is enough to assume $\eta \in (0,1)$. We have
		\begin{align}\label{HolderBoundbeta1}
			\abs{\beta_1(s)-\beta_1(s-\eta)}&\leq \abs{\int_{s-\eta}^s h_0^*(\tau)\,d \tau} \\ \nonumber
			&\leq \norm{h_0^*}_{\infty,c}\mu_0(s)^{c}\abs{\eta}\\	\nonumber
			&\leq \norm{h_0^*}_{\infty,c}\mu_0(s)^{c}\abs{\eta}^{\frac{1}{2}+\ve}
		\end{align} 
		Let 
		$$
		l(\tau){\coloneqq}e^{\g \tau}\sigma(\tau)-c_\infty.
		$$
		Following the classical fractional integral estimate of Hardy and Littlewood \cite[Theorem 14]{HardyLitt}, we decompose
		\begin{align*}
			\beta_2(s)-\beta_2(s-\eta)=&\int_0^s h_0^*(s-\tau)l(\tau) \dtau - \int_{0}^{s-\eta}h_0^*(s-\eta-\tau)l(\tau)\dtau\\
			=&h_0^*(s)\int_{0}^{s}l(\tau)\dtau - \int_{0}^{s}\qty[h_0^*(s)-h_0^*(s-\tau)]l(\tau)\dtau\\
			&-h_0^*(s)\int_0^{s-\eta}l(\tau)\dtau   -\int_{0}^{s-\eta}\qty[h_0^*(s-\eta-\tau)-h_0^*(s)]l(\tau)\dtau\\
			=&h_0^*(s)\int_{s-\eta}^{s}l(\tau)\dtau-\int_{0}^{\eta}\qty[h_0^*(s)-h_0^*(s-\tau)]l(\tau)\dtau \\&-\int_{\eta}^{s}\qty[h_0^*(s)-h_0^*(s-\tau)]\qty(l(\tau)-l(\tau-\eta))\dtau \\
			=:&A_1(s,\eta)+A_2(s,\eta)+A_3(s,\eta).
		\end{align*}
		For $s-\eta \in (\eta,1)$ we have
		\begin{align*}
			\abs{A_1}
			&\lesssim \abs{h_0^*(s)}\int_{s-\eta}^{s}\frac{1}{\sqrt{\tau}}\dtau \\
			&\lesssim \abs{h_0^*(s)}\qty(s^{1/2}-(s-\eta)^{1/2}) \\
			&\lesssim [h_0^*]_{0,\ve,[s,s+1]} s^{\ve-\frac{1}{2}}\eta \\
			&\lesssim \norm{h_0^*}_{\sharp,c,\ve} \mu(s)^c\eta^{\ve+\frac{1}{2}}.
		\end{align*}
		For $s-\eta\geq 1$ we get
		\begin{align*}
			\abs{A_1}&\leq \abs{h_0^*(s)}\int_{s-\eta}^{s} l(\tau)\,d\tau \\
			&\lesssim \abs{h_0^*(s)}\int_{s-\eta}^{s}e^{-\mathfrak{a} \tau}\,d\tau \\
			&\lesssim  \abs{h_0^*(s)} \eta\\
			&\lesssim  \norm{h_0^*}_{\sharp,c,\ve}\mu(s)^c \eta^{\frac{1}{2}+\ve}.
		\end{align*}
		For $s-\eta \in (0,\eta)$ we obtain
		\begin{align*}
			\abs{A_1}\lesssim& \abs{h_0^*(s)}\int_{s-\eta}^{s} \frac{1}{\sqrt{\tau}}\dtau \\
			\lesssim&  [h_0^*]_{0,\ve,[s-\eta,s-\eta+1]} \abs{s-\eta}^{\ve} \eta^{\frac{1}{2}}\\
			\lesssim& \norm{h_0^*}_{\sharp,c,\ve} \mu(s)^c\eta^{\frac{1}{2}+\ve}.
		\end{align*}
		Now we estimate $A_2$. We have
		\begin{align*}
			\abs{A_2}&\leq \norm{h_0^*}_{\sharp,c,\ve} \mu(s)^c \int_{0}^{\eta} \abs{\tau}^{\ve} \abs{l(\tau)} \, d\tau \\
			&\lesssim \norm{h_0^*}_{\sharp,c,\ve} \mu(s)^c \int_{0}^{\eta} \tau^{\ve} \frac{1}{\sqrt{\tau}} \dtau \\
			&\lesssim \norm{h_0^*}_{\sharp,c,\ve} \mu(s)^c \eta^{\frac{1}{2}+\ve}.
		\end{align*}
		Finally, we estimate $A_3$. Using the $L^\infty$ norm of $h_0^*$ for $\tau>1$ and $C^{\ve}$ seminorm for $\tau<1$ we obtain
		\begin{align*}
			\abs{A_3}
			\lesssim&\int_{\eta}^{s}\abs{h_0^*(s)-h_0^*(s-\tau)}\abs{l(\tau)-l(\tau-\eta)}\dtau\\
			\lesssim&\norm{h_0^*}_{\sharp,c,\ve}\int_{\eta}^{s} \abs{\tau}^{\ve} 		\abs{l(\tau)-l(\tau-\eta)}\dtau \\
			\lesssim&\norm{h_0^*}_{\sharp,c,\ve} \int_{\eta}^{s}\abs{\tau}^{\ve} [\tau^{-1/2}-(\tau-\eta)^{-1/2}] \dtau\\ 
			\lesssim&\norm{h_0^*}_{\sharp,c,\ve} \eta \int_{\eta}^{s}\abs{\tau}^{\ve} \tau^{-3/2} \dtau  \\
			\lesssim& \norm{h_0^*}_{\sharp,c,\ve} \eta^{\frac{1}{2}+\ve}  \\
			\lesssim& \norm{h_0^*}_{\sharp,c,\ve} \mu(s-1)^c \eta^{\frac{1}{2}+\ve}\\
			\lesssim& \norm{h_0^*}_{\sharp,c,\ve} \mu(s)^c \eta^{\frac{1}{2}+\ve},
		\end{align*}
		Combining the bounds on $A_1,A_2,A_3$ and  we obtain
		\begin{align}\label{betaCalphaEst}
			\abs{\beta_2(s)-\beta_2(s-\eta)}\lesssim \norm{h_0^*}_{\sharp,c,\ve}\mu(s)^c\abs{\eta}^{\frac{1}{2}+\ve}.
		\end{align}
		Finally, from (\ref{betaLinfEst}), (\ref{HolderBoundbeta1}) and (\ref{betaCalphaEst}) we obtain
		\begin{align*}
			\norm{\beta}_{\sharp,c,\frac{1}{2}+\ve}\lesssim \norm{h_0^*}_{\sharp,c,\ve} 
		\end{align*}
		Going back to the original variable $t$ using (\ref{covbetaLambda}), we obtain 
		\begin{align*}
			\norm{\Lambda}_{\sharp,c,\frac{1}{2}+\ve}\lesssim \norm{h_0^*}_{\sharp,c,\ve}, 
		\end{align*}
		and recalling (\ref{normh0*}) the proof of (\ref{estinvLambda}) is complete. 
	\end{proof}
	We proceed to prove the second part of Proposition \ref{Proposition:invj}: in case $h\in X_{\sharp,c,\frac{1}{2}+\ve}$, then $\Lambda$ is differentiable and $\dot \Lambda \in X_{\sharp,c,\ve}$.
	\begin{proof}[Proof of (\ref{LinEstdotLambda<h})]
		In the same notation of the previous lemma, we need to prove that $\beta_1(s),\beta_2(s)$ are differentiable and estimate the derivatives. Since
		\begin{align*}
			\beta_1(s){\coloneqq}-\int_{s}^{\infty}h_0^*(\tau)\dtau,
		\end{align*}
		we clearly have $\beta_1(s)\in C^{1}(0,\infty)$ and $\beta_1'(s)=c_\infty h(s)\in X_{\svvc}$ by hypothesis. To analyze $\beta_2$, following \cite[Theorem 19]{HardyLitt}, we introduce for any $\epsilon\geq 0$ the function
		\begin{align*}
			\beta_{2,\eps}(s)= \int_{0}^{s-\eps} h_0^*(\tau) l(s-\tau)\dtau,
		\end{align*}
		so that $\beta_{2,0}(s)=\beta_2(s)$. Since $\sigma(\tau)\in C^1$, we can differentiate $\beta_{2,\eps}(s)$ to obtain
		\begin{align*}
			\beta'_{2,\eps}(s)=&h_0^*(s-\eps)l(\eps)+\int_0^{s-\eps} h_0^*(\tau) l'(s-\tau)\, d\tau\\
			=&-[h_0^*(s)-h_0^*(s-\eps)]l(\eps)+l(s-\eps)h_0^*(s)\\
			&+\int_0^{s-\eps} [h_0^*(\tau)-h_0^*(s)] l'(s-\tau)\, d\tau.
		\end{align*}
		Observe that we can choose the extension $h_0^*$ such that $h_0^*(s)=o(s^{1/2})$ for $s\to 0$. Since $h_0^*\in X_\svvc$, when $\eps\to 0$ the right-hand side tends uniformly to
		\begin{align*}
			l(s)h_0^*(s)+g(s),
		\end{align*}
		where
		\begin{align*}
			g(s){\coloneqq}\int_0^{s} [h_0^*(\tau)-h_0^*(s)]l'(s-\tau)\dds.
		\end{align*}
		By hypothesis and the choice of the extension we have $l(s)h_0^*(s)\in X_{\sharp,c,\frac{1}{2}+\ve}$. Also, the function $g(s)$ is continuous since $h_0^*(s)\in C^{\frac{1}{2}+\ve}$. 
		\begin{align*}
			\beta_2(s_1)-\beta_2(s_2)&=\lim\limits_{\eps \to 0}\qty(\beta_{2,\eps}(s_1)-\beta_{2,\eps}(s_2))\\
			&=\lim\limits_{\eps \to 0}\int_{s_1}^{s_2} \beta_{2,\eps}'(\tau)\dtau \\
			&=\int_{s_1}^{s_2}l(\tau)h_0^*(\tau)+g(\tau)\dtau,
		\end{align*} 
		hence
		\begin{align*}
			l(s)h_0^*(s)+g(s)=\beta_2'(s).
		\end{align*}
		It remains to prove that $g(s)\in X_\svc$. 
		Using the asymptotic of $\sigma'(t)$ and the assumption (\ref{condc1c2lambda1}) with $\mathfrak{a}$ as in (\ref{condc,a,lambda_1}) we have
		\begin{align}\label{beta'estInf}
			\abs{g(s)}\lesssim & [h]_{0,\frac{1}{2}+\ve,[s-1,s]}\int_{s-1}^s l'(s-\tau) \abs{s-\tau}^{\frac{1}{2}+\ve} \dtau\\\nonumber
			&+\norm{h}_{\sharp,c,\frac{1}{2}+\ve}\int_0^{s-1}l'(s-\tau)\mu(\tau)^c \, d\tau\\\nonumber
			\lesssim &\norm{h}_{\sharp,c,\frac{1}{2}+\ve}	\qty[\mu(s)^c \int_{0}^1 \abs{w}^{-1+\ve} \, dw +\int_0^{s}e^{-2\g c \tau} e^{-\mathfrak{a}(s-\tau)} ]\, d\tau\\\nonumber
			\lesssim & \norm{h}_{\sharp,c,\frac{1}{2}+\ve}\mu(s)^c.
		\end{align}
		We write
		\begin{align*}
			g(s-\eta)-g(s)=&\int_0^s \qty[h(s)-h(\tau)]l'(s-\tau) \,d\tau - \int_{0}^{s-\eta} \qty[h(s-\eta)-h(\tau)]l'(s-\eta-\tau)\,d\tau\\
			=& \int_0^{s}\qty[h(s)-h(s-u)]l'(u)\ddu - \int_\eta^s [h(s-\eta)-h(s-u)]l'(u-\eta) \ddu \\
			=& -\int_{\eta}^s[h(s-\eta)-h(s-u)]\qty[l'(u-\eta)-l'(u)] \ddu \\
			&+\int_{\eta}^{s}[h(s)-h(s-\eta)]l'(u)\ddu + \int_{0}^\eta [h(s)-h(s-u)]l'(u)\ddu \\
			=:&B_1(s,\eta)+B_2(s,\eta)+B_3(s,\eta).
		\end{align*}
		Using again assumption (\ref{condc1c2lambda1}) we get
		\begin{align*}
			\abs{B_1}\lesssim & \norm{h_0^*}_{0,\frac{1}{2}+\ve,[s-1,s]} \int_{\eta}^1 \abs{u-\eta}^{\frac{1}{2}+\ve}\abs{(u-\eta)^{-3/2}-u^{-3/2}}\ddu\\
			& +\norm{h}_{\sharp,c,\frac{1}{2}+\ve} \int_{1}^s \mu(s-u)\eta \frac{e^{-\mathfrak{a}(u-\eta)}-e^{-\mathfrak{a}u}}{\eta}	\ \ddu \\
			&\lesssim \norm{h_0^*}_{\sharp,c,\frac{1}{2}+\ve}\mu(s)^c \eta^{\ve}.
		\end{align*}
		Also
		\begin{align*}
			\abs{B_2}&\lesssim  \abs{h_0^*(s)-h_0^*(s-\eta)}\eta^{-1/2}\\
			&\lesssim \norm{h_0^*}_{\sharp,c,\frac{1}{2}+\ve}\mu(s)^c\eta^{\eps},
		\end{align*}
		and
		\begin{align*}
			\abs{B_3}&\lesssim \norm{h_0^*}_{\sharp,c,\frac{1}{2}+\ve}\mu(s)^c\int_{0}^{\eta}u^{-1+\ve}\ddu\\
			&\lesssim \norm{h_0^*}_{\sharp,c,\frac{1}{2}+\ve}\mu(s)^c\eta^{\ve}.
		\end{align*}
		This proves
		\begin{align*}
			\abs{g(s)-g(s-\eta)}\lesssim \mu(s)^c\norm{h_0^*}_{\sharp,c,\frac{1}{2}+\ve} \abs{\eta}^{\ve}.
		\end{align*}
		Combining it with (\ref{beta'estInf}) we obtain
		\begin{align*}
			\norm{g}_{\sharp,c,\ve}\lesssim \norm{h_0^*}_{\sharp,c,\frac{1}{2}+\ve}.
		\end{align*}
		Summing up the estimates for $\beta_1'(s)$ and $\beta_2'(s)=l(s) h_0^*(s)+g(s)$ we obtain
		\begin{align*}
			\norm{\beta'(s)}_\svc \lesssim \norm{h_0^*}_\svvc.
		\end{align*}
		Finally, in the original variable $t$, using (\ref{covbetaLambda}) and (\ref{normh0*}), we obtain the bound (\ref{LinEstdotLambda<h}).
	\end{proof}
	
	\begin{remark}[the initial datum $J_1(x,t_0)$]
		From the proof of Proposition \ref{Proposition:invj} we have $\JJJ(t_0,x)=\int_0^1 h^*(s) I(x,\tau-s)\dds$ where $h_0^*$ is an arbitrary smooth function with $h_0^*(t)=o(t^{1/2})$ for $t\to 0$ and $h_0^*(1)=h(t_0)$, connecting to $h(t)$ at $t=t_0$ to maintain the $C^{\ve}$ regularity of $h$. We observe by estimate (\ref{EstJ1fromdotLambda}) that 
		$$
		\norm*{J_1(\cdot ,t_0)}_{L^\infty(\Om)}\lesssim \norm*{\JJJ[\dot \Lambda](\cdot ,t_0)}_{L^\infty(\Om)}\lesssim \abs*{\dot \Lambda(t_0)}\lesssim \mu_0(t_0)^{l_1}.
		$$ 
		Thus, our initial datum remains positive provided that $t_0$ is fixed sufficiently large.
	\end{remark}
	\appendix
	
	\section{Properties of the Robin function $H_\gamma(x,x)$.}\label{app:propH}
	In this appendix we prove some properties of the Robin function that we use in our construction. We recall that the Green function associated to the operator $-\Delta-\gamma$ satisfies
	\begin{align}\label{eqGgamma}
		&-\Delta_x G_\gamma(x,y)-\gamma G_{\gamma}(x,y)={4\pi  \alpha_3} \delta(x-y) \inn \Omega,\\\nonumber
		&G(\cdot,y)=0 \onn \partial\Omega.
	\end{align}
	As usual, we split $$G_\gamma(x,y)=\Gamma(x-y)-H_\gamma(x,y)\quad \text{where}\quad \Gamma(x)=\frac{\at}{\abs{x}},$$ 
	and the regular part $H_\gamma(x,y)$ satisfies
	\begin{align*}
		&-\Delta_x H_\gamma(x,y)-\gamma H_\gamma(x,y)=-\gamma \Gamma(x-y)\inn \Omega,\\
		&H_\gamma(\cdot,y)=\Gamma(\cdot-y)\onn \partial \Omega,
	\end{align*} 
	for any fixed $y\in \Omega$.
	We recall (from \cite{nearcrit} and reference therein) the following properties of $R_\g(x){\coloneqq}H_\gamma(x,x)$:
	\begin{enumerate}
		\item $R_\g(x)\in C^\infty(\Omega)$
		\item $\partial_\gamma R_\gamma(x)<0$ and belongs to $C^{\infty}(\Omega)$.
		\item for each $\gamma \in (0,\lambda_1)$ fixed, 
		$R_\gamma(x)\to +\infty$ as $x\to \partial \Omega$
	\end{enumerate}
	\begin{lemma}[Behavior near the first eigenvalue]\label{asycloselambda}
		The function $H_\gamma(x,y)$ satisfies
		\begin{align}\label{asygammatolambda1}
			H_\gamma(x,y)\sim -\frac{4\pi\alpha_3}{\lambda_1 - \gamma} \phi_1(y)\phi_1(x),\quad \ass \gamma \nnearrow \lambda_1.
		\end{align}
	\end{lemma}
	\begin{proof}
		We decompose $H_\gamma$ as
		\begin{align}\label{decH}
			H_\gamma(x,y)=\alpha(y) \phi_1(x)+H_0(x,y)+h_{\perp,\gamma}(x,y)
		\end{align}
		where
		\begin{align*}
			\alpha(y):= \int_{\Omega} \qty(H_\gamma(x,y)-H_0(x,y))\phi_1(x)\ddx,
		\end{align*}
		and $H_0$ satisfies
		\begin{align*}
			\Delta_x H_0(x,y)=0 \inn \Omega, \quad H_0(x,y)=\frac{\alpha_3}{\abs{x-y}} \onn \partial \Omega.
		\end{align*}
		Thus, for any fixed $y\in \Omega$, $h_{\perp,\gamma}(x,y)$ is the solution to
		\begin{align}\label{eqhperp}
			&\Delta_x h_{\perp,\gamma}+\gamma h_{\perp,\gamma}=\gamma G_0(x,y)+\alpha(y)\qty(\lambda_1-\gamma)\phi_1(x)\inn \Omega\\\nonumber
			&h_{\perp,\gamma}(x,y)=0\onn \partial \Omega.
		\end{align}
		By definition of $\alpha(y)$ we have
		\begin{align}\label{A4}
			\int_\Omega h_{\perp,\gamma}(x,y)\phi_1(x)\ddx &= \int_{\Omega}\qty(H_\gamma(x,y)-H_0(x,y))\phi_1(x)\ddx - \alpha(y)\norm{\phi_1}_2^2\\\nonumber
			&=0.
		\end{align}
		Testing \reff{eqGgamma} against $\phi_1$ we get
		\begin{align*}
			\int_{\Omega} G_\gamma(x,y)\phi_1(x)\ddx =\frac{4\pi \alpha_3   }{\lambda_1-\gamma}\phi_1(y).
		\end{align*}
		Also, testing (\ref{eqhperp}) against $\phi_1$ and using (\ref{A4}) we obtain
		\begin{align*}
			0=&(-\lambda_1+\gamma)\int_{\Omega} h_{\perp,\gamma}(x,y)\phi_1(x)\ddx\\
			=&\gamma \int_{\Omega} \phi_1(x)G_0(x,y)\ddx + \alpha(y)(\lambda_1-\gamma).
		\end{align*}
		Thus, we have
		\begin{align}\label{A5}
			\alpha(y)&=-\frac{\gamma}{\lambda_1-\gamma}\int_{\Omega}G_0(x,y)\phi_1(x)\ddx\\\nonumber
			&=-\frac{\gamma}{\lambda_1}\frac{4\pi \alpha_3 \phi_1(y)}{\lambda_1-\gamma},
		\end{align}
		and plugging (\ref{A5}) in (\ref{decH}) we obtain
		\begin{align}
			H_\gamma(x,y)=-\frac{\gamma}{\lambda_1}\frac{4\pi \alpha_3}{\lambda_1-\gamma}\phi_1(y)\phi_1(x)+H_0(x,y)+h_{\perp,\gamma}(x,y).
		\end{align}
		We notice that only the first and last term in the right-hand side depends on $\gamma$.
		Hence, to prove (\ref{asygammatolambda1}) we just need to prove that $h_{\perp,\gamma}(x,y)$ is bounded as $\gamma\to \lambda_1^{-}$. This is a consequence of the Poincar\'e inequality applied to functions in $H_0^1 $ which are orthogonal to $\phi_1$. Indeed, expanding $h_{\perp,\gamma}$ in the $L^2$-basis made of Laplacian eigenvalues we get
		\begin{align*}
			\norm{\nabla h_{\perp,\gamma}}_2^2&=\int_{\Omega} h_{\perp,\gamma}\qty(-\Delta h_{\perp,\gamma})\ddx    \\
			&=  \int_\Omega \qty(	\sum_{k\geq 2} \alpha_k \phi_k(x)	) \qty(  \sum_{k\geq 2} \alpha_k \phi_k(x) \lambda_k   ) \ddx \\
			&=\sum_{k\geq 2} \alpha_k^2 \lambda_k \\
			&\geq \lambda_2 \norm{h_{\perp,\gamma}}_2^2.
		\end{align*}
		Now, testing equation \reff{eqhperp} against $h_{\perp,\gamma}$, using (\ref{A4}) and Cauchy–Schwarz inequality we get
		\begin{align*}
			\qty(\lambda_2-\gamma) \norm{h_{\perp,\gamma}}_2^2
			&\leq \norm{\nabla h_{\perp,\gamma}}_2^2 - \gamma \norm{h_{\perp,\gamma}}_2^2\\ 
			&=\gamma \int_{\Omega} \qty(H_0(x,y)-\frac{\alpha_3}{\abs{x-y}}	)h_{\perp,\gamma}(x,y)\ddx \\
			&\leq \gamma \norm{H_0(\cdot,y)-\frac{\alpha_3}{\abs{\cdot -y}}}_2 \norm{h_{\perp,\gamma}}_2.
		\end{align*}
		We conclude that
		\begin{align*}
			\norm{h_{\perp,\gamma}}_2 &\leq \frac{\gamma}{\lambda_2-\gamma}\norm{H_0(\cdot,y)-\frac{\alpha_3}{\abs{\cdot -y}}}_2 \\&\leq \frac{\lambda_1}{\lambda_2 -\lambda_1}\norm{H_0(\cdot,y)-\frac{\alpha_3}{\abs{\cdot -y}}}_2,
		\end{align*}
		with the right-hand side independent of $\gamma$.
		By standard elliptic estimates we get 
		\begin{align*}
			\norm{h_{\perp,\gamma}(\cdot,y)}_\infty \leq K_\Om(y),
		\end{align*}
		with $K$ independent of $\gamma$. This concludes the proof.
	\end{proof}
	The following lemma gives the asymptotic behavior of $\gamma^*(x)$ as $x$ approaches the boundary $\pp \Omega$.
	\begin{lemma}\label{asymptgamma*}
		The unique number $\gamma^*(x)\in (0,\lambda_1)$ defined by the relation
		\begin{align*}
			H_{\gamma^*}(x,x)=0
		\end{align*}
		satisfies
		\begin{align}\label{asygamma}
			\gamma^*(x)\sim \lambda_1 - 8\pi \qty[\partial_\nu \phi_1(x')]^2 d(x,\partial \Omega)^3 \ass x\to x'\in \partial \Omega,
		\end{align}
		where $d(x,\partial \Omega)=\abs{x-x'}$.
	\end{lemma}
	\begin{proof}
		We divide the proof in two steps. Given $x\in \Omega$ let $D_x \subset  \partial \Omega$  the set of points $x'$ such that
		\begin{align*}
			\abs{x-x'}=d(x,\partial \Omega).
		\end{align*}
		If $D_x$ is not a singleton we choose the unique $x'=(x_1',x_2',x_3')$ with the property $x_i'\leq y_i'$ for all components $i=1,2,3$ and point $y'\in D_x$. This defines $x'{\coloneqq}x'(x)$ uniquely.

		\emph{Step 1}. Firstly we prove \reff{asygamma} for domains such that, for all $x\in \Omega$, the reflection point $x''(x){\coloneqq}2x'(x)-x$ satisfies
		\begin{align}\tag{P}\label{prP}
			x''\notin \Omega.
		\end{align}
		We decompose
		\begin{align}\label{dechgamma}
			H_\gamma(x,y)=\frac{\alpha_3}{\abs{x''-y}}+F(x,y),
		\end{align}
		where $F$ satisfies
		\begin{align}\label{eqF}
			&\Delta_x F+ \gamma F = \gamma \alpha_3 g_1(x,y)\inn \Omega,\\\nonumber
			&F(x,y)=0 \onn \partial \Omega,
		\end{align}
		and
		\begin{align*}
			g_1(x,y):= \frac{1}{\abs{x-y}}- \frac{1}{\abs{x''-y}}	
		\end{align*}
		We write
		\begin{align*}
			F(x,y)=\alpha(y)\phi_1(x)+w_{\perp}(x,y)
		\end{align*}
		and select $\alpha(y)$ so that $\int_{\Omega} w_\perp(x,y) \phi_1(x)\ddx =0$. By decomposition \reff{dechgamma} and \reff{gdec} we obtain
		\begin{align*}
			\alpha(y)&=\int_\Omega \qty(F(x,y) - w_\perp(x,y)) \phi_1(x)\ddx\\&=\int_\Omega \qty(H_\gamma(x,y) - \frac{\alpha_3}{\abs{x-y''}})\phi_1(x)\ddx \\&= \int_{\Omega} g_1(x,y) \phi_1(x)\ddx - \int_{\Omega} G_\gamma(x,y)\phi_1(x)\ddx \\&=
			\int_{\Omega} g_1(x,y) \phi_1(x)\ddx +\frac{4\pi \alpha_3 \phi_1(y)}{\gamma - \lambda_1}
		\end{align*}
		The equation for $w_\perp$ is
		\begin{align*}
			\Delta w_{\perp}+\g w_{\perp} =\alpha(y)(\lambda_1-\g)\phi_1+\g \alpha_3 g_1 (x)
		\end{align*}
		Multiplying this equation by $w_{\perp}$ and integrating by parts we get
		\begin{align*}
			\norm{\nabla w_{\perp}(\cdot y)}_2^2- \gamma \norm{w_{\perp}(\cdot,y)}_2^2 = -\gamma \alpha_3 \int_{\Omega}g_1(x,y)w(x,y)\ddx.
		\end{align*}
		Using the improved Poincar\'e inequality 
		\begin{align*}
			(\lambda_2-\g)\norm{w_{\perp}}_2^2\leq\norm{\nabla w_{\perp}(\cdot, y)}_2^2- \gamma \norm{w_{\perp}(\cdot,y)}_2^2
		\end{align*}
		and Cauchy–Schwarz we obtain
		\begin{align}\label{boundwg2}
			\norm{w_{\perp}(\cdot,y)}_2\leq \frac{\gamma}{\lambda_2-\gamma}\alpha_3\norm{g_1(\cdot,y)}_2< \frac{\lambda_1}{\lambda_2-\lambda_1}\alpha_3 \norm{g_1(\cdot, y)}_2.	
		\end{align}
		Now, we want to estimate uniformly in $y$ the right-hand side of
		\begin{align*}
			H_\gamma(x,y)=\frac{\alpha_3}{\abs{ x''-y}}+\phi_1(x)\int_\Omega g_1(z,y) \phi_1(z) \ddz - \frac{4\pi \alpha_3 \phi_1(y)\phi_1(x)}{\gamma - \lambda_1}+w_\perp(x,y).
		\end{align*}
		Without loss of generality, suppose $0\in \Omega$. Let $M{\coloneqq}2\text{diam}(\Omega)$ we have
		\begin{align*}
			0<\int_{\Omega} \frac{1}{\abs{x-y}^2}\ddx \leq \int_{B_{M(\Omega)}(y)}  \frac{1}{\abs{x-y}^2} \ddx \leq  C_\Omega.
		\end{align*}
		Let $\Omega''=\qty{	x'' \in \RR^3 : x''=x''(x) \text{ for some } x\in \Omega	}$. We have
		\begin{align*}
			0<\int_{\Omega} \frac{1}{\abs{x''-y}^2}\ddx\leq \int_{\Omega^{''}\cup \Omega}\frac{1}{\abs{x-y}^2}\leq \int_{B_{M_2}}\leq C_{\Omega},
		\end{align*}
		where $M_2=2\diam(\Omega^{''}\cup \Omega)$
		hence we get
		\begin{align*}
			\sup_{y\in \Omega}\norm{g_1(\cdot,y)}_2< C_{\Omega}.
		\end{align*}
		We combine this bound with \reff{boundwg2} to get
		\begin{align*}
			\norm{w_\perp(\cdot,y)}_2\leq K_\Omega,
		\end{align*}
		with $K_\Om$ independent of $y$ and by standard elliptic estimates we get 
		$$
		\sup_{y \in \Omega} \norm{w(\cdot,y)}_\infty \leq K_\Om,
		$$
		with a possibly larger constant $K_\Om$. We conclude that 
		\begin{align}\label{expHgammacloseboundary}
			H_\gamma(x,y)=\frac{\alpha_3}{\abs{x''-y}}+\frac{4\pi \alpha_3 \phi_1(y)\phi_1(x)}{\gamma-\lambda_1}+ \phi_1(x)B(y) + w_\perp(x,y),
		\end{align}
		where
		\begin{align*}
			B(y){\coloneqq}\int_{\Omega}g_1(z,y)\phi_1(z)\ddz,
		\end{align*}
		with $w_\perp(x,y)$ bounded in $\Omega \times \Omega$. Also we notice that
		\begin{align*}
			&0< \int_{\Omega} \phi_1(z) \frac{1}{\abs{z-y}}\ddz \leq \norm{\phi_1}_\infty \int_{B_{M(\Omega)}}  \frac{1}{\abs{z-y}}\ddz \leq C_{\Omega},
		\end{align*}
		and
		\begin{align*}
			0<\int_\Omega \frac{\phi_1(x)}{\abs{ x''(x)-y}}\ddx \leq \norm{\phi_1}_\infty \int_{B_{M''}}\frac{1}{\abs{x-y}} \ddx \leq C_\Omega.
		\end{align*}
		This proves the boundedness of $B(y)$. Now, the equation for $\gamma^*(x)$ reads as
		\begin{align*}
			0=\frac{\alpha_3}{d(x,\partial \Omega)}+ \frac{4\pi \alpha_3 \phi_1(x)^2}{\gamma^*(x)-\lambda_1} +\phi_1(x)B(x)+w_\perp(x,x).
		\end{align*}
		Let $c{\coloneqq}\abs{\partial_\nu \phi_1(x')}$. We expand $\phi_1(x)$ at $x'\in \partial \Omega$ to get
		\begin{align*}
			\frac{8\pi c^2 d(x,\partial \Omega)^3 }{\lambda_1 - \gamma^*(x)}=\qty[1+2c d(x,\partial \Omega)^2  B(x)+2 d(x,\partial \Omega) w(x,x)]\qty(1+O\qty(d(x,\partial \Omega))	)
		\end{align*}
		Since $B(x)$ and $w(x,x)$ are bounded, we conclude that
		\begin{align}\label{finalasygamma}
			\frac{8\pi c^2 d(x,\partial \Omega)^3}{\lambda_1-\gamma^*(x)}\sim 1 \quad \ass x\to x'\in \partial \Omega.
		\end{align}
		\emph{Step 2}. Now, we modify the method in Step 1 to obtain an expansion similar to \reff{expHgammacloseboundary} and conclude that \reff{asygamma} is true for general smooth bounded domains. Let $y\in \Omega_{\epsilon/4}$. Now we prove \reff{asygamma} for all smooth domains $\Omega$. Fix $\epsilon=\epsilon(\Omega)>0$ so small that the set $\Omega_\epsilon{\coloneqq}\{x\in \Omega : d(x,\partial \Omega)<\epsilon\}$ possesses the property \reff{prP} and let $\eta_\epsilon$ be a smooth cut-off function with $\supp(\eta_\epsilon)\subset \Omega_\epsilon$ and $\eta_\epsilon(x)\equiv 1$ for $x\in \Omega_{\epsilon/2}$. We write
		\begin{align*}
			H_\gamma(x,y)=&\eta_{\epsilon}(x)\eta_{\epsilon}(y)H_\gamma(x,y)+\qty(1-\eta_{\epsilon}(x)\eta_{\epsilon}(y))H_\gamma(x,y)\\=&
			\frac{\alpha_3}{\abs{x''-y}}\eta_{\epsilon}(x)\eta_{\epsilon}(y)+F_2(x,y)
		\end{align*}
		where
		\begin{align*}
			F_2(x,y)=\eta_{\epsilon}(x)\eta_{\epsilon}(y)F(x,y)+\qty(1-\eta_{\epsilon}(x)\eta_{\epsilon}(y))H_\gamma(x,y).
		\end{align*}
		We notice that $\eta_{\epsilon}(x)\eta_{\epsilon}(y)F(x,y)$, where $F$ satisfies (\ref{eqF}), is well-defined in $\Omega$ thanks to the cut-off functions. The problem for $F_2$ is
		\begin{align*}
			&\Delta_x F_2(x,y) + \gamma F_2(x,y) =\alpha_3 g_2(x,y)\inn \Omega,\\
			&F_2(x,y)=0\onn \partial \Omega,
		\end{align*}
		where
		\begin{align*}
			g_2(x,y){\coloneqq}g_{2,1}(x,y)+g_{2,2}(x,y)+g_{2,3}(x,y)+g_{2,4}(x,y),
		\end{align*}
		and
		\begin{align*}
			&g_{2,1}(x,y){\coloneqq}\frac{\gamma}{\abs{x-y}},\\& g_{2,2}(x,y){\coloneqq} -\gamma \frac{\eta_{\epsilon}(x)\eta_{\epsilon}(y)}{\abs{x''-y}},\\
			& g_{2,3}(x,y){\coloneqq}2 \eta_\epsilon(y)\frac{\ddiv \eta_\epsilon(x)}{\abs{x''-y}^3},\\
			& g_{2,4}(x,y){\coloneqq}-\eta_{\epsilon}(y) \frac{\Delta_x \eta(x)}{\abs{x''-y}}.
		\end{align*}
		We decompose
		\begin{align*}
			F_2(x,y)= \beta(y)\phi_1(x)+w_2(x,y),
		\end{align*}
		where $\beta$ is chosen such that $\int_\Om w_2(x,y)\phi_1(x)\ddx =0$, that gives
		\begin{align*}
			\beta(y)=\int_{\Omega}F_2(x,y)\phi(x)\ddx =& \int_{\Omega}\phi_1(x)\qty[	-G_{\gamma}(x,y)-\frac{\eta_\epsilon(x)\eta_{\epsilon}(y)}{\abs{x''-y}}
			+\frac{\alpha_3}{\abs{x-y}}	] \ddx \\=& \frac{4\pi \phi_1(y)}{-\lambda_1+\gamma}\int_{\Omega} \frac{\alpha_3 \phi_1(x)}{\abs{x-y}}\ddx - \eta_{\epsilon}(y)\int_{\Omega_\epsilon}\frac{\alpha_3 \eta_\epsilon(x)}{\abs{x''-y}}.
		\end{align*}
		Next we prove that $w_{2}(x,y)$ is uniformly bounded in $\Omega\times \Omega$. Using the improved Poincar\'e inequality and standard elliptic estimates as in Step 1, we reduce the problem to estimate the $L^2$-norm of $g(\cdot,y)$ uniformly in $y\in \Omega_{\epsilon/4}$.
		We have
		\begin{align*}
			&\norm{g_{2,1}}_2^2 = \gamma \int_{\Omega}\frac{1}{\abs{x-y}^2}\ddx \leq \gamma \int_{B_M} \frac{1}{\abs{x-y}^2} \ddx \leq C_{\Omega},\\
			&\norm{g_{2,2}}_2^2 \leq \gamma \int_{\Omega_\epsilon}\frac{1}{\abs{ x''-y}} \ddx \leq \gamma \int_{\Omega^{''}}\frac{1}{\abs{x-y}}\ddx \leq \gamma \int_{B_{M_2}} \frac{1}{\abs{ x-y}^2}\leq C_\Omega, \\
			&\norm{g_{2,4}}_2^2 \leq  \int_{\Omega_\epsilon\setminus \Omega_{\epsilon/2}} \frac{\abs{\Delta \eta_\epsilon(x)}}{\abs{ x''-y}^2}\leq C \epsilon^{-2} \norm{g_{2,2}}_2^2,\\
			&\norm{g_{2,3}}_2^2\leq C \int_{\Omega_\epsilon\setminus \Omega_{\epsilon/2}}  \frac{1}{\abs{ x''-y}^4}\leq C_\Omega \epsilon^{-4} \abs{ \Omega}.
		\end{align*}
		Since $\epsilon$ depends only on $\Omega$ we obtain
		\begin{align*}
			\norm{g_2(\cdot,y)}_2^2< C_{\Omega,\epsilon}.
		\end{align*}
		Now, we prove the boundedness of 
		\begin{align*}
			B_\epsilon(y):= \underbrace{\int_{\Omega} 	\frac{1}{\abs{z-y}} \phi_1(z)\ddz}_{:=B_{1}} -\underbrace{\int_{\Omega}\frac{\eta_{\epsilon(z)}\eta_{\epsilon}(y)}{\abs{ z''-y}}	\phi_1(z)\ddz}_{:=B_{2,\epsilon}}.
		\end{align*}
		Indeed, we have
		\begin{align*}
			&\abs{B_1}\leq\int_{\Omega}\frac{\phi_1(z)}{\abs{z-y}}\ddz \leq \norm{\phi_1}_\infty C_\Omega,
		\end{align*}
		and 
		\begin{align*}
			\abs{B_{2,\epsilon}}&\leq \eta_\epsilon(y)\int_{\Omega_\epsilon}\frac{\phi_1(z)}{\abs{z''-y}} \ddz \\&\leq \norm{\phi_1}_\infty \int_{B_{M_2}}\frac{1}{\abs{z-y}}\ddz \leq C_\Omega.
		\end{align*}
		Finally, the equation for $\gamma^*(x)$ is 
		\begin{align*}
			0=H_{\gamma^*}(x,x)=\frac{1}{2 d(x,\partial \Omega)}+\frac{4\pi \phi_1(x)^2}{\gamma^*(x)-\lambda_1}+B_\epsilon(x)\phi_1(x)+w_2(x,x),
		\end{align*}
		and by the boundedness of $B_\epsilon(x)$ and $w_2(x,x)$ we obtain \reff{finalasygamma}.
	\end{proof}

	{\bf Acknowledgements:}
	The authors acknowledge the support from the Royal Society Research Professorship RP-R1-180114, United Kingdom. G. Ageno acknowledges support from the ERC/UKRI 2208 Horizon Europe Grant SWAT EP/X030644/1 and M. del Pino from the ERC/UKRI Horizon Europe Grant ASYMEVOL EP/Z000394/1. The authors  express gratitude to the anonymous referee for their careful reading, which contributed to the improvement of the manuscript.
	%\bibliography{PhD_bibtex}
	%\bibliographystyle{apalike}
	\bibliographystyle{abbrv}

\end{document}